\documentclass[twoside]{article}[11pt]

\usepackage{amssymb,amsmath}
\usepackage{amsthm}
\usepackage[colorlinks,linkcolor=violet,citecolor=violet]{hyperref}
\usepackage[nameinlink]{cleveref}
\usepackage{float}
\usepackage{fullpage}
\usepackage[dvipsnames]{xcolor}
\usepackage{fancyhdr,datetime}
\usepackage{graphicx}
\usepackage{subcaption}
\usepackage{tikz,pgf,tikz-3dplot}
\usepackage{tikz-cd}
\usepackage{rotating}
\usepackage{etoolbox}
\usepackage{dynkin-diagrams}
\usepackage{mathrsfs}
\usepackage[percent]{overpic}
\usepackage[export]{adjustbox}
\usepackage{stmaryrd}

\renewcommand\footnoterule{\kern-3pt \hrule width \textwidth \kern 2.6pt}

\let\oldsection\section 
\renewcommand{\section}{
  \renewcommand{\theequation}{\thesection.\arabic{equation}}
  \oldsection}

\newcommand{\mexp}[1]{\ensuremath{\exp(-2\pi \mathrm{i}\,\sprod{ #1 })}}

\newcommand{\N}{\mathbb{N}}

\newcommand\Z{\mathbb{Z}}
\newcommand\C{\mathbb{C}}
\newcommand\Q{\mathbb{Q}}
\newcommand\R{\mathbb{R}}

\newcommand{\Rt}{\ensuremath{\R[x]}}

\newcommand{\RX}{\ensuremath{\Q[z]}}
\newcommand{\QX}{\ensuremath{\Q[z]}}

\newcommand{\Rx}{\ensuremath{\Q[x^{\pm}]}}
\newcommand{\Qx}{\ensuremath{\Q[x^{\pm}]}}

\newcommand{\nops}[1]{\ensuremath{\vert #1  \vert}}

\newcommand{\coeff}{\ensuremath{\mathrm{Coeff}}} 
\newcommand{\rank}{\ensuremath{\mathrm{Rank}}} 
\newcommand{\trace}{\ensuremath{\mathrm{Tr}}} 
\renewcommand{\det}{\ensuremath{\mathrm{Det}}} 
\newcommand{\tbox}[1]{\ensuremath{\quad \mbox{#1} \quad}} 

\newcommand{\gva}{\ensuremath{\mathcal{G}}}

\newcommand{\RootA}[1][n-1]{\ensuremath{\mathrm{A}_{#1}}}
\newcommand{\RootB}[1][n]{\ensuremath{\mathrm{B}_{#1}}}
\newcommand{\RootC}[1][n]{\ensuremath{\mathrm{C}_{#1}}}
\newcommand{\RootD}[1][n]{\ensuremath{\mathrm{D}_{#1}}}
\newcommand{\RootE}[1][n]{\ensuremath{\mathrm{E}_{#1}}}
\newcommand{\RootF}[1][n]{\ensuremath{\mathrm{F}_{#1}}}
\newcommand{\RootG}[1][n]{\ensuremath{\mathrm{G}_{#1}}}

\newcommand{\weyl}{\ensuremath{\mathcal{W}}} 

\newcommand{\roots}{\ensuremath{\rho}}
\newcommand{\Roots}{\ensuremath{\mathrm{R}}}
\newcommand{\Base}{\ensuremath{\mathrm{B}}}

\newcommand{\sprod}[1]{\ensuremath{\langle#1\rangle}} 

\newcommand{\fweight}[1]{\ensuremath{\omega_{#1}}}	
\newcommand{\weight}{\ensuremath{\mu}} 
\newcommand{\Weights}{\Omega} 

\newcommand{\orb}[1]{\ensuremath{{\theta}_{#1}}} 
\newcommand{\bigorb}[1]{\ensuremath{{\Theta}_{#1}}} 

\newcommand{\TT}{\ensuremath{\conj{T}}} 
\newcommand{\conj}[1]{\ensuremath{\widehat{#1}}}

\newcommand{\TTorus}{\mathbb{T}}		
\newcommand{\TToruss}{\mathbb{T}_1}	

\newcommand{\bigcos}{\ensuremath{\vartheta}} 
\newcommand{\Image}{\ensuremath{\mathcal{T}}} 

\def\row#1/#2!{#1_{\IfStrEq{#2}{}{n-1}{#2}} & \dynkin{#1}{#2}\\}

\newtheorem{lemma}{Lemma}[section]
\newtheorem{example}[lemma]{Example}

\newtheorem{definition}[lemma]{Definition}
\newtheorem{proposition}[lemma]{Proposition}
\newtheorem{remark}[lemma]{Remark}
\newtheorem{theorem}[lemma]{Theorem}

\newtheorem{corollary}[lemma]{Corollary}

{\scshape}{\slshape}

\newcommand{\italgf}{\slshape  }

\title{	Orbit spaces of Weyl groups acting on compact tori:\\
		a unified and explicit polynomial description}

\author{	
	Evelyne Hubert\thanks{Inria d'Universit\'{e} C\^{o}te d'Azur} , 
	Tobias Metzlaff\footnotemark[1] \thanks{RPTU Kaiserslautern--Landau} , 
	Cordian Riener\thanks{UiT The Arctic University}
}

\newcommand{\Addresses}{{
		\bigskip
		\footnotesize
		
		E.~Hubert, \textsc{Centre Inria d'Universit\'{e} C\^{o}te d'Azur, 06902 Sophia Antipolis, France}\par\nopagebreak
		\textit{E-mail address}: \texttt{evelyne.hubert@inria.fr}\par\nopagebreak
		\textit{ORCID}: \texttt{0000-0003-1456-9524}
		
		\medskip
		
		T.~Metzlaff, \textsc{Department of Mathematics, RPTU Kaiserslautern--Landau, 67663 Kaiserslautern, Germany}\par\nopagebreak
		\textit{Former Affiliation}: \textsc{Centre Inria d'Universit\'{e} C\^{o}te d'Azur, 06902 Sophia Antipolis, France}\par\nopagebreak
		\textit{E-mail address}: \texttt{tobias.metzlaff@rptu.de}\par\nopagebreak
		\textit{ORCID}: \texttt{0000-0002-0688-7074}
		
		\medskip
		
		C.~Riener, \textsc{Department of Mathematics, UiT The Arctic University, 9037 Troms\o, Norway}\par\nopagebreak
		\textit{E-mail address}: \texttt{cordian.riener@uit.no}\par\nopagebreak
		\textit{ORCID}: \texttt{0000-0002-1192-3500}
		
}}

\begin{document}

\maketitle\thispagestyle{empty}

\begin{abstract}
The Weyl group of a crystallographic root system has a nonlinear action on the compact torus. 
The orbit space of this action is a compact basic semi--algebraic set. 
We present a polynomial description of this set for the Weyl groups of type $\mathrm{A}$, $\mathrm{B}$, $\mathrm{C}$, $\mathrm{D}$ and $\mathrm{G}$. 
Our description is given through a polynomial matrix inequality. 
The novelty lies in an approach via Hermite quadratic forms and a closed formula for the matrix entries. 

The orbit space of the nonlinear Weyl group action is the orthogonality region of generalized Chebyshev polynomials. 
In this polynomial basis, we show that the matrices obtained for the five types follow the same, surprisingly simple pattern. 
This is applied to the optimization of trigonometric polynomials with crystallographic symmetries. 
~\\
~\\
Partial results of this article were presented as a poster at the ISSAC 2022 conference \cite{chromaticissac22} and in the doctoral thesis \cite{TobiasThesis}. 
~
\clearpage
\end{abstract}


\tableofcontents


\Addresses

\clearpage

\section{Introduction}
\label{section_introduction}
\setcounter{equation}{0}


For a finite group acting on an affine variety, the set of all orbits is called orbit space. 
The coordinate ring contains the finitely generated subring of invariants. 
The image of the variety under fundamental invariants admits a parametrization of the orbit space. 
Being able to describe this image allows for numerous applications in differential geometry \cite{dubrovin1993}, dynamical systems \cite{gatermann2000}, polynomial optimization \cite{riener2013} and quantum systems \cite{abud83,gufan01,sartori05}. 

When the group acts linearly as, for instance, a matrix group on a real vector space, explicit descriptions of the orbit space are available for some groups. 
Specific examples are the symmetric group $\mathfrak{S}_n$ \cite{procesi78}, the Weyl groups $\mathfrak{S}_n\ltimes \{\pm 1\}^n$ and $\mathfrak{S}_n\ltimes \{\pm 1\}^{n}_+$ \cite{talamini20}, some exceptional groups \cite{talamini10} and low--dimensional groups \cite{sartori96}.
Recent advances in the area of Fourier analysis and optimization have lead to the study of Weyl groups acting nonlinearly on the compact torus. 
For all the infinite Weyl group families and one exceptional case, this article offers a unified and explicit description as a compact basic semi--algebraic set for the orbit space of this action. 


In discrete Fourier analysis, the orbit space of the Weyl group acting on the compact torus is the region of orthogonality for several families of generalized Chebyshev polynomials \cite{HoffmanWithers}. 
These polynomials have been widely studied since the works of Koornwinder \cite{KOORNWINDER1974}, MacDonald \cite{MacDonald1990} and beyond \cite{DunnLidl1980,EierLidl1982,beerends91}. 
In some cases, one can show that the roots of these polynomials lie in the orbit space and provide nodes for Gaussian cubature formulae due to Xu et al. \cite{Xu08,Xu09,Moody2011,Xu12,Xu15,hakova16}. 
The roots are also suitable sampling points \cite{MuntheKaas2012} and working in the Chebyshev basis brings advantages in interpolation \cite{HubertSinger2020}. 

Moreover, being able to describe the orbit space of a nonlinear Weyl group action allows for symmetry reductions in trigonometric optimization. 
The authors have recently applied the present result in \cite{chromatic22}, where an invariant trigonometric polynomial is rewritten as a linear combination of Chebyshev polynomials on the orbit space and then minimized as a classical polynomial on a semi--algebraic set. 
Such optimization problems arise, for example, in the computation of chromatic numbers \cite{BdCOV}. 

These applications rely on the invariant--theoretic results of Bourbaki \cite{bourbaki456}, Steinberg \cite{steinberg75} and Farkas \cite{farkas84}, which are presented in Lorenz' book \cite{lorenz06}. 


We briefly explain the setup (see also \Cref{preliminaries1} for the details). 
A Weyl group can be seen as a finite Euclidean reflection group that leaves a full--dimensional lattice invariant \cite{kane13}.
Such a group has a representation over the integers $\gva\subseteq \mathrm{GL}_n(\Z)$ by a change of basis stemming from the lattice generators. 
Thus, we have a nonlinear group action on the algebraic torus
\begin{equation}\label{eq_IntroNonlinearAction}
	\gva \times (\C^*)^n\to (\C^*)^n,\quad (B,x) \mapsto (x^{B^{-1}_{\cdot 1}},\ldots,x^{B^{-1}_{\cdot n}}) .
\end{equation}
The action leaves the compact torus $\TTorus^n\subseteq(\C^*)^n$ invariant. 
The object of interest in the above applications is the orbit space $\TTorus^n/\gva$. 
This set can be embedded in $\R^n$ as a compact basic semi--algebraic one, 
see for example \Cref{fig_OrbitSpaceIntro}, 
and is parameterized as follows: 
The nonlinear action of $\gva$ on $(\C^*)^n$ induces an action on the ring of Laurent polynomials $\Rx$ in $n$ variables. 
The invariant subring $\Rx^\gva$ is generated by algebraically independent fundamental invariants $\orb{1},\ldots,\orb{n}$ \cite[Chapitre VI]{bourbaki456}. 
A canonical choice for the $\orb{i}$ is to take orbit sums, where one may assume without loss of generality that the $\orb{i}$ only take on real values on $\TTorus^n$. 
Then the image of
\begin{equation}\label{eq_CanonicalProjection}
	\TTorus^n \to \R^n , \quad x\mapsto (\orb{1}(x),\ldots,\orb{n}(x)),
\end{equation}
is a compact subset of $\R^n$ and every image point corresponds in a $1:1$ manner to an orbit in $\TTorus^n/\gva$. 
We call this image the real $\TTorus$--orbit space of $\gva$. 
The goal is to find polynomial inequalities that describe it as a basic semi--algebraic set.


This article gives such a polynomial description of the $\TTorus$--orbit space for all four infinite Weyl group families $\RootA$, $\RootB$, $\RootC$ and $\RootD$ that leave the associated weight lattice invariant \cite[Planches I -- IV]{bourbaki456}. 
Furthermore, there are five exceptional cases $\RootE[6,7,8]$, $\RootF[4]$, $\RootG[2]$ and our result applies precisely to $\RootG[2]$ \cite[Planche IX]{bourbaki456}.
In those cases, the Weyl group is $\gva\cong\mathfrak{S}_n \ltimes \{\pm 1\}^k$, where $\mathfrak{S}_n$ is the symmetric group, $\ltimes$ denotes the semi--direct product and $0\leq k\leq n$ is an integer. 

We construct matrices $H$ with entries in $\QX$ that satisfy the following property: 
For $z\in\R^n$, the equation $z = (\orb{1}(x),\ldots,\orb{n}(x))$ has a solution $x\in\TTorus^n$ if and only if $H(z)$ is positive semi--definite. 
The matrices $H$ for all five cases are given through formulae that only depend on the dimension $n$. 



While closed expressions for the $\orb{i}$ are known for certain groups and lattices, see \cite{lorenz06,hamm2014}, the concrete geometry of the $\TTorus$--orbit space is mostly unexplored: 
To obtain a characterization, one could follow the general approach of Procesi and Schwarz for compact Lie groups acting on varieties \cite[\S 4]{procesischwarz85}. 
An orbit space is described there by the Gram matrix of the differentials that generate the $\gva$--equivariant maps. 
This matrix is symmetric and subsequently has up to $n\,(n+1)/2$ distinct entries. 
To be of use, the entries need to be rewritten in the fundamental invariants, that is, in the coordinates of the orbit space. 
The effort to conduct this rewriting step grows with $n$ and creates a bottleneck in computations. 
The matrices $H$ we obtain in this article have the remarkable advantage to come directly in terms of the generating invariants without additional computations necessary. 
As they have Hankel structure, they only have up to $2\,n-1$ distinct entries. 
Furthermore, we obtain sharp bounds for the degree and the number of basis elements per entry in the Chebyshev basis. 


In \Cref{HermiteCharacterizationAn,HermiteCharacterizationCn,HermiteCharacterizationBn,HermiteCharacterizationDn,HermiteCharacterizationGn}, we give the entries for the matrices $H$ in the coordinates $z$ of the $\TTorus$--orbit space, that is, in the standard monomial basis $\{z^\alpha\,\vert\,\alpha\in\N^n\}$ of $\RX$. The definition of the generalized Chebyshev polynomials, which we denote by $\TT_\alpha$, allows us to unify the formula for the five families. 
In the orthogonal basis $\{\TT_\alpha\,\vert\,\alpha\in\N^n\}$ of $\RX$, we show that the matrices $H$ are always identical and follow a remarkable pattern. 

\begin{theorem}\label{MainThmIntro}
Let $\gva$ be a Weyl group of type $\RootA$, $\RootB$, $\RootC$, $\RootD$ or $\RootG[2]$. 
A real point $z$ is contained in the real $\TTorus$--orbit space of $\gva$ if and only if $H(z)\succeq 0$, 
where $H\in\RX^{n\times n}$ is the matrix with Hankel structure\footnote{The entry $(i,j)$ only depends on $i+j$.\\\phantom{iiiiiii}If $\Roots$ is of type $\RootA$ and $n\geq 3$, then $z\in\R^{n-1}$ and $H$ is $n\times n$, because $\Roots$ has rank $n-1$ while $\gva\cong\mathfrak{S}_n$ (similar for $\RootG[2]$).}
	\[
	H =
	\begin{bmatrix}
		\frac{\TT_{0}-\TT_{2\,e_{1}}}{4}& 
		\frac{\TT_{e_{1}} -\TT_{3\,e_{1}}}{8}& 
		\frac{\TT_{0}- \TT_{4\,e_{1}}}{16}&
		\frac{2\TT_{e_{1}}- \TT_{3\,e_{1}} - \TT_{5\,e_{1}}}{32}&
		\cdots\\
		
		\frac{\TT_{e_{1}} -\TT_{3\,e_{1}}}{8}& 
		\frac{\TT_{0}- \TT_{4\,e_{1}}}{16}&
		\frac{2 \TT_{e_{1}}- \TT_{3\,e_{1}} - \TT_{5\,e_{1}}}{32}&
		\frac{2 \TT_{0} +  \TT_{2\,e_{1}}-2 \TT_{4\,e_{1}} -  \TT_{6\,e_{1}}}{64}&
		\cdots\\
		
		\frac{\TT_{0}- \TT_{4\,e_{1}}}{16}& 
		\frac{2\TT_{e_{1}}- \TT_{3\,e_{1}} - \TT_{5\,e_{1}}}{32}& 
		\frac{2 \TT_{0} +  \TT_{2\,e_{1}}-2 \TT_{4\,e_{1}} -  \TT_{6\,e_{1}}}{64}&
		\frac{5 \TT_{e_1} - \TT_{3\,e_1} - 3 \TT_{5\,e_1} - \TT_{7\,e_1}}{128}&
		\cdots\\
		
		\frac{2\TT_{e_{1}}- \TT_{3\,e_{1}} - \TT_{5\,e_{1}}}{32}&
		\frac{2 \TT_{0} +  \TT_{2\,e_{1}}-2 \TT_{4\,e_{1}} -  \TT_{6\,e_{1}}}{64}&
		\frac{5 \TT_{e_1} - \TT_{3\,e_1} - 3 \TT_{5\,e_1} - \TT_{7\,e_1}}{128}&
		\frac{5 \TT_{0} + 4 \TT_{2\,e_1} - 4 \TT_{4\,e_1} - 4 \TT_{6\,e_1} - \TT_{8\,e_1}}{256}&
		\cdots\\
		
		\vdots & \vdots & \vdots & \vdots & \ddots
	\end{bmatrix}.
	\]
$($We give the proof and the closed formula for the entries in \emph{\Cref{MainThm}}.$)$
\end{theorem}

\begin{figure}
\begin{center}
	\begin{subfigure}{.2\textwidth}
		\centering
		\includegraphics[width=\textwidth]{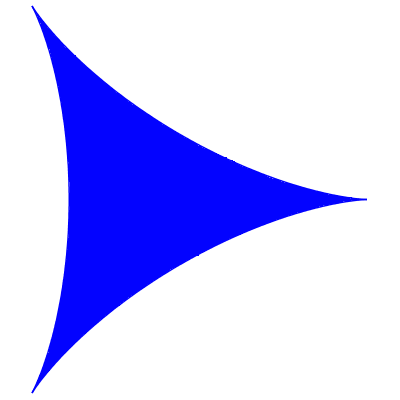}
		\caption{$\RootA[2]$}
	\end{subfigure}\quad
	\begin{subfigure}{.2\textwidth}
		\centering
		\includegraphics[width=\textwidth]{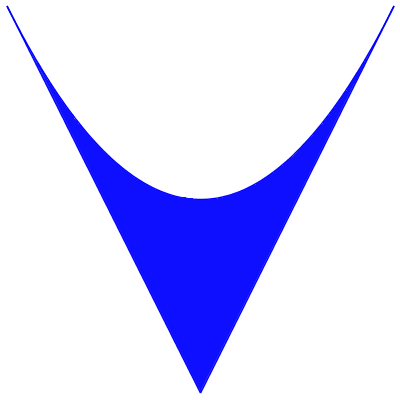}
		\caption{$\RootC[2]$}
	\end{subfigure}\quad
	\begin{subfigure}{.2\textwidth}
		\centering
		\includegraphics[width=\textwidth]{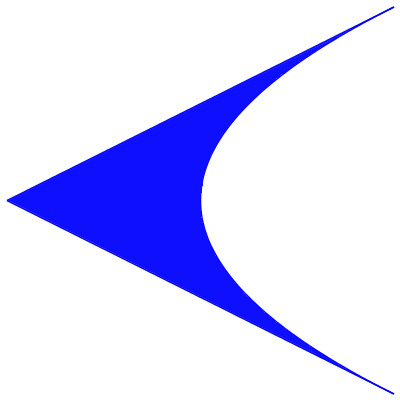}
		\caption{$\RootB[2]$}
	\end{subfigure}\quad
	\begin{subfigure}{.2\textwidth}
		\centering
		\includegraphics[width=\textwidth]{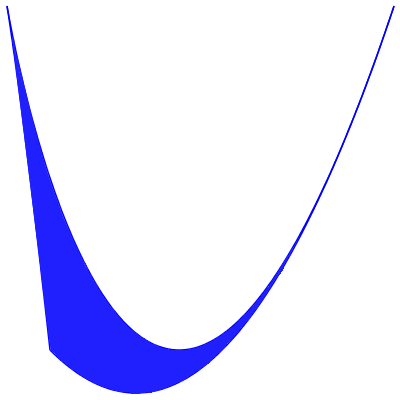}
		\caption{$\RootG[2]$}
	\end{subfigure}\quad
	\begin{subfigure}{.3\textwidth}
		\centering
		\includegraphics[width=\textwidth]{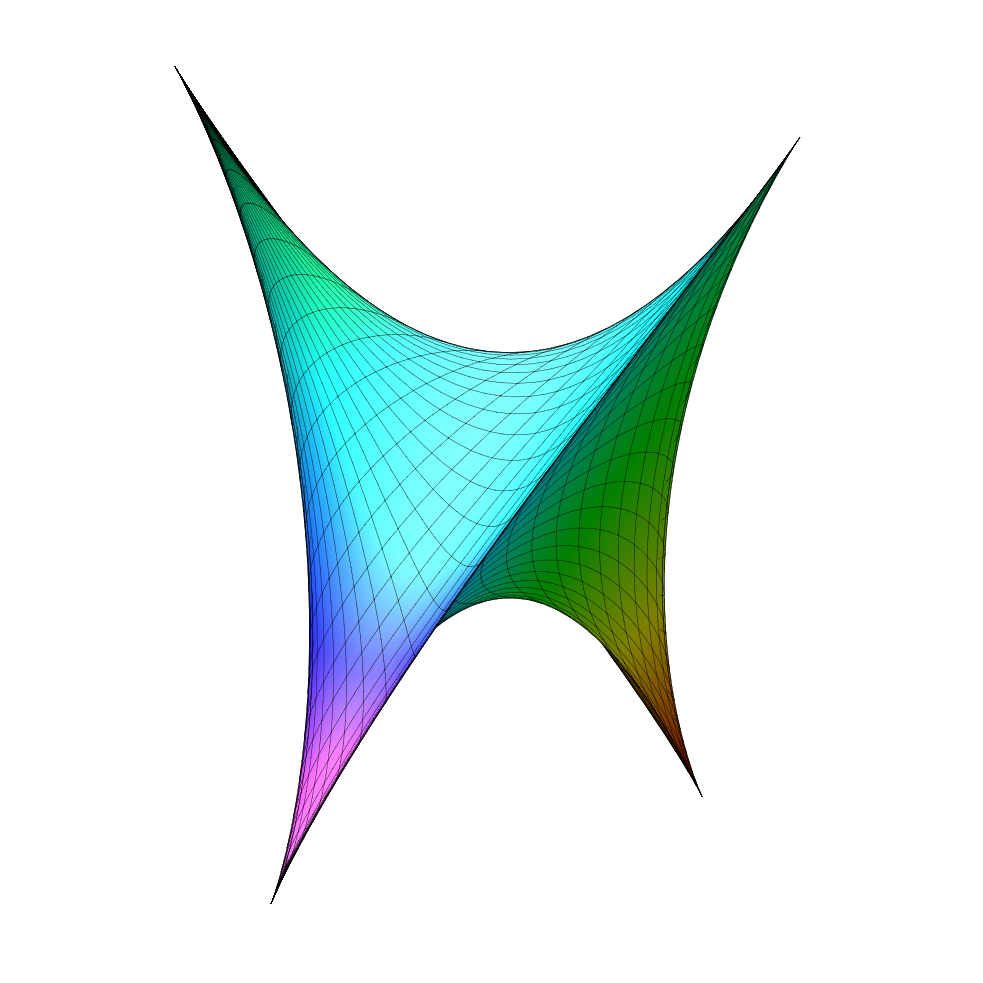}
		\caption{$\RootA[3]$}
	\end{subfigure}\quad
	\begin{subfigure}{.3\textwidth}
		\centering
		\includegraphics[width=\textwidth]{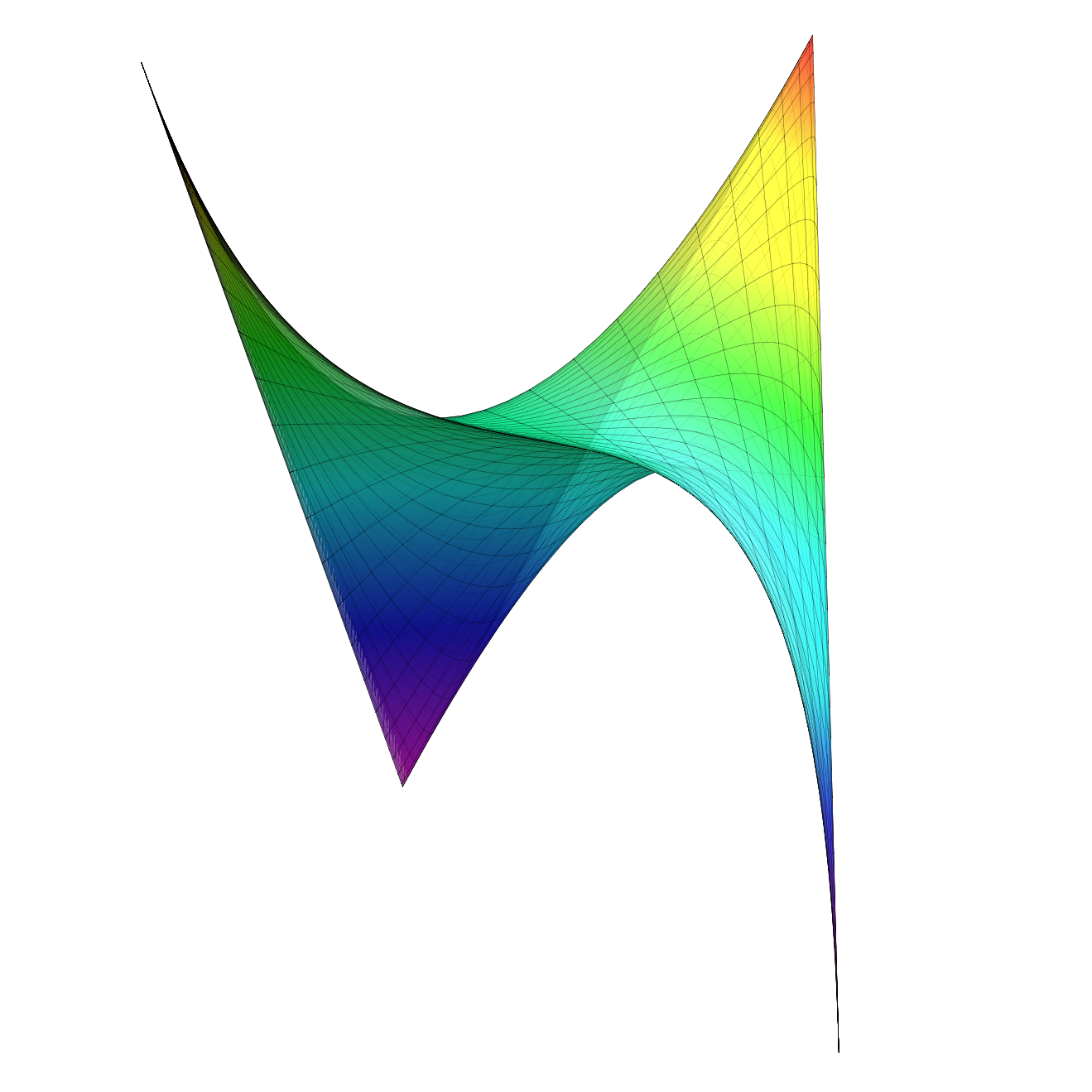}
		\caption{$\RootC[3]$}
	\end{subfigure}\quad
	\begin{subfigure}{.3\textwidth}
		\centering
		\includegraphics[width=\textwidth]{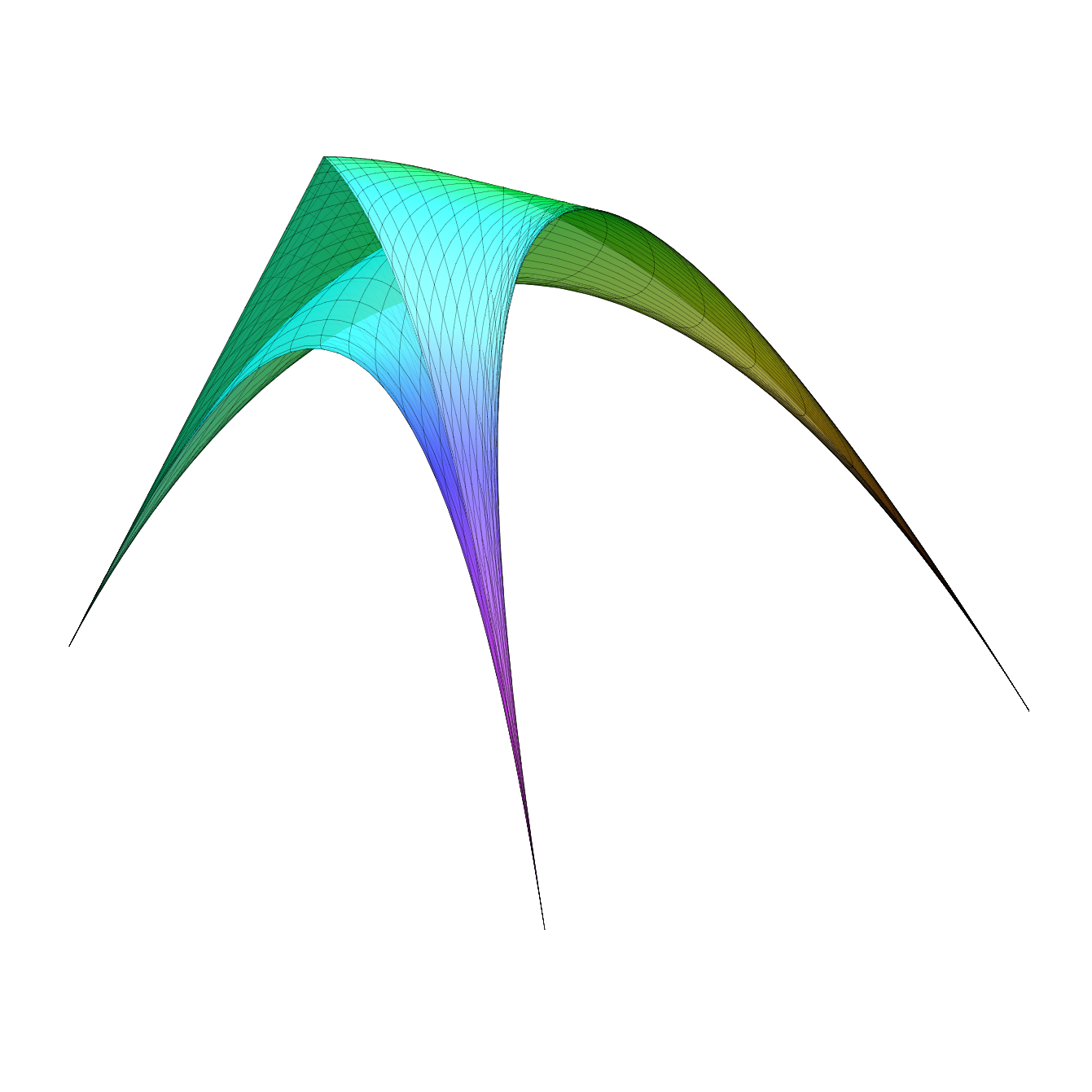}
		\caption{$\RootB[3]$}
	\end{subfigure}
	\caption{
		The real $\TTorus$--orbit space for the irreducible root systems of rank $2$ and $3$.
	}\label{fig_OrbitSpaceIntro}
\end{center}
\end{figure}

The simplest instance is the univariate case $n=1$. 
Then $\gva = \{ \pm I_1\} \subseteq \mathrm{GL}_1(\Z)$ has $\TTorus$--orbit space $[-1,1]\subseteq \R$. 
Indeed, any element of $\Q[x,x^{-1}]^\gva$ can be written uniquely as a polynomial in $(x+x^{-1})/2$ and the image of $\TTorus$ under $x\mapsto(x+x^{-1})/2$ is $[-1,1]$.
To obtain a polynomial description, we consider the matrix
\[
H(z)
=	\begin{bmatrix}
\frac{\TT_{0}(z) - \TT_{2}(z)}{4}
\end{bmatrix}
=	\begin{bmatrix}
\frac{1 - (2\,z^2 - 1)}{4}
\end{bmatrix}
=	\begin{bmatrix}
\frac{1 - z^2}{2}
\end{bmatrix}
\in	\RX^{1\times 1},
\]
where $\TT_\ell$ is the usual univariate Chebyshev polynomial of the first kind with $\TT_\ell \left((x+x^{-1})/2\right) = (x^\ell+x^{-\ell})/2$ for $\ell\in\N$. 
Hence, we have $H(z)\succeq 0$ if and only if $1\geq z^2$, that is, if $z \in [-1,1]$. This is the domain on which the $\TT_\ell$ are orthogonal with respect to the weight function $(2\,\det(H(z)))^{-1/2} = 1/\sqrt{1-z^2}$. 
For $n\geq 2$, some popular examples of $\TTorus$--orbit spaces are depicted in \Cref{fig_OrbitSpaceIntro}. 

We explain the idea of the proof for $\RootC$, the other cases are very similar: 
The Weyl group of $\RootC$ is the semi--direct product $\gva\cong \mathfrak{S}_n \ltimes \{\pm 1\}^n$. 
In particular, the ring of invariants can be written as $\Qx^\gva \cong \Q[(y+y^{-1})/2]^{\mathfrak{S}_n}$. 
Here, the entries of the vectors $y$ and $y^{-1}$ (the inverse applies entry--wise) are the $2\,n$ monomials appearing in the first fundamental invariant $\orb{1}$. 
The group $\mathfrak{S}_n$ acts by permuting the $(y_i+y_i^{-1})/2$. 
Hence, the problem reduces to rewriting the $\mathfrak{S}_n$--invariants, that is, the elementary symmetric polynomials $\sigma_i((y+y^{-1})/2)$, in terms of the $\gva$--invariants $\orb{i}(x)$. 
The case $\RootC$ is the easiest, because we immediately have
\begin{equation}\label{eq_ESPInvariant}
	\sigma_i\left(\frac{y(x)+y(x)^{-1}}{2}\right) = c_i\,\orb{i}(x),
\end{equation}
where $c_i \in \Q$ is a constant. 
Hence, when $z\in\C^n$ with $z_i=\orb{i}(x)$ for some $x\in(\C^*)^n$, then the $c_i\,z_i$ are the coefficients of a polynomial with root vector $(y+y^{-1})/2$. 
We show that $z$ is contained in the $\TTorus$--orbit space if and only if $y(x)\in\TTorus^n$, that is, if and only if all the roots are contained in the interval $[-1,1]$. 
A necessary and sufficient criterion for the latter is that the associated Hermite quadratic form is positive semi--definite. 
We compute the matrix $H$ of this form and express its entries as polynomials in $z_i$. 

The contents of this article are as follows. 
In \Cref{preliminaries1}, we introduce the required definitions and notations. 
This section contains mostly standard material from the theory of root systems and lattices \cite{bourbaki456}. 
The following \Cref{section:symmetricpolynomialsandroots} is our toolbox for characterizing when all solutions of a symmetric system such as \Cref{eq_ESPInvariant} are contained in the compact torus. 
These tools are based on the material in \cite{cox05,cox13} and then used in
\Cref{section_proofA,section_proofC,section_proofB,section_proofD,section_proofG}, 
to give the formula from \Cref{MainThmIntro} in the basis of standard monomials. 
Finally, we prove \Cref{MainThmIntro} in \Cref{section_proofMain} by using the definition of the generalized Chebyshev polynomials.

The description of $\TTorus^n/\gva$ we offer is by no means unique as it depends on the choice of invariants $\orb{i}$. 
However, since the ring of invariants is a polynomials algebra, there exist unique polynomial mappings between different invariants, say $\orb{i}=\pi_i(\tilde{\theta}_1,\ldots,\tilde{\theta}_n)$. 
Hence, if $H(z)\succeq 0$ describes the $\TTorus$--orbit space defined by the $\orb{i}$, then $H(\pi(z))\succeq 0$ describes the one defined by the $\tilde{\theta}_i$. 
Alternatively, one can express the left--hand side of \Cref{eq_ESPInvariant} in terms of the $\tilde{\theta}_i$ to get the matrix immediately. 

To easily access the matrices, we offer a Maple package \textsc{GeneralizedChebyshev}\footnote{The package is available under \href{https://github.com/TobiasMetzlaff/GeneralizedChebyshev}{https://github.com/TobiasMetzlaff/GeneralizedChebyshev}.}: 
The output of
\begin{itemize}
	\item \textcolor{blue}{RHermiteMatrix$(\Roots,n)$} is the matrix $H$ in the standard monomials $z^\alpha$,
	\item \textcolor{blue}{RTHermiteMatrix$(\Roots,n)$} is the matrix $H$ in the Chebyshev polynomials $\TT_\alpha$,
\end{itemize}
where $\Roots$ is the type $\mathrm{A},\mathrm{B},\mathrm{C},\mathrm{D}$ or $\mathrm{G}$ and $n\in\N$ is the dimension. 
All figures in this article were produced with this package and the plotting tools of Maple.

\section{The $\TTorus$--orbit space of a Weyl group}
\label{preliminaries1}
\setcounter{equation}{0}

Based on \cite{bourbaki456}, we revisit the definition of a crystallographic root system and how it admits a Weyl group with an invariant, full--dimensional lattice. 
From \cite{lorenz06}, we then recall actions of Weyl groups on tori and define the orbit space. 
A key theorem, proven in \cite{bourbaki456}, is that the ring of invariants for such actions is a polynomial algebra. 
This brings the definition of generalized Chebyshev polynomials. 
Finally, we prove that the orbit space is always as a real compact set. 

\subsection{Root systems and Weyl groups}

We take the definitions from \cite[Ch. VI]{bourbaki456}. For a finite--dimensional real vector space $V$ with inner product $\sprod{\cdot , \cdot }$, a subset $\Roots\subseteq V$ is called a \textbf{root system} in $V$, if the following conditions hold.
\begin{enumerate}
	\item[R1] $\Roots$ is finite, spans $V$ and does not contain $0$.
	\item[R2] If $\rho, \tilde{\rho} \in \Roots$, then $\langle\tilde{\rho},\rho^\vee\rangle \in \Z$, where $\rho^\vee:=\frac{2\,\rho}{\langle\rho,\rho\rangle}$.
	\item[R3] If $\rho, \tilde{\rho} \in \Roots$, then $s_\rho(\tilde{\rho}) \in \Roots$, where $s_\rho$ is the reflection $ s_\rho(u) := u - \langle u,\rho^\vee\rangle \rho$ for $u \in  V$.
\end{enumerate}
The elements of $\Roots$ are called \textbf{roots} and the \textbf{rank} of $\Roots$ is $\rank(\Roots):=\dim(V)$. The \textbf{Weyl group} $\weyl$ of $\Roots$ is the group generated by the reflections $s_\roots$ for $\roots \in \Roots$. This is a finite subgroup of the orthogonal group on $V$ with respect to the inner product $\sprod{\cdot , \cdot }$. Furthermore, $\Roots$ is said to be \textbf{reduced}, if the following holds.
\begin{enumerate}
	\item[R4] For $\rho \in \Roots$ and $c \in \R$, we have $c\rho \in \Roots$ if and only if $c = \pm 1$.$\phantom{\frac{1}{2}}$
\end{enumerate}
We assume that the ``reduced'' property R4 always holds when we speak of a ``root system''. 
An example is given in \Cref{example_multiplicativeaction}. 


A \textbf{weight} of $\Roots$ is an element $\weight\in V$, such that, for all $\roots\in\Roots$, we have $\sprod{\weight,\roots^\vee}  \in \Z$. 
By the ``crystallographic'' property R2, every root is a weight. 
The set of all weights is called the \textbf{weight lattice} $\Weights$.
This lattice is full--dimensional and invariant under the action of $\weyl$ on $V$, that is, $\Weights \otimes_\Z \R = V$ and $\weyl \, \Weights = \Weights$.

Let $\dim(V)=n$ and $\Base=\{\rho_1,\ldots,\rho_n\}\subseteq \Roots$ be a set of roots with the following properties.
\begin{enumerate}
\item[B1] $\Base$ is an $\R$--vector space basis for $V$.
\item[B2] Every root $\rho \in \Roots$ can be written as $\rho = \alpha_1\,\rho_1 + \ldots +\alpha_n \, \rho_n$ or $\rho = -\alpha_1\,\rho_1 - \ldots -\alpha_n \,\rho_n$ for some $\alpha \in \N^n$.
\end{enumerate}
Such a set $\Base$ always exists \cite[Ch. VI, \S 1, Thm. 3]{bourbaki456} and we define the \textbf{fundamental weights} as the linearly independent $\fweight{1}, \ldots , \fweight{n} \in \Weights$, such that, for $1\leq i,j \leq n$, we have $\sprod{ \fweight{i}, \roots_j^\vee} = \delta_{i,j}$. 

\begin{proposition}\label{remark_PermutationOrb2}
Denote by $\mathfrak{S}_n$ the symmetric group on $n$ points. There exists a permutation $\sigma\in\mathfrak{S}_n$ of order at most $2$, such that, for all $1 \leq i \leq n$, we have $-\fweight{i} \in \weyl \fweight{\sigma(i)}$.
\end{proposition}
\begin{proof}
The statement follows from the existence of the longest Weyl group element $s_0\in\weyl$, which takes $\{\rho_1,\ldots,\rho_n\}$ to $\{-\rho_1,\ldots,-\rho_n\}$, see \cite[Ch. VI, \S 1, Coro. 3 of Prop. 17]{bourbaki456}. 
Hence, there is a permutation $\sigma\in\mathfrak{S}_n$ with $s_0(\rho_i)=-\rho_{\sigma(i)}$. 
Since $s_0^2=\mathrm{Id}_V$ and the inner product is $\weyl$--invariant, $\sigma$ is of order at most $2$ and
\[
	-s_0(\fweight{i})
=	\sum\limits_{j=1}^n \sprod{-s_0(\fweight{i}),\roots_{j}^\vee} \, \fweight{j}
=	\sum\limits_{j=1}^n \sprod{\fweight{i},-s_0(\roots_{j}^\vee)} \, \fweight{j}
=	\sum\limits_{j=1}^n \sprod{\fweight{i},\roots_{\sigma(j)}^\vee} \, \fweight{j}
=	\fweight{\sigma(i)}
\]
is also a fundamental weight.
\end{proof}


Assume that $V=V^{(1)}\oplus\ldots\oplus V^{(k)}$ is the direct sum of proper orthogonal subspaces and that, for each $1\leq i\leq k$, $\Roots^{(i)}$ is a root system in $V^{(i)}$. 
Then $\Roots:=\Roots^{(1)}\cup\ldots\cup\Roots^{(k)}$ is a root system in $V$ called the \textbf{direct sum} of the $\Roots^{(i)}$. 
If a root system is not the direct sum of at least two root systems, then it is called \textbf{irreducible} \cite[Ch. VI, \S 1.2]{bourbaki456}.

Every root system can be uniquely decomposed into irreducible components \cite[Ch. VI, \S 1, Prop. 6, 7]{bourbaki456}. 
There exist seven families of irreducible root systems, which are denoted by $\RootA$, $\RootB$, $\RootC$, $\RootD$, $\RootE[6,7,8]$, $\RootF[4]$ and $\RootG[2]$, see \cite[Ch. VI, \S 4, Thm. 3]{bourbaki456} and \cite[Planches I -- IX]{bourbaki456}. 


\subsection{Multiplicative actions}
\label{subsec_NonlinearActions}


To introduce the orbit space of a Weyl group, we require some multiplicative invariant theory from \cite{lorenz06}. 
Let $\weyl$ be a finite subgroup of $\mathrm{GL}_n(\R)$ and $\Weights\subseteq\R^n$ be an $n$--dimensional lattice, spanned by $\fweight{1},\ldots,\fweight{n}$. 
We assume that $\Weights$ is $\weyl$--invariant. 
Then the $\Z$--module isomorphism $\varphi: \Z^n \to \Weights,\,e_i\mapsto \fweight{i}$ yields a change of basis, so that
\begin{equation}\label{eq_integerrepresentation}
	\gva
:=	\{	\varphi^{-1} \, s \, \varphi\,\vert\, s\in\weyl\}
\end{equation}
is a finite subgroup of $\mathrm{GL}_n(\Z)$. 
We call $\gva$ the \textbf{integer representation} of $\weyl$ with respect to $\fweight{1},\ldots,\fweight{n}$. 
The group $\gva$ acts on the standard basis vectors $e_i\in\Z^n$ as $\weyl$ acts on the lattice generators $\fweight{i}\in\Weights$. 
Clearly, this group does not only depend on $\weyl$ itself, but also on the lattice $\Weights$. 

Denote by $(\C^*)^n := (\C\setminus\{0\})^n$ the algebraic $n$--torus. 
For $x=(x_1,\ldots,x_n)\in(\C^*)^n$ and a column vector $\alpha=(\alpha_1,\ldots,\alpha_n) \in \Z^n$, we define $x^\alpha:=x_1^{\alpha_1}\ldots x_n^{\alpha_n}\in\C^*$. 
Then we have a nonlinear action
\begin{equation}\label{equation_NonlinearAction}
	\star : 	\begin{array}[t]{ccl}
	\gva \times (\C^*)^n & \to 		&	(\C^*)^n, \\
	\left(B, \,x \right) & \mapsto &	B\star x := x^{B^{-1}}
    									:= (x^{B^{-1}_{\cdot 1}},\ldots,x^{B^{-1}_{\cdot n}}),
   \end{array}
\end{equation}
where $B_{\cdot i}^{-1}\in\Z^n$ denotes the $i$--th column vector of $B^{-1}\in\gva$. 
The orbit of an element $x\in(\C^*)^n$ is denoted by $\gva\star x$. 

The coordinate ring of $(\C^*)^n$ with rational coefficients is the ring of multivariate Laurent polynomials $\Rx:=\Q[x_1,x_1^{-1},\ldots, x_n,x_n^{-1}]$. The monomials of $\Rx$ are the $x^\alpha = x_1^{\alpha_1} \ldots x_n^{\alpha_n}$ with $\alpha\in\Z^n$ and $\star$ induces a linear action on $\Rx$, given by its values on the monomial basis
\begin{equation}\label{equation_LinearAction}
\cdot : 	\begin{array}[t]{ccl}
   				\gva \times \Rx & \to &  \Rx, \\
  			 	\left(B, \,x^\alpha\right) & \mapsto & B\cdot x^\alpha :=x^{B \alpha}.
   			\end{array}
\end{equation}
For $f=\sum_\alpha f_\alpha\,x^\alpha\in\Rx$, $B\in\gva$ and $x\in (\C^*)^n$, we have $(B\cdot f)(x) = \sum_\alpha f_\alpha\,x^{B \alpha} = f(B^{-1} \star x)$. 
If $B \cdot f = f$ whenever $B\in\gva$, then $f$ is called \textbf{$\gva$--invariant}. 
The set of all $\gva$--invariant Laurent polynomials is a subring denoted by $\Rx^\gva$. 
The \textbf{orbit polynomial} associated to $\alpha\in\Z^n$ is
\begin{equation}\label{defi_orbitpoly} 
	\bigorb{\alpha}
:=	\dfrac{1}{\vert\gva\vert}\sum\limits_{B\in\gva} x^{B \alpha}\in\Rx^\gva.
\end{equation}
If $\tilde{\alpha}\in\gva\alpha$, then $\bigorb{\tilde{\alpha}}=\bigorb{\alpha}$. The distinct $\bigorb{\alpha}$ form a basis for $\Rx^\gva$ as a $\Q$--vector space \cite[Eq. (3.4)]{lorenz06}. We have $\bigorb{0}=1$ and, for $\alpha,\beta\in\Z^n$,
\begin{equation}\label{proposition_RecurrenceOrbitPolynomials}
	\bigorb{\alpha}\,\bigorb{\beta}
=	\frac{1}{\nops{\gva}} \sum\limits_{B\in\gva}\bigorb{\alpha+B\beta}.
\end{equation}

As a $\Q$--algebra, $\Rx^\gva$ is finitely generated \cite[Coro. 3.3.2]{lorenz06}. 
Hence, we may assume that $\Rx^\gva=\Q[ \orb{1},\ldots,\orb{m}]$ for some minimal $m\in\N$ and $\orb{1},\ldots,\orb{m}\in\Rx^\gva$. 
The generators $\orb{i}$ are called the \textbf{fundamental invariants} of $\gva$ and we define the map
\begin{equation}
\bigcos: 
\begin{array}[t]{ccl}
	(\C^*)^n	&	\to		&	\C^m,\\
	x			&	\mapsto	&	(\orb{1}(x),\ldots,\orb{m}(x)).
\end{array}
\end{equation}
This map is surjective. 
The question is what happens when we restrict to a maximal compact subgroup: 
We set $\TTorus:=\{x \in \C \,\vert\, x\,\overline{x} =1 \}\subseteq \C^*$, where $\overline{x}$ is the complex conjugate. 
For $x\in\TTorus$, we have $x^{-1} = \overline{x}$ and the compact $n$--torus $\TTorus^n$ is $\gva$--invariant, that is, $\gva\star\TTorus^n = \TTorus^n$. 

Analogously to \cite[Ch. 7, \S 4, Thm. 10]{cox13}, one can show the following.

\begin{proposition}\label{theorem_OrbitSperation}
The map
\[
\begin{array}{ccl}
\TTorus^{n}/\gva	&	\to		&	\C^m,\\
\gva \star x		&	\mapsto	&	\bigcos(x) = (\orb{1}(x),\ldots,\orb{m}(x))
\end{array}
\]
is well--defined and injective.
\end{proposition}


\subsection{$\TTorus$--orbit space and Chebyshev polynomials}

Our goal is now to describe the image $\bigcos(\TTorus^n)$ of the compact torus under the map $\bigcos$, when the lattice $\Weights$ is the weight lattice of a root system with Weyl group $\weyl$. 
The integer representation is denoted by $\gva$. 
For $1\leq i\leq n$, we fix $\orb{i} := \bigorb{e_i}$ as the orbit polynomial associated to $e_i$ from \Cref{defi_orbitpoly}. 

\begin{theorem}\label{theorem_BourbakiGenerators}
\emph{\cite[Ch. VI, \S 3, Thm. 1]{bourbaki456}} The orbit polynomials $\orb{1}, \ldots, \orb{n}$ are algebraically independent and $\Rx^{\gva} = \Q[\orb{1},\ldots,\orb{n}]$ is a polynomial algebra.
\end{theorem}

Conversely, Steinberg \cite{steinberg75} and Farkas \cite{farkas84} proved that $\Rx^\gva$ is a polynomial algebra if and only if $\gva$ is the integer representation of a Weyl group defined via the weight lattice as in \Cref{eq_integerrepresentation}. 

According to \Cref{theorem_OrbitSperation}, the map $x\mapsto\bigcos(x):=(\orb{1}(x),\ldots,\orb{n}(x))$ is injective and thus the image $\bigcos(\TTorus^n)$ is in bijection to $\TTorus^n/\gva$. 

\begin{definition}
We call $\Image:=\bigcos(\TTorus^n)$ the \textbf{$\TTorus$--orbit space} of $\gva$. 
\end{definition}

By fixing the fundamental invariants, we have chosen an embedding of the orbit space in the variety defined by the relations of the fundamental invariants $\orb{i}$. 
Since those are algebraically independent, it is an embedding in a real subspace isomorphic to $\R^n$. 
We shall see that the $\TTorus$--orbit space $\Image$ is a compact basic semi--algebraic set. 


To give the defining polynomial inequalities, it suffices to consider the irreducible root systems. 
Indeed, the Weyl group $\weyl$ of a root system $\Roots$ is the outer product of the Weyl groups corresponding to the irreducible components, see the discussion before \cite[Ch. VI, \S 1, Prop. 5]{bourbaki456}. 
If $\Roots$ is not irreducible, that is, if $\Roots=\Roots^{(1)}\cup\ldots\cup\Roots^{(k)}$, then the $\TTorus$--orbit space $\Image$ can be written as a Cartesian product of $\TTorus$--orbit spaces $\Image^{(1)}\times\ldots\times\Image^{(k)}$.

\begin{example}\label{example_multiplicativeaction}
The symmetric group $\mathfrak{S}_{3}$ acts on $\R^3/[1,1,1]^t$ by permutation of coordinates. 
It is the Weyl group $\weyl$ of the root system $\RootA[2]$. 
The weight lattice $\Weights=\Z\,\fweight{1}\oplus\Z\,\fweight{2}$ is hexagonal with $-\fweight{1}\in\weyl\fweight{2}$, see \Cref{A2FundomChangeOfBasis}. 
Over the integers, $\weyl$ is represented as the group $\gva\subseteq \mathrm{GL}_2(\Z)$ generated by
\[
B_1
=	\begin{bmatrix}
	1 &  1 \\
	0 & -1
\end{bmatrix}
\tbox{and}
B_2
=	\begin{bmatrix}
	-1 &  0 \\
	1 &  1
\end{bmatrix}.
\]
The $\TTorus$--orbit space $\Image$ of $\gva$ is a $2$--dimensional real semi--algebraic set, namely the deltoid in \Cref{A2FundomChangeOfBasis}, whose polynomial description is given in \Cref{section_proofA}. 
It is parametrized by the map
\[
	\bigcos:
	\TTorus^2\to\{z\in\C^2\,\vert\,z_1=\overline{z_2}\} \cong \R^2,\,
	x\mapsto (\orb{1}(x),\orb{2}(x))
=	\left(\frac{x_1+x_2/x_1+1/x_2}{3},\,
	\frac{x_2+x_1/x_2+1/x_1}{3}\right).
\]
The cusps of $\Image$ are the three points $\bigcos(\mathfrak{e}^u)$, where $u \in \{ 0, \fweight{1}, \fweight{2} \}$ and $\mathfrak{e}^u:=(e^{-2\pi\mathrm{i}\langle\fweight{1},u\rangle},e^{-2\pi\mathrm{i}\langle\fweight{2},u\rangle})\in\TTorus^2$.
\begin{figure}[H]
	\begin{center}
		\includegraphics[width=0.9\textwidth]{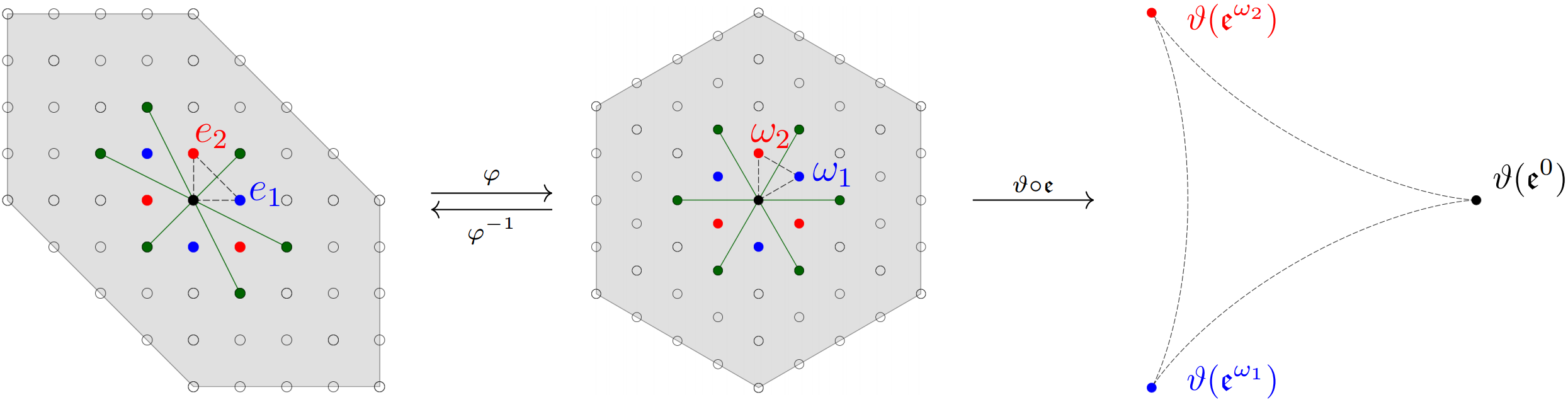}
		\caption{
		The \textcolor{OliveGreen}{root system} $\RootA[2]$ admits a hexagonal weight lattice $\Weights$ (circles $\circ$) with fundamental weights $\textcolor{blue}{\fweight{1}},\textcolor{red}{\fweight{2}}$ and their orbits. 
		Here are the usual orthogonal representation with $\weyl$--symmetry (middle) and the integer representation with $\gva$--symmetry (left). 
		The isomorphism $\varphi:\Z^2\to\Weights$ yields the change of basis from $\{e_1,e_2\}$ to $\{\fweight{1},\fweight{2}\}$. 
		The $\TTorus$--orbit space of $\gva$ (right) is the ``deltoid'', also known as ``Steiner's hypocycloid''.
		}
		\label{A2FundomChangeOfBasis}
	\end{center}
\end{figure}

\end{example}


Thanks to \Cref{theorem_BourbakiGenerators}, we can define a family of multivariate polynomials associated to the weight lattice $\Weights\cong\Z^n$, which generalize the univariate Chebyshev polynomials. Denote by $\RX:=\Q[z_1,\ldots,z_n]$ the multivariate polynomial ring.

\begin{definition}\label{defi_ChebyPoly1}
The \textbf{generalized Chebyshev polynomials of the first kind} associated to $\alpha\in\Z^n$ is the unique $T_{\alpha}\in\RX$ satisfying $T_{\alpha}(\bigcos(x)) = \bigorb{\alpha}(x)$.
\end{definition}

For every $\tilde{\alpha}\in\Z^n$, there exists a unique $\alpha\in\N^n$, such that $\tilde{\alpha}\in\gva\alpha$ \cite[Ch. VI, \S 1, Thm. 2]{bourbaki456}. Hence, we have $T_\alpha=T_{\tilde{\alpha}}$ and $\{T_{\alpha}\,\vert\,\alpha\in\N^n\}$ is a basis for the $\R$--vector space $\RX$. We have $T_0=1$ and $T_{e_i}=z_i$ for $1\leq i \leq n$. The $T_\alpha$ satisfy the same recurrence formula as the $\bigorb{\alpha}$, see \Cref{proposition_RecurrenceOrbitPolynomials}.

Beyond those properties, one can show that the $T_\alpha$ are orthogonal on the $\TTorus$--orbit space and the weight function is closely related to the determinant of the Jacobian of the map $\bigcos$. 
For a proof, we refer to \cite{HoffmanWithers}. 

\begin{example}
The generalized Chebyshev polynomials of the first kind associated to $\RootA[2]$ are
\begin{gather*}
T_{0 2} = 3\,z_2^2 - 2\,z_1, \,
T_{1 1} = 3/2\,z_1\,z_2 - 1/2, \, 
T_{2 0} = 3\,z_1^2 - 2\,z_2, \,
T_{0 3} = 9\,z_2^3 - 9\,z_1\,z_2 + 1, \\
T_{1 2} = 9/2\,z_1\,z_2^2 - 3\,z_1^2 - 1/2\,z_2, \,
T_{2 1} = -3\,z_2^2 - 1/2\,z_1 + 9/2\,z_1^2\,z_2, \,
T_{3 0} = 9\,z_1^3 - 9\,z_1\,z_2 + 1, \, \ldots
\end{gather*}
The generalized Chebyshev polynomials of the first kind associated to $\RootB[2]$ are
\begin{gather*}
T_{0 2} = 4\,z_2^2-2\,z_1-1, \, 
T_{1 1} = 2\,z_1\,z_2-z_2, \, 
T_{2 0} = 4\,z_1^2-8\,z_2^2+4\,z_1+1, \,
T_{0 3} = 16\,z_2^3-12\,z_1\,z_2-3\,z_2, \\
T_{1 2} = 8\,z_1\,z_2^2-4\,z_1^2-3\,z_1, \, 
T_{2 1} = 8\,z_1^2\,z_2-16\,z_2^3+6\,z_1\,z_2+3\,z_2, \,
T_{3 0} = 16\,z_1^3-48\,z_1\,z_2^2+24\,z_1^2+9\,z_1, \, \ldots
\end{gather*}
\end{example}


The following is only relevant if the permutation $\sigma\in\mathfrak{S}_n$ from \Cref{remark_PermutationOrb2} is not the identity. 
In this case, $-I_n\notin\gva$ and $\Image$ is not necessarily a subset of $\R^n$. 
The irreducible root systems, for which this is the case, are both $\RootA$ and $\RootD[2n-1]$ whenever $n\geq 3$ as well as $\RootE[6]$. 
For $x\in\TTorus^n$ and $1\leq i\leq n$, we have
\[
	\overline{\orb{i}(x)}
=	\orb{i}(-I_n\star x)
=	(-I_n\cdot\orb{i})(x)
=	\orb{\sigma(i)}(x).
\]
When $j=\sigma(j)$, we leave the $j$--th coordinate of $\bigcos$ as it is and set $\orb{j,\R} := \orb{j}$. 
When $j < \sigma(j)$, we replace the $j$--th, respectively $\sigma(j)$--th, coordinate of $\bigcos$ by $\orb{j,\R} := \Re(\orb{j}) = (\orb{j} + \orb{\sigma(j)})/2$, respectively $\orb{\sigma(j),\R} := \Im(\orb{j}) = (\orb{j} - \orb{\sigma(j)})/(2\mathrm{i})$. 
The so obtained map
\begin{equation}\label{equation_RealBigCos}
\bigcos_{\R}:
\begin{array}[t]{ccl}
\TTorus^{n}	& \to 		&	\R^{n}, \\
x 			& \mapsto 	&	\left( \orb{1,\R}( x),\ldots, \orb{n,\R}( x)\right)
\end{array}
\end{equation}
has image $\Image_\R\subseteq [-1,1]^n\subseteq \R^n$. This image is called the \textbf{real $\TTorus$--orbit space} of $\gva$. On $\Image_\R$, we now define the real generalized Chebyshev polynomials.

\begin{proposition}\label{eq_RealPartCheby}
Let $\alpha,\conj{\alpha} \in \N^n$ with $-\alpha\in\gva\conj{\alpha}$. 
Then there exist unique polynomials $\TT_\alpha , \TT_{\conj{\alpha}} \in \RX$, such that
\[
	T_{{\alpha}}	(\bigcos(x))
=	\TT_{{\alpha}}	(\bigcos_{\R}(x)) + \mathrm{i} \, \TT_{\conj{\alpha}}	(\bigcos_{\R}(x))
\tbox{and}
	T_{\conj{\alpha}}	(\bigcos(x))
=	\TT_{{\alpha}}	(\bigcos_{\R}(x)) - \mathrm{i}\, \TT_{\conj{\alpha}}	(\bigcos_{\R}(x)) .
\]
\end{proposition}
\begin{proof}
Assume that $T_{\alpha} = \sum_\beta c_\beta\,z^\beta\in\QX$ for some $c_\beta\in\Q$ and $\beta\in\N^n$. 
Then we have
\[
	T_{\alpha} (\bigcos(x))
=	\sum_\beta c_\beta \prod\limits_{j=1}^n \left(\Re(\orb{j}(x))+\mathrm{i}\,\Im(\orb{j}(x))\right)^{\beta_j}
\tbox{and}
	T_{\conj{\alpha}} (\bigcos(x))
=	\sum_\beta c_\beta \prod\limits_{j=1}^n \left(\Re(\orb{\sigma(j)}(x))+\mathrm{i}\,\Im(\orb{\sigma(j)}(x))\right)^{\beta_j} .
\]
If $j=\sigma(j)$, then $\Re(\orb{j}) = \orb{j,\R}$ and $\Im(\orb{j}) = 0$. 
Otherwise, the definition of $\orb{j,\R}$ implies
\[
	\Re(\orb{j}(x))
=	\begin{cases}
	\orb{j,\R}(x),&\mbox{if } j < \sigma(j)\\
	\orb{\sigma(j),\R}(x),&\mbox{if } j > \sigma(j)
	\end{cases}
\tbox{and}
	\Im(\orb{j}(x))
=	\begin{cases}
	\orb{\sigma(j),\R}(x),&\mbox{if } j < \sigma(j)\\
	-\orb{j,\R}(x),&\mbox{if } j > \sigma(j)
	\end{cases}
.
\]
Altogether, we obtain
\[
\begin{array}[t]{rcl}
	\frac{T_{\alpha} (\bigcos(x)) + T_{\conj{\alpha}} (\bigcos(x))}{2}
&=&	\sum\limits_\beta \frac{c_\beta}{2}	\prod\limits_{j=\sigma(j)} \orb{j,\R}(x)^{\beta_j}\\
&&	\left(	\prod\limits_{j<\sigma(j)} \left( \orb{j,\R}(x) + \mathrm{i}\,\orb{\sigma(j),\R}(x) \right)^{\beta_j}
			\left( \orb{j,\R}(x) - \mathrm{i}\,\orb{\sigma(j),\R}(x) \right)^{\beta_{\sigma(j)}} \right.\\
&&	+\left.	\prod\limits_{j<\sigma(j)} \left( \orb{j,\R}(x) - \mathrm{i}\,\orb{\sigma(j),\R}(x) \right)^{\beta_j}
			\left( \orb{j,\R}(x) + \mathrm{i}\,\orb{\sigma(j),\R}(x) \right)^{\beta_{\sigma(j)}} \right)
			.
\end{array}
\]
This is a polynomial in $\bigcos_\R(x) = (\orb{1,\R}(x),\ldots,\orb{n,\R}(x))$, denoted by $\TT_\alpha$. 
Since the left hand side is real for every $x$, the coefficient in front of $\mathrm{i}$ must be $0$. 
Hence, we have $\TT_\alpha\in\QX$. 
Similarly, by computing $(T_{\alpha} - T_{\conj{\alpha}})/(2\mathrm{i})$, we obtain $\TT_{\conj{\alpha}}\in\QX$. 
Since this is an analytical identity on $\C^n$, the polynomials are unique. 
\end{proof}

\begin{example}\label{example_A2Cheb}
Consider the root system $\RootA[2]$ with Weyl group $\mathfrak{S}_3$. 
The generators of the integer representation $\gva$ are given in \emph{\Cref{example_multiplicativeaction}} and we have $\conj{{1 \brack 0}} = {0 \brack 1} \in\Z^2$, that is, $-e_1\in\gva\,e_2$. 
The generalized Chebyshev polynomials of degree $2$ are
\[
	T_{0 2}	=	3\,z_2^2 - 2\,z_1,		\quad
	T_{1 1}	=	3/2\,z_1\, z_2 - 1/2,	\quad
	T_{2 0}	=	3\,z_1^2 - 2\,z_2		\in \RX .
\]
After substitution $z_1\mapsto \widehat{z}_1+\mathrm{i}\,\widehat{z}_2,\,z_2\mapsto \widehat{z}_1-\mathrm{i}\,\widehat{z}_2$, we have
\[
T_{2 0},\,T_{0 2} = (3\,\widehat{z}_1^2 - 3\,\widehat{z}_2^2 - 2\,\widehat{z}_1)	\pm \mathrm{i} \, (6\, \widehat{z}_1\,\widehat{z}_2 + 2\,\widehat{z}_2), \quad
T_{1 1} = (3/2\,\widehat{z}_1^2+3/2\,\widehat{z}_2^2-1/2),
\]
and so the new polynomials from \emph{\Cref{eq_RealPartCheby}} are
\[
	\TT_{0 2} =		6\, \widehat{z}_1\,\widehat{z}_2 + 2\,\widehat{z}_2,	\quad
	\TT_{1 1} =		3/2\,\widehat{z}_1^2+3/2\,\widehat{z}_2^2-1/2,		\quad
	\TT_{2 0} =		3\,\widehat{z}_1^2 - 3\,\widehat{z}_2^2 - 2\,\widehat{z}_1			.
\]
Those are defined on the embedding of the ``deltoid'' in $\R^2$ from \emph{\Cref{example_multiplicativeaction}}, which is the real $\TTorus$--orbit space $\Image_\R$. At higher degrees, the real generalized Chebyshev polynomials of the first kind associated to $\RootA[2]$ are
\begin{gather*}
\TT_{0 3} = 27\,\widehat{z}_1^2\,\widehat{z}_2-9\,\widehat{z}_2^3, \quad 
\TT_{1 2} = 6\,\widehat{z}_1\,\widehat{z}_2-1/2\,\widehat{z}_2+9/2\,\widehat{z}_1^2\,\widehat{z}_2+9/2\,\widehat{z}_2^3, \\
\TT_{2 1} = -3\,\widehat{z}_1^2+3\,\widehat{z}_2^2-1/2\,\widehat{z}_1+9/2\,\widehat{z}_1^3+9/2\,\widehat{z}_1\,\widehat{z}_2^2, \quad 
\TT_{3 0} = 9\,\widehat{z}_1^3-27\,\widehat{z}_1\,\widehat{z}_2^2-9\,\widehat{z}_1^2-9\,\widehat{z}_2^2+1, \\
\TT_{0 4} = 108\,\widehat{z}_1^3\,\widehat{z}_2-108\,\widehat{z}_1\,\widehat{z}_2^3-36\,\widehat{z}_1^2\,\widehat{z}_2-36\,\widehat{z}_2^3-12\,\widehat{z}_1\,\widehat{z}_2+4\,\widehat{z}_2, \\
\TT_{1 3} = 27/2\,\widehat{z}_2^3-5/2\,\widehat{z}_2+27\,\widehat{z}_1\,\widehat{z}_2^3+27/2\,\widehat{z}_1^2\,\widehat{z}_2+27\,\widehat{z}_1^3\,\widehat{z}_2-3\,\widehat{z}_1\,\widehat{z}_2, \\
\TT_{2 2} = -18\,\widehat{z}_1^3-1/2+54\,\widehat{z}_1\,\widehat{z}_2^2+6\,\widehat{z}_1^2+6\,\widehat{z}_2^2+27/2\,\widehat{z}_1^4+27\,\widehat{z}_1^2\,\widehat{z}_2^2+27/2\,\widehat{z}_2^4, \\
\TT_{3 1} = -27/2\,\widehat{z}_1^3-27/2\,\widehat{z}_1\,\widehat{z}_2^2+5/2\,\widehat{z}_1-3/2\,\widehat{z}_1^2+3/2\,\widehat{z}_2^2+27/2\,\widehat{z}_1^4-27/2\,\widehat{z}_2^4, \\ 
\TT_{4 0} = 27\,\widehat{z}_1^4-162\,\widehat{z}_1^2\,\widehat{z}_2^2+27\,\widehat{z}_2^4-36\,\widehat{z}_1^3-36\,\widehat{z}_1\,\widehat{z}_2^2+6\,\widehat{z}_1^2-6\,\widehat{z}_2^2+4\,\widehat{z}_1,\,\ldots
\end{gather*}
\end{example}

\section{Symmetric polynomial systems}
\label{section:symmetricpolynomialsandroots}
\setcounter{equation}{0}

In order to give a polynomial description for the $\TTorus$--orbit space $\Image$ of a Weyl group $\gva$, our first step is to characterize when a symmetric polynomial system has a solution in the compact torus. 
The present section should be regarded as a toolbox for the proofs of the main results. 

For $1\leq i \leq n$, we denote by $\sigma_i(y_1,\ldots,y_n) = \sum_{\nops{K}=i} \prod_{k\in K} y_k$ the $i$--th elementary symmetric function in $n$ indeterminates $y_i$, where $K$ ranges over all subsets of $\{1,\ldots,n\}$ with cardinality $i$. 
We shall be confronted with the following two types of polynomial systems. 
\[
\begin{array}{clcl}
	\mathrm{(I)}	&	\sigma_i(y_1			,\ldots,y_n				)	&=&	(-1)^i\,c_i \tbox{for} 1\leq i \leq n \tbox{with} c_1,\ldots,c_n\in\C \\
	\mathrm{(II)}	&	\sigma_i\left(\frac{y_1+y_1^{-1}}{2}	,\ldots,\frac{y_n+y_n^{-1}}{2}	\right)	&=&	(-1)^i\,c_i \tbox{for} 1\leq i \leq n \tbox{with} c_1,\ldots,c_n\in\R
\end{array}
\]
The goal is to determine, whether all complex solutions $y=(y_1,\ldots,y_n)$ of system $\mathrm{(I)}$, respectively $\mathrm{(II)}$, are contained in the compact torus $\TTorus^n$.

We pursue this goal, because the above situation arises in \Cref{section_proofA,section_proofC,section_proofB,section_proofD,section_proofG,section_proofMain}. 
Later, the $y_i$ will assume the role of multivariate monomials in $\Rx$, such that $y(x)\in\TTorus^n$ if and only if $z=\bigcos(x)\in\Image$ and the left hand side of system $\mathrm{(I)}$ or $\mathrm{(II)}$ is an invariant in $\Rx^\gva$. 
This section will provide an indirect criterion for $z\in\Image$ in terms of the coefficients $c_i$, which will later be substituted by polynomials in $z$. 


\subsection{Solutions in the compact torus}

Recall Vieta's formula $\prod\limits_{k=1}^n (x-y_k) = x^n + \sum\limits_{i=1}^n (-1)^i\,\sigma_i(y_1,\ldots,y_n)\,x^{n-i}$.
Here, $x$ is a univariate indeterminate. 

\begin{lemma}\label{Correspondence2S}
\begin{enumerate}
\item The map
\[
\begin{array}[t]{ccl}
\C^*	&	\to		&	\C,\\
 x		&	\mapsto	&	\frac{x+x^{-1}}{2}
\end{array}
\]
is surjective and the preimage of the real interval $[-1,1] \subseteq \R$ is $\TTorus$.
\item System $\mathrm{(I)}$ has a unique solution $y\in\C^n$ up to permutation of its coordinates $y_i$.
\item System $\mathrm{(II)}$ has a unique solution $y\in(\C^*)^n$ up to permutation and inversion of its coordinates $y_i$. 
\end{enumerate}
\end{lemma}
\begin{proof}
\emph{1.} For $r\in\C$, consider the univariate polynomial $p:=x^2-2\,r\,x+1$. Then $0$ is not a root of $p$ and we have $p(x)=0$ if and only if $r=(x+x^{-1})/2$, that is, $r$ is in the image of the map. If $r\in [-1,1]$, then the discriminant of $p$ is $r^2-1\leq 0$ and the two roots are $ x, \overline{x} = x,x^{-1} = r \pm \mathrm{i}\,\sqrt{1-r^2}\in\TTorus$. On the other hand, for $x\in\TTorus$, we have $ (x +  x^{-1})/2=(x+\overline{ x})/2 = \Re( x) \in [-1,1]$.

\emph{2.} By Vieta's formula, a solution $y$ of system $\mathrm{(I)}$ is the vector of roots of the polynomial $x^n+c_1\,x^{n-1}+\ldots+c_n$. Since $\C$ is algebraically closed, the statement follows.

\emph{3.} By \emph{1.}, we can write the roots $r_1,\ldots,r_n$ of the polynomial $x^n+c_1\,x^{n-1}+\ldots+c_n$ as $r_i=(y_i+y_i^{-1})/2$ for some $y \in (\C^*)^n$. Then $y$ is a unique solution of $\mathrm{(II)}$ up to permutation and inversion.
\end{proof}

From now on, we speak of ``the'' solution of system $\mathrm{(I)}$, respectively $\mathrm{(II)}$.

\begin{proposition}\label{SolutionsSystemII}
For $c_1,\ldots,c_n \in\R$, the corresponding solution of system $\mathrm{(II)}$ is contained in $\TTorus^n$ if and only if all the roots of the univariate polynomial
\[
	x^n+c_1\,x^{n-1}+\ldots+c_{n-1}\,x+c_n
\]
are contained in $[-1,1]$.
\end{proposition}
\begin{proof}
If $y\in(\C^*)^n$ is the solution of system $\mathrm{(II)}$, then the roots of $x^n+c_1\,x^{n-1}+\ldots+c_n$ are $r_i := (y_i+ y_i^{-1})/2$. 
According to \Cref{Correspondence2S}, we have $ y\in\TTorus^n$ if and only if $r_1,\ldots, r_n\in[-1,1]$.
\end{proof}

Similarly, we can characterize solutions of system $\mathrm{(I)}$. Recall that the univariate Chebyshev polynomial of the first kind associated to $\ell\in\N$ is the unique $T_\ell$ with $T_\ell((x+x^{-1})/2)=(x^\ell+x^{-\ell})/2$. 
For $p\neq0$ a polynomial in $x$, we denote by $\coeff(x^\ell,p)$ the coefficient of the monomial $x^\ell$ in $p$ for $0\leq \ell \leq \deg(p)$.

\begin{proposition}\label{SolutionsSystemI}
For $c_1,\ldots,c_{n-1}\in\C$ with $\overline{c_i}=(-1)^n \,c_{n-i}$ and $c_0:=(-1)^n  \,c_n:=1$, the corresponding solution of system $\mathrm{(I)}$ is contained in $\TTorus^n$ if and only if all the roots of the univariate polynomial
\[
	T_n(x) + d_1\,T_{n-1}(x) + \ldots + d_{n-1} \, T_1(x) + \dfrac{d_n}{2} \, T_0(x)
	\tbox{with}
	d_{\ell}=\sum_{i=0}^\ell \overline{c_i}\,c_{\ell-i} \in \R
\]
are contained in $[-1,1]$.
\end{proposition}
\begin{proof}
The solution of system $\mathrm{(I)}$ is contained in $\TTorus^n$ if and only if all the roots of the univariate polynomial $p:=x^n+c_1\,x^{n-1}+\ldots+c_n$ are contained in $\TTorus$.

The roots of $p$ are nonzero, because $p(0)=c_n=(-1)^n \neq 0$. We fix another univariate polynomial $\tilde{p}:=x^n+\overline{c_1}\,x^{n-1}+\ldots+\overline{c_n}$. Since $\tilde{p}(x) = (-x)^n\,p(x^{-1})$, the set of roots of $p\tilde{p}$ consists of the roots of $p$ and their inverses. 
In particular, all the roots of $p\tilde{p}$ are contained in $\TTorus$ if and only if all the roots of $p$ are. 
For $0\leq \ell\leq n$, the coefficients of $p\tilde{p}$ satisfy
\[
	\coeff(x^\ell,p\tilde{p})	
=	\sum\limits_{i=0}^\ell 	c_{n-i} \, \overline{c_{n-\ell+i}}
=	\sum\limits_{i=0}^\ell \overline{c_{i}} \, c_{\ell-i}
=	\sum\limits_{i=0}^\ell \coeff(x^{n-i},\tilde{p}) \, \coeff(x^{n-\ell+i},p)	
=	\coeff(x^{2n-\ell},p\tilde{p}) .
\]
The coefficients are real, because
\[
	\sum\limits_{i=0}^\ell \overline{c_{i}} \, c_{\ell-i}
=	\sum\limits_{i=0}^{\lfloor (\ell-1)/2 \rfloor} \underbrace{(\overline{c_i}\,c_{\ell-i} +  \overline{c_{\ell-i}}\,c_i)}_{\in\R} +
\begin{cases}
	c_{\ell/2}\,\overline{c_{\ell/2}},	&	\ell \mbox{ even}\\
	0,									&	\ell \mbox{ odd}
\end{cases}
\in	\R .
\]
Thus, $\coeff(x^\ell,p\tilde{p})=\coeff(x^{2n-\ell},p\tilde{p})=d_\ell\in\R$ and 
we can write
\[
	p\tilde{p}
=	\sum\limits_{\ell=1}^{n}	d_{n-\ell} (x^{n+\ell}+x^{n-\ell}) + d_n
=	2  x^n \left( \sum\limits_{\ell=1}^{n}	d_{n-\ell} T_\ell\left(\dfrac{x+x^{-1}}{2}\right) + \dfrac{d_n}{2} T_0\left(\dfrac{x+x^{-1}}{2}\right) \right)
=:	2 x^n g\left(\dfrac{x+x^{-1}}{2}\right).
\]
With \Cref{Correspondence2S}, we see that $x\in\TTorus$ is a root of $p\tilde{p}$ if and only if $(x+x^{-1})/2\in [-1,1]$ is a root of $g$.
\end{proof}

The following example illustrates the proof. 

\begin{example}\label{Example_PalindromicPoly}
For $c\in\C$, set $p=x^3-c\,x^2+\overline{c}\,x-1$ and consider the palindromic polynomial $p\tilde{p}\in\Rt$ with
\begin{align*}
	&	\,\dfrac{ 1}{2 x^{3}} \, p(x) \tilde{p}(x) \\
	=	&	\,\dfrac{ 1}{2 x^{3}} \, \left((x^6+1)-(c+\overline{c})(x^5+x)+(c\,\overline{c}+c+\overline{c})(x^4+x^2)-\dfrac{c^2+\overline{c}^2+2}{2}\,x^3\right)\\
	=	&	\,\dfrac{x^3+x^{-3}}{2}-(c+\overline{c})\,\dfrac{x^2+x^{-2}}{2}+(c\,\overline{c}+c+\overline{c})\,\dfrac{x+x^{-1}}{2}-\dfrac{c^2+\overline{c}^2+2}{2} \\
	=	&	\,T_3\left(\dfrac{x+x^{-1}}{2}\right)-(c+\overline{c})\,T_2\left(\dfrac{x+x^{-1}}{2}\right)+(c\,\overline{c}+c+\overline{c})\,T_1\left(\dfrac{x+x^{-1}}{2}\right)-\dfrac{c^2+\overline{c}^2+2}{2}\,T_0\left(\dfrac{x+x^{-1}}{2}\right)\\
	=	&	\,g\left(\dfrac{x+x^{-1}}{2}\right).
\end{align*}
All the roots of $p$ are contained in $\TTorus$ if and only if all the roots of $g$ are contained in $[-1,1]$, that is, if and only if the solution of system $\mathrm{(I)}$ with right hand side $c_1=-c,c_2=\overline{c},c_3=-1$ is in $\TTorus^3$.
\end{example}

\subsection{Characterization via Hermite quadratic forms}
\label{subsec_SturmSylvester}

Let $p,q\in\Rt$ be univariate polynomials with $p\neq 0$. The multiplication by $q$ in the algebra $\Rt/\langle p \rangle$ is
\[
m_q:
\begin{array}[t]{ccl}
\Rt/\langle p \rangle	&	\to		&	\Rt/\langle p\rangle , \\
f + \langle p \rangle	&	\mapsto	&	q \, f + \langle p \rangle .
\end{array}
\]
A classical result, usually referred to as Sylvester's version of Sturm's theorem, allows us to effectively characterize when a univariate polynomials has all of its roots in $[-1,1]$.

\begin{theorem}\label{HermiteCharacterization}
All the roots of $p$ are contained in $S(q):=\{ x \in\R \,\vert\, q(x) \geq 0 \}$ if and only if the Hermite quadratic form
\[
	H(p,q):\begin{array}[t]{rcl}
	\Rt/\langle p \rangle	& 	\to		&	\R , \\
	f+\langle p \rangle		&	\mapsto	&	\trace (m_{q\,f^2})
	\end{array}
\]
is positive semi--definite.
\end{theorem}
\begin{proof}
Let $H\in\R^{n\times n}$ be the symmetric matrix associated to $H(p,q)$ for a fixed basis of $\Rt/\langle p \rangle$, where $n=\deg(p)$. Denote by $N_+$, respectively $N_-$, the number of strictly positive, respectively negative, eigenvalues of $H$, including their multiplicities. By \cite[Ch. 2, Thm. 5.2]{cox05}, the rank and the signature of $H(p,q)$ are
\[
\begin{array}{rcrcl}
N_+ + N_-	& = &	\mathrm{Rank}(H(p,q))	& = &	\vert\{  x\in\C \, \vert\, p(x) =0 ,\, q( x) \neq 0 \}\vert ,	\\
N_+ - N_-	& = &	\mathrm{Sign}(H(p,q))	& = &	\underbrace{\vert\{  x\in \R \, \vert \,p(x)=0,\, q( x) > 0 \}\vert}_{=:n_+} \, - \, \underbrace{\vert\{  x\in \R \, \vert \, p(x)=0 ,\, q( x) < 0 \}\vert}_{=:n_-}.
\end{array}
\]
If all the roots of $p$ are contained in $S(q)$, then $n_-=0$, and thus
\[
	N_+ + N_-
=	\mathrm{Rank}(H(p,q))
=	\mathrm{Sign}(H(p,q))
=	n_+
=	N_+ - N_-.
\]
Hence, $N_-=0$ and all eigenvalues of $H$ are nonnegative, that is, $H(p,q)$ is positive semi--definite.

Conversely, assume that $H(p,q)$ is positive semi--definite. Then $N_-=0$ and
\[
	N_+
=	\mathrm{Sign}(H(p,q))
=	n_+-n_-\leq \mathrm{Rank}(H(p,q))
=	N_+,
\]
that is, $ n_+ - n_- = \mathrm{Rank}(H(p,q))$. 
On the other hand, $\mathrm{Rank}(H(p,q)) \geq n_+ + n_-$. 
Hence, $n_-=0$ and $\mathrm{Rank}(H(p,q)) = n_+$ implies that all the roots of $p$ are real and contained in $S(q)$.
\end{proof}


Finally, we compute the matrix of the Hermite quadratric form. 
If $p=x^n+c_1\,x^{n-1}+\ldots+c_{n-1}\,x+c_n$, then the matrix of the multiplication $m_x$ in $\Rt/\langle p \rangle$ in the basis $\{1,x,\ldots,x^{n-1}\}$ is the companion matrix
\begin{equation}\label{equation_CompMatrix}
	\begin{bmatrix}
		0		&			&	0		&	-c_n	\\
		1		&	\ddots	&			&	\vdots	\\
		&	\ddots	&	0		&	-c_2	\\
		0		&			&	1		&	-c_1	
	\end{bmatrix}
\end{equation}
of $p$, because $x\,x^i=x^{i+1}$ whenever $0\leq i\leq n-2$ and $x\,x^{n-1} = x^n\equiv-c_1\,x^{n-1}-\ldots-c_{n-1}\,x-c_n \mbox{  mod  } \langle p \rangle$.

\begin{corollary}\label{Corollary_SolutionsSystemII}
For $c_1,\ldots,c_{n} \in\R$, denote by $C\in\R^{n\times n}$ the matrix from \emph{\Cref{equation_CompMatrix}}. 
Then the solution of system $\mathrm{(II)}$ is contained in $\TTorus^n$ if and only if $H\succeq0$, where $H\in\R^{n\times n}$ has entries $H_{ij} = \trace(C^{i+j-2}-C^{i+j})$.
\end{corollary}
\begin{proof}
$H$ is the matrix associated to the Hermite quadratic form $H(p,q)$ from \Cref{HermiteCharacterization} with $q = 1 - x^2 $ in the basis $\{1,x,x^2,\ldots,x^{n-1}\}$. Indeed, by \cite[Ch. 2, Prop. 4.2]{cox05}, the matrix entries are
\[
	H_{ij}
=	\trace(m_{q\,x^{i-1}\,x^{j-1}})
=	\trace(m_{x^{i+j-2}-x^{i+j}})
=	\trace(m_x^{i+j-2}-m_x^{i+j})
\]
for $1\leq i,j\leq n$. 
Since the trace is basis independent, we can replace $m_x$ with its matrix $C$ given by \Cref{equation_CompMatrix}. 
The statement now follows from \Cref{SolutionsSystemII}, because $\{x\in\R\,\vert\,q(x)\geq 0\} = [-1,1]$.
\end{proof}

On the other hand, the univariate Chebyshev polynomials also form a basis $\{T_0,T_1,\ldots,T_{n-1}\}$ of $\Rt/\langle p \rangle$. 
If $n\geq 3$ and $p=T_n+d_1\,T_{n-1}+\ldots+d_{n-1}\,T_1+d_n/2 \, T_0\in\Rt$, then the matrix of $m_x$ is
\begin{equation}\label{equation_CompMatrixChebyshev}
	\begin{bmatrix}
		0	&	1/2	&			&			&	0	&	-d_n/4		\\
		1	&	0	&	\ddots	&			&		&	-d_{n-1}/2	\\
		&	1/2	&	\ddots	&	\ddots	&		&	\vdots			\\
		&		&	\ddots	&	\ddots	&	1/2	&	-d_3/2 		\\
		&		&			&	\ddots	&	0	&	(1-d_2)/2	\\
		0	&		&			&			&	1/2	&	-d_1/2 
	\end{bmatrix}.
\end{equation}
The entries originate from the recurrence formula $2\,x\,T_j=T_{j+1}+T_{j-1}$. 
In particular, the last column is obtained through
\[
2\,x\,T_{n-1}	=	T_n + T_{n-2}	\equiv	-d_1\,T_{n-1}+(1-d_2)\,T_{n-2}-d_3\,T_{n-3}-\ldots-d_{n-1}\,T_1-d_n/2\,T_0  \tbox{mod} \langle p \rangle .
\]

\begin{corollary}\label{Corollary_SolutionsSystemI}
Let $n\geq 3$. For $c_1,\ldots,c_{n-1} \in\C$ with $\overline{c_i}=(-1)^n \,c_{n-i}$ and $c_0:=(-1)^n \,c_n:=1$, set
\[
	d_{\ell}
=	\sum_{i=0}^\ell \overline{c_i}\,c_{\ell-i} \in \R
\]
for $1\leq \ell\leq n$ and denote by $C\in\R^{n\times n}$ the matrix from \emph{\Cref{equation_CompMatrixChebyshev}}. 
Then the solution of system $\mathrm{(I)}$ is contained in $\TTorus^n$ if and only if $H\succeq0$, where $H\in\R^{n\times n}$ has entries $H_{ij} = \trace(C^{i+j-2}-C^{i+j})$.
\end{corollary}
\begin{proof}
The proof is analogous to the one of \Cref{Corollary_SolutionsSystemII} with \Cref{SolutionsSystemI}.
\end{proof}




\section{Type $\RootA$}
\label{section_proofA}
\setcounter{equation}{0}

\begin{figure}[H]
	\begin{center}
		\begin{overpic}[width=0.45\textwidth,grid=false,tics=10]{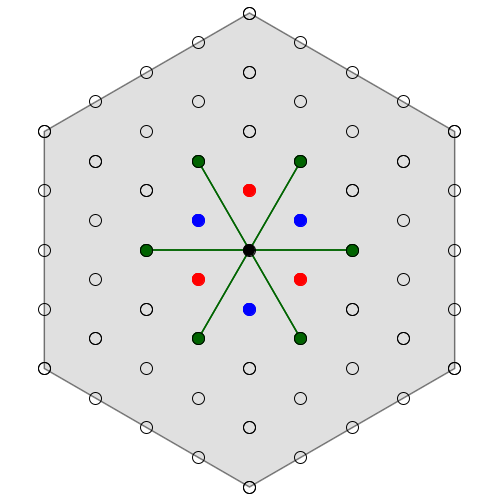}
			\put (63,55) {\large \textcolor{blue}{$\displaystyle \fweight{1}$}}
			\put (50,65) {\large \textcolor{red}{$\displaystyle \fweight{2}$}}
		\end{overpic}
		\quad
		\begin{overpic}[width=0.45\textwidth,grid=false,tics=10]{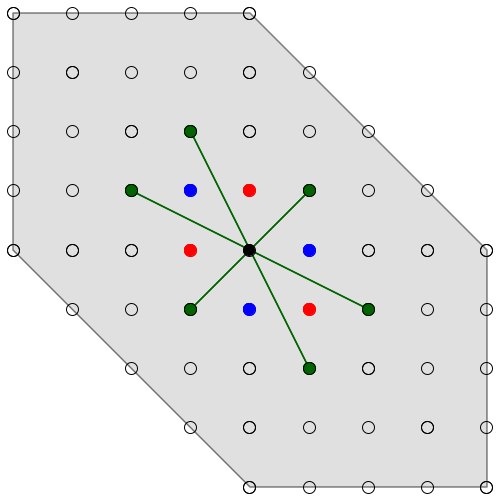}
			\put (64,50) {\large \textcolor{blue}{$\displaystyle e_{1}$}}
			\put (50,65) {\large \textcolor{red}{$\displaystyle e_{2}$}}
		\end{overpic}
	\end{center}
	\caption{The \textcolor{OliveGreen}{root system} $\RootA[2]$ and its weight lattice in the usual orthogonal representation and the integer representation. The orbits of the fundamental weights are the \textcolor{blue}{blue} and \textcolor{red}{red} lattice elements.}
	\label{example_rootsystemA2}
\end{figure}

The group $\mathfrak{S}_{n}$ acts on $\R^n$ by permutation of coordinates and leaves the subspace $V := \R^n/\langle [1,\ldots,1]^t\rangle = \{u\in\R^{n}\,\vert\,u_1+\ldots +u_{n}=0\}$ invariant. The root system $\RootA$ given in \cite[Planche I]{bourbaki456} is a root system of rank $n-1$ in $V$ with fundamental weights
\begin{equation}\label{equation_WeightsRootsA}
	\fweight{i}=\sum_{j=1}^i e_j -\dfrac{i}{n} \sum_{j=1}^{n} e_j = \dfrac{1}{n} (\underbrace{n-i,\ldots, n-i}_{i\,\mbox{\small{ times}}},\underbrace{ -i, \ldots, -i}_{n-i\,\mbox{\small{ times}}})
	\quad(1\leq i\leq n-1).
\end{equation}
Here, the $e_i\in\R^n$ are the Euclidean standard basis vectors. 
In particular, the weight lattice is $\Weights:=\{ u/n \,\vert\, u\in V\cap \Z^n\}$. 
The Weyl group $\weyl\cong\mathfrak{S}_{n}$ of $\RootA$ is orthogonal with respect to the Euclidean scalar product and we have $-\fweight{i}\in\weyl\,\fweight{n-i}$. The orbit $\weyl\,\fweight{i}$ has cardinality $\binom{n}{i}$ for $1\leq i \leq n-1$, see \Cref{example_rootsystemA2}. 

In this section, we give a closed formula for the matrix polynomial from \Cref{MainThmIntro} in the standard monomial basis for the root system $\RootA$. The ring of Laurent polynomials is $\Rx=\Q[x_1,x_1^{-1},\ldots, x_{n-1},x_{n-1}^{-1}]$ and the polynomial ring is $\RX=\Q[z_1,\ldots,z_{n-1}]$. 


\subsection{Orbit polynomials}

With $\weyl\cong\mathfrak{S}_{n}$ and \Cref{equation_WeightsRootsA}, one observes that a weight $\weight\in\Weights$ is contained in the $\weyl$--orbit of the first fundamental weight $\fweight{1}$ if and only if
\[
	\weight = 
	\begin{cases}
	\fweight{1},\\
	-\fweight{n-1},\quad \mbox{or}\\
	\fweight{i+1}-\fweight{i},\quad 1\leq i\leq n-2.
	\end{cases}
\]
We denote by $\gva\in\mathrm{GL}_{n-1}(\Z)$ the integer representation of $\weyl$ defined by \Cref{eq_integerrepresentation}. 
Then the corresponding orbit of monomials $\gva \cdot x_1 = \{ x^{B e_1} \,\vert\, B\in \gva \} \subseteq \Rx$ consists of $n$ distinct monomials, namely
\begin{equation}\label{eq_monomialsA}
	y_1:=x_1 ,\quad y_2:=x_2\,x_{1}^{-1} ,\quad \ldots ,\quad y_{n-1}:=x_{n-1}\,x_{n-2}^{-1} ,\quad  y_n:=x_{n-1}^{-1} .
\end{equation}
For $1\leq i\leq n$, let $\sigma_i$ be the $i$--th elementary symmetric function in $n$ indeterminates. 
Recall that, for $1\leq i \leq n-1$, the $\gva$--invariant orbit polynomial associated to $e_i\in\Z^{n-1}$ from \Cref{defi_orbitpoly} is denoted by $\orb{i}=\bigorb{e_i}$. 

One can now show that $\sigma_i(y(x)) \in \Qx^\gva$ is invariant and thus express it explicitly in terms of the fundamental invariants. 

\begin{proposition}\label{OrbitSumsAn}
In $\Qx$, we have
\[
	\sigma_i(y_1(x),\ldots,y_{n}(x))
=	\binom{ n }{ i } \, \orb{i}(x)
\quad(1\leq i\leq n-1)
\tbox{and}
	\sigma_n(y_1(x),\ldots,y_{n}(x))
=	1.
\]
\end{proposition}
\begin{proof}
It follows from \Cref{eq_monomialsA} that $x_i = y_1(x) \ldots y_i(x)$ and $\gva$ acts on the $y_i(x)$ by permutation. Hence,
\[
	\orb{i}(x)
=	\frac{1}{\vert \gva \vert}\,\sum\limits_{B\in\gva} x^{B\,e_i}
=	\dfrac{\nops{\mathrm{Stab}_{\gva}(x_i)}}{\nops{\gva}} \sum\limits_{\substack{J\subseteq \{1,\ldots,n\} \\ \nops{ J} =i}} \prod\limits_{j\in J} y_j(x)
=	\dfrac{1}{\nops{\gva \cdot x_i}} \,\sigma_i(y_1(x),\ldots ,y_n(x)).
\]
With $\nops{\gva \cdot x_i}=\binom{n}{i}$ and $1=y_1(x)\ldots y_n(x)$, we obtain the statement.
\end{proof}

We denote $\TToruss^n:=\{ y \in \TTorus^n \,\mid\, y_1\ldots y_n = 1 \}$. 
The next statement follows immediately from \Cref{eq_monomialsA}. 

\begin{lemma}\label{proposition_PsiAn}
The map
\[
\Upsilon:
\begin{array}[t]{ccl}
	\TTorus^{n-1}	&	\to 	& \TToruss^n,\\
	x 				&	\mapsto	& 	( y_1(x),\ldots,y_{n}(x) ),
\end{array}
\]
is bijective. 
\end{lemma}

\subsection{Hermite characterization with standard monomials}

We now characterize, whether a given point $z$ is contained in the $\TTorus$--orbit space $\Image$ of $\gva$, that is, 
if the equation $\orb{i}(x)=z_i$ has a solution $x\in\TTorus^{n-1}$. 
As shown in previous subsections, this is equivalent to deciding whether a symmetric polynomial system of type $\mathrm{(I)}$ has its solution in $\TTorus^n$. 
We state our main result for $\RootA$ in the standard monomial basis.

\begin{theorem}\label{HermiteCharacterizationAn}
Let $n\geq 3$. 
Define the $(n-1)$--dimensional $\R$--vector space $\mathcal{Z}:=\{ z\in\C^{n-1} \,\vert\,\forall\, 1\leq i \leq n-1:\, \overline{z_i}=z_{n-i} \}$ and the matrix $H \in \RX^{n\times n}$ by
\begin{align*}
H(z)_{ij}	&=	 \trace((C(z))^{i+j-2}-(C(z))^{i+j}),\tbox{where}
C(z)		=
\begin{bmatrix}
0	&	1/2	&			&			&	0	&	-d_n(z)/4		\\
1	&	0	&	\ddots	&			&		&	-d_{n-1}(z)/2	\\
	&	1/2	&	\ddots	&	\ddots	&		&	\vdots			\\
	&		&	\ddots	&	\ddots	&	1/2	&	-d_3(z)/2 		\\
	&		&			&	\ddots	&	0	&	(1-d_2(z))/2	\\
0	&		&			&			&	1/2	&	-d_1(z)/2 
\end{bmatrix} ,
\\
d_\ell(z)	&=	(-1)^{\ell}\,\sum\limits_{i=0}^{\ell} \binom{ n }{ i }\,\binom{ n }{ \ell - i }\,z_i\,z_{n-\ell+i}\bigg\vert_{z_0=z_n=1} \tbox{for} 1\leq \ell \leq n.
\end{align*}
For $z\in\mathcal{Z}$, we have $H(z)\in\R^{n\times n}$ and $\Image = \{z\in\mathcal{Z}\,\vert\,H(z)\succeq 0\}$.
\end{theorem}
\begin{proof}
For $z\in \C^{n-1}$, define $c_0:=1$, $c_n:=(-1)^n$, $c_i:=(-1)^i\,\binom{n}{i}\,z_i\in\C$ for $1\leq i \leq n-1$ as well as $d_\ell:=d_\ell(z)$ for $1\leq \ell \leq n$.

To show ``$\subseteq$'', assume that $z\in\Image$ and fix $ x\in\TTorus^{n-1}$, such that $\orb{i}( x)=z_i$. By \Cref{OrbitSumsAn} and \Cref{proposition_PsiAn}, the unique solution of
\[
\mathrm{(I)}  \quad  \sigma_i(y_1,\ldots,y_n) = (-1)^i\,c_i \tbox{for} 1\leq i\leq n
\]
is $y=\Upsilon( x)\in\TToruss^n$. 
Note that $\orb{j}( x)$ and $\orb{n-j}( x)$ are complex conjugates, because $-\fweight{j}\in\weyl\, \fweight{n-j}$. Therefore, $z\in \mathcal{Z}$ and $d_\ell=\sum_{i=0}^\ell c_i\,\overline{c_{\ell-i}}\in\R$ yields the last column of $C(z)$. \Cref{Corollary_SolutionsSystemI} gives us $H(z)\succeq 0$.

For ``$\supseteq$'' on the other hand, assume $z\in \mathcal{Z}$ with $H(z)\succeq 0$. By \Cref{Corollary_SolutionsSystemI}, the solution $y$ of system $(\mathrm{I})$ is contained in $\TToruss^n$. Let $ x\in\TTorus^{n-1}$ be the unique preimage of $ y$ under $\Upsilon$. Then by \Cref{OrbitSumsAn},
\[
	z_i=(-1)^i\,\binom{ n }{ i }^{-1}\,c_i=\binom{ n }{ i }^{-1}\,\sigma_i(y_1,\ldots,y_n)=\orb{i}( x)
\]
for $1\leq i\leq n-1$ and so $z=\bigcos(x)\in\Image$.
\end{proof}

For $n=2$, we are in the univariate case with $\Image=[-1,1]$.

\subsection*{Example: $\RootA[2]$}

We study the root system from \Cref{example_rootsystemA2} with Weyl group $\weyl\cong\mathfrak{S}_3$. 
Since $\weyl\fweight{1}=\{\fweight{1},\fweight{2}-\fweight{1},-\fweight{2}\}$ and $\weyl\fweight{2}=\{\fweight{2},\fweight{1}-\fweight{2},-\fweight{1}\}$, the integer representation of $\weyl$ is
\[
	\gva
=	\left\{
	\begin{bmatrix} 1 & 0 \\ 0 & 1 \end{bmatrix},
	\begin{bmatrix} -1 & 0 \\ 1 & 1 \end{bmatrix},
	\begin{bmatrix} 1 & 1 \\ 0 & -1 \end{bmatrix},
	\begin{bmatrix} -1 & -1 \\ 1 & 0 \end{bmatrix},
	\begin{bmatrix} 0 & 1 \\ -1 & -1 \end{bmatrix},
	\begin{bmatrix} 0 & -1 \\ -1 & 0 \end{bmatrix}
	\right\}.
\]
Let $z_1,z_2\in\R$ and $z=(z_1+\mathrm{i} z_2,z_1-\mathrm{i} z_2)$. 
The matrix $C\in \RX^{3\times 3}$ from \Cref{HermiteCharacterizationAn} is
\[
C(z)
=
\begin{bmatrix}
0	&	1/2	&	(1+9\,z_1^2-9\,z_2^2)/2				\\
1	&	0	&	(1- 9\,z_1^2-9\,z_2^2-6\,z_1)/2	\\
0	&	1/2	&	3\,z_1
\end{bmatrix}.
\]
Define the matrix $H\in\RX^{3\times 3}$ with entries $(H(z))_{ij} = \trace(C(z)^{i+j-2})-\trace(C(z)^{i+j})$. Then $z \in \Image$, or equivalently $(z_1,z_2)\in\Image_\R$, if and only if $H(z)$ is positive semi--definite. Assume that
\[
\det(x\,I_3-H(z))=x^3-h_1(z)\,x^2+h_2(z)\,x-h_3(z)
\]
is the characteristic polynomial of $H(z)$, where
\[
\begin{array}{rcl}
	\textcolor{red}{h_3(z)}&\textcolor{red}{=}			&	\textcolor{red}{-\coeff(x^0,\det(x\,I_3-H(z)))} \\
	&\textcolor{red}{=}					&	\textcolor{red}{2187/64\,z_2^4\,(3\,z_1 + 1)^2\,(-3\,z_1^4 - 6\,z_1^2\,z_2^2 - 3\,z_2^4 + 8\,z_1^3 - 24\,z_1\,z_2^2 - 6\,z_1^2 - 6\,z_2^2 + 1)} \\
	\textcolor{blue}{h_2(z)}&\textcolor{blue}{=}	&	\phantom{-} \textcolor{blue}{\coeff(x^1,\det(x\,I_3-H(z)))} \\
	&\textcolor{blue}{=}					&	\textcolor{blue}{243/256\,z_2^2\,(-243\,z_1^8 - 972\,z_1^6\,z_2^2 - 1458\,z_1^4\,z_2^4 - 972\,z_1^2\,z_2^6 - 243\,z_2^8 + 324\,z_1^7 - 1620\,z_1^5\,z_2^2 }\\
	&									&	\textcolor{blue}{- 4212\,z_1^3\,z_2^4 - 2268\,z_1\,z_2^6 - 432\,z_1^6 - 2052\,z_1^4\,z_2^2 - 5400\,z_1^2\,z_2^4 - 324\,z_2^6 + 180\,z_1^5 - 3384\,z_1^3\,z_2^2 } \\
	&									&	\textcolor{blue}{- 684\,z_1\,z_2^4 + 18\,z_1^4 - 804\,z_1^2\,z_2^2 + 42\,z_2^4 + 76\,z_1^3 + 404\,z_1\,z_2^2 - 8\,z_1^2 - 108\,z_2^2 + 60\,z_1 + 25)}, \\
	\textcolor{OliveGreen}{h_1(z)}&\textcolor{OliveGreen}{=}	&	\textcolor{OliveGreen}{-\coeff(x^2,\det(x\,I_3-H(z)))} \\
	&\textcolor{OliveGreen}{=}			&	\textcolor{OliveGreen}{1/32\,(-729\,z_1^6 + 1458\,z_1^5 + (10935\,z_2^2 - 1215)\,z_1^4 + (-2916\,z_2^2 + 540)\,z_1^3 + 351\,z_2^2 + 63}\\
	&									&	\textcolor{OliveGreen}{+ (-10935\,z_2^4 + 1458\,z_2^2 - 135)\,z_1^2 + (-4374\,z_2^4 + 972\,z_2^2 + 18)\,z_1 + 729\,z_2^6 - 1215\,z_2^4)}.
\end{array}
\]

\begin{figure}[H]
\begin{center}
	\begin{subfigure}{.4\textwidth}
		\centering
		\includegraphics[width=6cm]{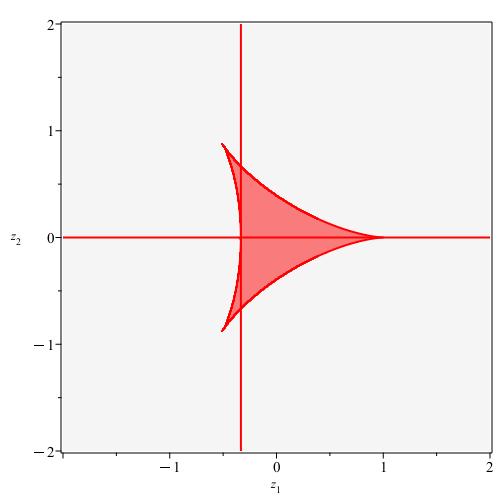}
		\caption{\textcolor{red}{$h_3(z)\geq 0$}}
	\end{subfigure}%
	\begin{subfigure}{.4\textwidth}
		\centering
		\includegraphics[width=6cm]{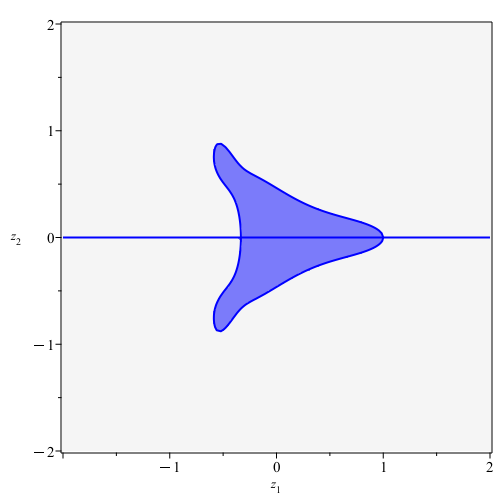}
		\caption{\textcolor{blue}{$h_2(z)\geq 0$}}
	\end{subfigure}
	\begin{subfigure}{.4\textwidth}
		\centering
		\includegraphics[width=6cm]{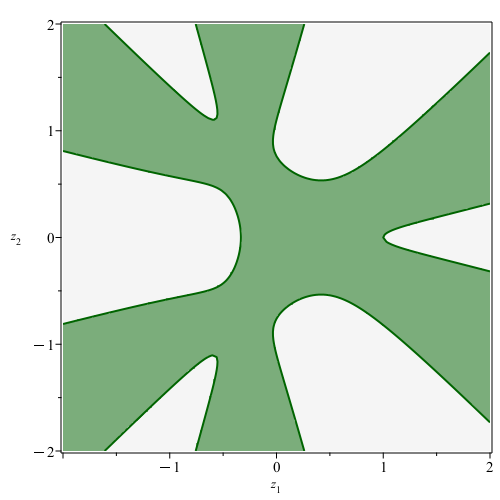}
		\caption{\textcolor{OliveGreen}{$h_1(z)\geq 0$}}
	\end{subfigure}
	\begin{subfigure}{.4\textwidth}
		\centering
		\includegraphics[width=6cm]{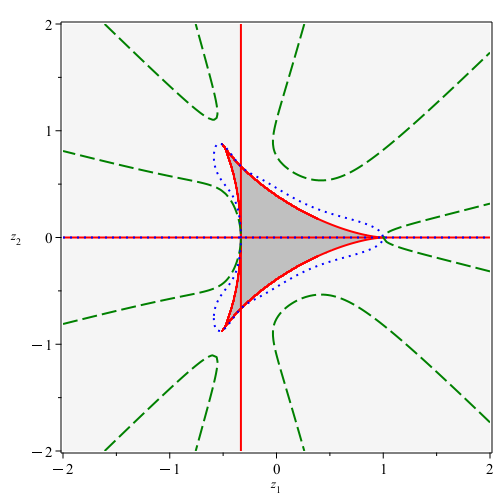}
		\caption{\textcolor{gray}{$h_1(z),h_2(z),h_3(z)\geq 0$}}
	\end{subfigure}
\caption{The intersection of the semi--algebraic sets defined by the coefficients of the characteristic polynomial of $H(z)$ is the $\TTorus$--orbit space associated to $\RootA[2]$.}
\label{example_figureA2}
\end{center}
\end{figure}

The matrix $H(z)$ is positive semi--definite if and only if $h_i(z)\geq 0$ for $1\leq i \leq 3$.
In \Cref{example_figureA2}, a solid red line, blue dots and green dashes indicate the varieties of these three polynomials. 
From those plots, we suspect that the real $\TTorus$--orbit space is invariant under rotation by $2\pi/3$ and horizontal reflection, that is, under the dihedral group $\mathfrak{D}_{3}$ of order $6$. 
The $\mathfrak{D}_{3}$--invariants are $\RX^{\mathfrak{D}_3}=\R[g_1,g_2]$ with $g_1(z):=z_1^2+z_2^2$ and $g_2(z):=z_1\,(z_1^2-3\,z_2^2)$. Hence, we have
\[
	\textcolor{red}{h_3(z)=}
	\underbrace{\textcolor{red}{2187/64\,z_2^4\,(3\,z_1 + 1)^2}}_{\geq 0} \,
	\underbrace{\textcolor{red}{( -6\,g_1(z) -3\,g_1(z)^2 +8\,g_2(z) +1 )}}_{\mathfrak{D}_3\mbox{--invariant}} . 
\]
and the variety of the $\mathfrak{D}_3$--invariant factor is the boundary of $\Image_\R$. 
We now identify the vertices of $\Image_\R$, which correspond to the fundamental weights and the origin. With $\bigcos_\R=\left( \frac{\orb{1}+\orb{2}}{2} , \frac{\orb{1}-\orb{2}}{2\mathrm{i}} \right)$, those are
\begin{align*}
&	\mathrm{Vertex}_1
:=	\bigcos_{\R}(\mexp{\fweight{1},\fweight{1}},\mexp{\fweight{2},\fweight{1}})
=	\bigcos_{\R}\left(\exp\left(-\dfrac{4}{3}\pi\mathrm{i}\right),\exp\left(-\dfrac{2}{3}\pi\mathrm{i}\right)\right)
=	\left(-\dfrac{1}{2},-\dfrac{\sqrt{3}}{2}\right), \\
&	\mathrm{Vertex}_2
:=	\bigcos_{\R}(\mexp{\fweight{1},\fweight{2}},\mexp{\fweight{2},\fweight{2}})
=	\bigcos_{\R}\left(\exp\left(-\dfrac{2}{3}\pi\mathrm{i}\right),\exp\left(-\dfrac{4}{3}\pi\mathrm{i}\right)\right)
=	\left(-\dfrac{1}{2},\phantom{-}\dfrac{\sqrt{3}}{2}\right), \\
&	\mathrm{Vertex}_3
:=	\bigcos_{\R}(\mexp{\fweight{1},0},\mexp{\fweight{2},0})
=	\bigcos_{\R}(1,1)
=	(1,0) 
	\, .  \phantom{\dfrac{\sqrt{1}}{3}}
\end{align*}

For a generic point $(z_1,z_2)\in \R^2$, $H(z)$ has rank $3$. 
The variety ``$h_1=h_2=h_3=0$'' corresponds to rank $0$ Hermite matrices $H(z)$ and contains $\mathrm{Vertex}_3$ as well as
\[
	\bigcos_{\R}(\exp(-2\pi\mathrm{i}\langle \fweight{1},(\fweight{1}+\fweight{2})/2 \rangle),\exp(-2\pi\mathrm{i}\langle \fweight{2},(\fweight{1}+\fweight{2})/2 \rangle))
=	\bigcos_{\R}(\exp(-\pi\mathrm{i}),\exp(\pi\mathrm{i}))
=	\left(-\dfrac{1}{3},0\right) \, .
\]
The varieties ``$h_2=h_3=0$'', containing $\mathrm{Vertex}_1$ and $\mathrm{Vertex}_2$, and ``$h_1=h_3=0$'', containing $(-1/3,\pm 2/3)$, determine rank $1$ Hermite matrices. 
Every other point on the boundary of $\Image_\R$ corresponds to a rank $2$ Hermite matrix.


\section{Type $\RootC$}
\label{section_proofC}
\setcounter{equation}{0}

\begin{figure}[H]
	\begin{center}
		\begin{overpic}[width=0.3\textwidth,grid=false,tics=10]{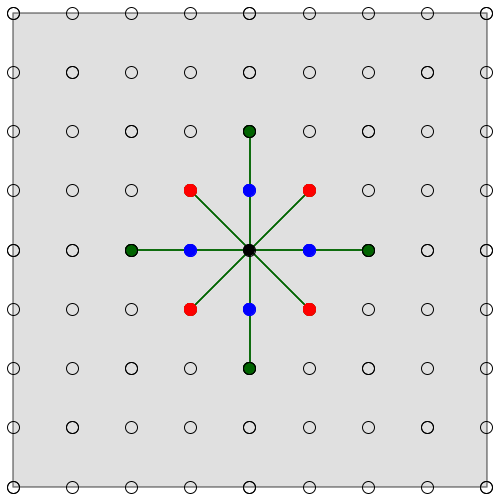}
			\put (63,52) {\large \textcolor{blue}{$\displaystyle \fweight{1}$}}
			\put (63,65) {\large \textcolor{red}{$\displaystyle \fweight{2}$}}
		\end{overpic}
		\quad
		\begin{overpic}[width=0.59\textwidth,grid=false,tics=10]{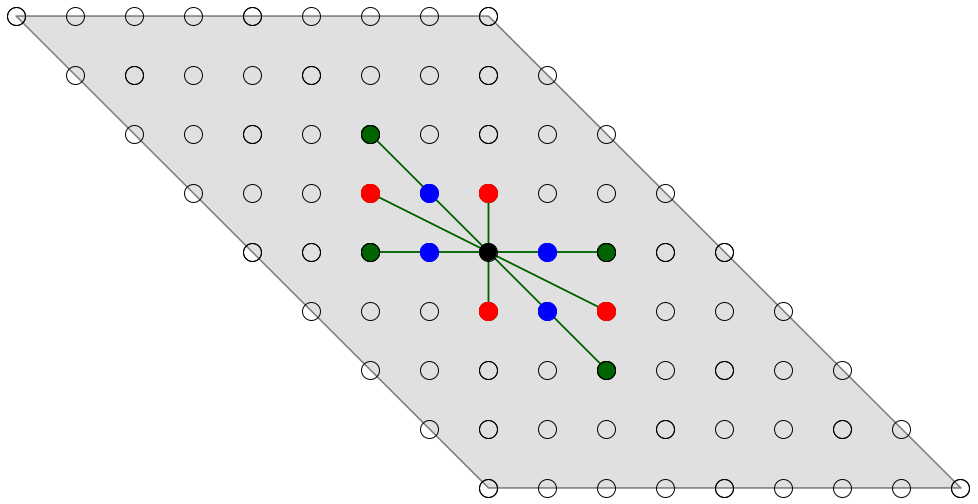}
			\put (57,27) {\large \textcolor{blue}{$\displaystyle e_{1}$}}
			\put (51,33) {\large \textcolor{red}{$\displaystyle e_{2}$}}
		\end{overpic}
	\end{center}
	\caption{The \textcolor{OliveGreen}{root system} $\RootC[2]$ and its weight lattice in the usual orthogonal representation and the integer representation. The orbits of the fundamental weights are the \textcolor{blue}{blue} and \textcolor{red}{red} lattice elements.}
	\label{example_rootsystemC2}
\end{figure}


The groups $\mathfrak{S}_{n}$ and $\{\pm 1\}^n$ act on $\R^n$ by permutation of coordinates and multiplication of coordinates by $\pm 1$. The root system $\RootC$ given in \cite[Planche III]{bourbaki456} is a root system in $\R^n$ with fundamental weights
\begin{equation}\label{equation_WeightsRootsC}
\fweight{i}=e_1 + \ldots + e_i \quad(1\leq i\leq n).
\end{equation}
The weight lattice is $\Weights=\Z^n$. 
The Weyl group $\weyl \cong \mathfrak{S}_{n}\ltimes\{\pm 1\}^n$ of $\RootC$ is orthogonal with respect to the Euclidean scalar product. 
We have $-I_n\in\weyl$ and thus $-\fweight{i}\in\weyl\,\fweight{i}$. 
Furthermore, the orbit $\weyl\,\fweight{i}$ has cardinality $2^i\,\binom{n}{i}$ for $1\leq i \leq n$, see \Cref{example_rootsystemC2}. 

In this section, we give a closed formula for the matrix polynomial from \Cref{MainThmIntro} in the standard monomial basis for $\RootC$. This is a root system of rank $n$. Hence, the ring of Laurent polynomials is $\Rx=\Q[x_1,x_1^{-1},\ldots, x_{n},x_{n}^{-1}]$ and the polynomial ring is $\RX=\Q[z_1,\ldots,z_{n}]$.

\subsection{Orbit polynomials}

We denote by $\gva\in\mathrm{GL}_{n}(\Z)$ the integer representation of $\weyl$ with respect to the fundamental weights in \Cref{equation_WeightsRootsC}. Then the orbit $\gva \cdot x_1 = \{ x^{B e_1} \,\vert\, B\in \gva \} \subseteq \Rx$ consists of $2n$ distinct monomials, namely
\begin{equation}\label{eq_monomialsC}
	y_1:=x_1 ,\quad y_2:=x_2\,x_{1}^{-1} ,\quad \ldots ,\quad y_n:=x_{n}\,x_{n-1}^{-1}
\end{equation}
and their inverses. For $1\leq i\leq n$, let $\sigma_i$ be the $i$--th elementary symmetric function in $n$ indeterminates and recall that $\orb{i}=\bigorb{e_i}$ is the $\gva$--invariant orbit polynomial associated to $e_i\in\Z^{n}$ from \Cref{defi_orbitpoly}.

\begin{proposition}\label{OrbitSumsCn}
In $\Qx$, we have
\[
\sigma_i\left(\frac{y_1(x)+y_1(x)^{-1}}{2},\ldots,\frac{y_n(x)+y_n(x)^{-1}}{2}\right) = \binom{ n }{ i } \, \orb{i}(x)\quad(1\leq i\leq n).
\]
\end{proposition}
\begin{proof}
It follows from \Cref{eq_monomialsC} that $x_i = y_1(x) \ldots y_i(x)$ and $\gva$ acts on the $y_i^{\pm 1}(x)$ by permutation. Hence,
\[
	\orb{i}(x)
=	\frac{1}{\vert \gva \vert}\,\sum\limits_{B\in\gva} x^{B\,e_i}
=	\frac{\vert\mathrm{Stab}_{\gva}(x_i)\vert}{\vert \gva \vert}\,\sum\limits_{\substack{J \subseteq \{1 ,\ldots, n\} \\ \vert J \vert = i }} \, \sum\limits_{\delta\in\{\pm 1\}^J} \, \prod_{j\in J} y_j(x)^{\delta_j}
=	\frac{1}{\nops{\gva\cdot x_i}} \, \sum\limits_{\substack{J \subseteq \{1 ,\ldots, n\} \\ \vert J \vert = i }} \, \prod_{j\in J} ( y_j(x)+y_j(x)^{-1} )
\]
With $\nops{\gva\cdot x_i}=2^i\,\binom{n}{i}$, we obtain the statement.
\end{proof}

The next statement follows immediately from \Cref{eq_monomialsC}. 

\begin{lemma}\label{proposition_PsiC}
The map
\[
\Upsilon:
\begin{array}[t]{ccl}
	\TTorus^{n}	&	\to 	& \TTorus^n,\\
	x 			&	\mapsto	& 	( y_1(x),\ldots,y_{n}(x) ),
\end{array}
\]
is bijective. 
\end{lemma}

\subsection{Hermite characterization with standard monomials}

We now characterize, whether a given point $z$ is contained in the $\TTorus$--orbit space $\Image$ of $\gva$, that is, 
if the equation $\orb{i}(x)=z_i$ has a solution $x\in\TTorus^n$. 
As shown in previous subsections, this is equivalent to deciding whether a symmetric polynomial system of type $\mathrm{(II)}$ has its solution in $\TTorus^n$. 
We state our main result for $\RootC$ in the standard monomial basis.

\begin{theorem}\label{HermiteCharacterizationCn}
Define the matrix $H\in\RX^{n\times n}$ by
\begin{align*}
H(z)_{ij}	&=	\trace(C(z)^{i+j-2}-C(z)^{i+j}),
\tbox{where}
C(z)		=
\begin{bmatrix}
0		&	\cdots	&	0		&	-c_n(z)			\\
1		&			&	0		&	-c_{n-1}(z)		\\
		&	\ddots	&			&	\vdots			\\
0		&			&	1		&	-c_1(z)
\end{bmatrix},\\
c_i(z)		&=	(-1)^i\,\binom{n}{i}\,z_i \tbox{for} 1\leq i\leq n.
\end{align*}
Then $\Image=\Image_\R=\{ z\in \R^n \,\vert\, H(z)\succeq 0 \}$.
\end{theorem}
\begin{proof}
Let $z\in\R^n$ and set $c_i:=c_i(z)\in\R$ for $1\leq i \leq n$.

To show ``$\subseteq$'', assume that $z\in\Image$. Then there exists $ x\in\TTorus^{n}$, such that $\orb{i}( x)=z_i$ for $1\leq i \leq n$. By \Cref{OrbitSumsCn} and \Cref{proposition_PsiC}, the solution of the symmetric polynomial system
\[
\mathrm{(II)}  \quad  \sigma_i\left(\frac{y_1+y_1^{-1}}{2}	,\ldots,\frac{y_n+y_n^{-1}}{2}	\right) = (-1)^i\,c_i \quad (1\leq i\leq n)
\]
is $y=\Upsilon( x)\in\TTorus^n$. 
Applying \Cref{Corollary_SolutionsSystemII} yields $H(z)\succeq 0$.

For ``$\supseteq$'' on the other hand, assume $H(z)\succeq 0$. By \Cref{Corollary_SolutionsSystemII}, the solution $y$ of the above system $(\mathrm{II})$ is contained in $\TTorus^n$. Let $ x\in\TTorus^{n}$ be the unique preimage of $ y$ under $\Upsilon$. Then $z_i=\orb{i}( x)$ and so $z=\bigcos(x)$ is contained in $\Image$.
\end{proof}

\subsection*{Example: $\RootC[2]$}

We study the root system from \Cref{example_rootsystemC2} with Weyl group $\weyl\cong\mathfrak{S}_2 \ltimes \{\pm 1\}^2$. 
Since $\weyl\fweight{1}=\{\fweight{1},-\fweight{1},\fweight{1}-\fweight{2},\fweight{2}-\fweight{1}\}$ and $\weyl\fweight{2}=\{\fweight{2},-\fweight{2},2\fweight{1}-\fweight{1},\fweight{2}-2\fweight{1}\}$, the integer representation of $\weyl$ is
\[
\gva
=	\left\{
\begin{bmatrix} 1 & 0 \\ 0 & 1 \end{bmatrix},
\begin{bmatrix} -1 & 0 \\ 1 & 1 \end{bmatrix},
\begin{bmatrix} 1 & 2 \\ 0 & -1 \end{bmatrix},
\begin{bmatrix} 1 & 2 \\ -1 & -1 \end{bmatrix},
\begin{bmatrix} -1 & -2 \\ 1 & 1 \end{bmatrix},
\begin{bmatrix} -1 & -2 \\ 0 &  1 \end{bmatrix},
\begin{bmatrix} 1 & 0 \\ -1 & -1 \end{bmatrix},
\begin{bmatrix} -1 & 0 \\ 0 & -1 \end{bmatrix}
\right\}.
\]
Let $z=(z_1,z_2) \in\R^2$. The matrix $C\in \RX^{2\times 2}$ from \Cref{HermiteCharacterizationCn} is
\[
C(z)=
\begin{bmatrix}
0	&	-z_2 \\
1	&	2 \,z_1 \\
\end{bmatrix}.
\]
Then $z$ is contained in $\Image$ if and only if the resulting Hermite matrix
\[
H(z)=
\begin{bmatrix}
-4\, z_1^2 +  2\,z_2 + 2 & -8 \,z_1^3 + 6\, z_1 \,z_2 + 2 \,z_1 \\
-8 \,z_1^3 + 6\, z_1 \,z_2 + 2 \,z_1 & -16\,z_1^4 + 16\,z_1^2\,z_2 + 4\,z_1^2 - 2\,z_2^2 - 2\,z_2
\end{bmatrix}
\]
is positive semi--definite, which is equivalent to its determinant and trace
\[
\begin{array}{rcl}
	\textcolor{red}{ \det (H(z))}		&	\textcolor{red}{ =}	&	\textcolor{red}{-4\,(z_1^2 - z_2)\,(2\,z_1 + 1 + z_2)\,(2\,z_1 - 1 - z_2)} \\
	\textcolor{blue}{ \trace (H(z))}	&	\textcolor{blue}{=}	&	\textcolor{blue}{-16\,z_1^4 + 16\,z_1^2\,z_2 - 2\,z_2^2 + 2}
\end{array}
\]
being nonnegative. The varieties of these two polynomials in $z_1,z_2$ are depicted in \Cref{example_figureC2}.
\begin{figure}[H]
\begin{center}
	\begin{subfigure}{.3\textwidth}
		\centering
		\includegraphics[width=4.7cm]{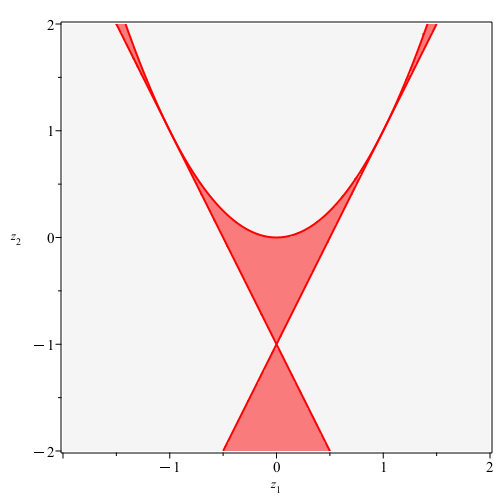}
		\caption{\textcolor{red}{$\det (H(z))\geq 0$}}
	\end{subfigure}%
	\begin{subfigure}{.3\textwidth}
		\centering
		\includegraphics[width=4.7cm]{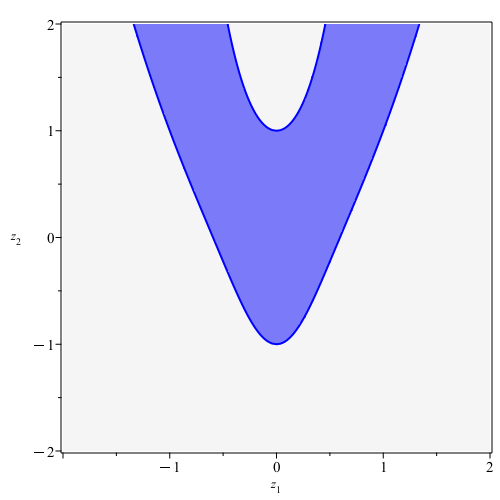}
		\caption{\textcolor{blue}{$\trace (H(z))\geq 0$}}
	\end{subfigure}
	\begin{subfigure}{.3\textwidth}
		\centering
		\includegraphics[width=4.7cm]{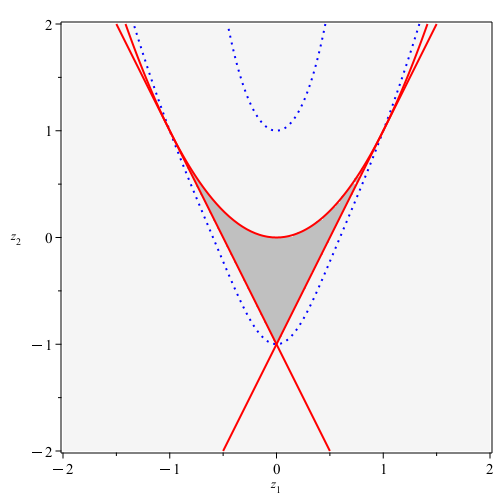}
		\caption{\textcolor{gray}{$\det (H(z)),\trace (H(z))\geq 0$}}
	\end{subfigure}
\caption{The semi--algebraic sets defined by the determinant and trace of $H(z)$.}
\label{example_figureC2}
\end{center}
\end{figure}
We observe three intersections of ``$\det (H(z)) =0$'' (red solid line) and ``$\trace (H(z)) =0$'' (blue dots) in the the vertices
\begin{align*}
&	\mathrm{Vertex}_1
:=	\bigcos(\mexp{\fweight{1},\fweight{1}/2},\mexp{\fweight{2},\fweight{1}/2})=\bigcos(-1,-1)
=	(0,-1), \\
&	\mathrm{Vertex}_2
:=	\bigcos(\mexp{\fweight{1},\fweight{2}/2},\mexp{\fweight{2},\fweight{2}/2})
=	\bigcos(-1,1)
=	(-1,1), \\
&	\mathrm{Vertex}_3
:=	\bigcos(\mexp{\fweight{1},0},\mexp{\fweight{2},0})
=	\bigcos(1,1)
=	(1,1).
\end{align*}
The shape of this domain is dictated by the determinant, but from the positivity condition one can observe that the trace is also required. Alternatively, the inequality given by the trace could be replaced by the constraint that the orbit space is contained in the square $[-1,1]^2$.

\section{Type $\RootB$}
\label{section_proofB}
\setcounter{equation}{0}

\begin{figure}[H]
	\begin{center}
		\begin{overpic}[width=0.45\textwidth,grid=false,tics=10]{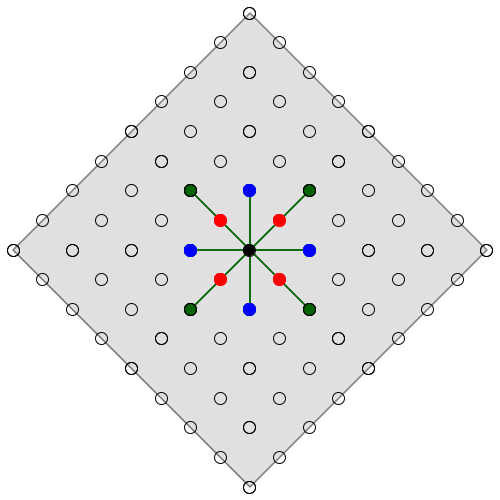}
			\put (60,53) {\large \textcolor{blue}{$\displaystyle \fweight{1}$}}
			\put (53,60) {\large \textcolor{red}{$\displaystyle \fweight{2}$}}
		\end{overpic}
		\quad
		\begin{overpic}[width=0.45\textwidth,grid=false,tics=10]{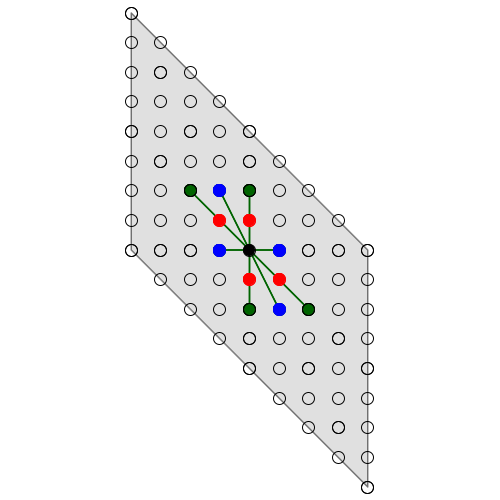}
			\put (56,52) {\large \textcolor{blue}{$\displaystyle e_{1}$}}
			\put (50,58) {\large \textcolor{red}{$\displaystyle e_{2}$}}
		\end{overpic}
	\end{center}
	\caption{The \textcolor{OliveGreen}{root system} $\RootB[2]$ and its weight lattice in the usual orthogonal representation and the integer representation. The orbits of the fundamental weights are the \textcolor{blue}{blue} and \textcolor{red}{red} lattice elements.}
	\label{example_rootsystemB2}
\end{figure}


The root system $\RootB$ given in \cite[Planche II]{bourbaki456} is a root system in $\R^n$. Its Weyl group is isomorphic to that of $\RootC$, that is $\weyl\cong \mathfrak{S}_{n}\ltimes\{\pm 1\}^n$. 
However, the fundamental weights are now
\begin{equation}\label{equation_WeightsRootsB}
\fweight{i}=e_1 + \ldots + e_i\quad(1\leq i\leq n-1) \tbox{and} \fweight{n}=(e_1+\ldots+e_n)/2.
\end{equation}
In particular, the $\TTorus$--orbit space can be very different from $\RootC$, see \Cref{fig_OrbitSpaceIntro}. 
We have $-I_n\in\weyl$ and thus $-\fweight{i}\in\weyl\,\fweight{i}$. Furthermore, the orbit $\weyl\,\fweight{i}$ has cardinality $2^i\,\binom{n}{i}$ for $1\leq i \leq n$, see \Cref{example_rootsystemB2}. 

In this section, we give a closed formula for the matrix polynomial from \Cref{MainThmIntro} in the standard monomial basis for $\RootB$. This is a root system of rank $n$. Hence, the ring of Laurent polynomials is $\Rx=\Q[x_1,x_1^{-1},\ldots, x_{n},x_{n}^{-1}]$ and the polynomial ring is $\RX=\Q[z_1,\ldots,z_{n}]$.

\subsection{Orbit polynomials}

We denote by $\gva\in\mathrm{GL}_{n}(\Z)$ the integer representation of $\weyl$ with respect to the fundamental weights in \Cref{equation_WeightsRootsB}. Then the orbit $\gva \cdot x_1 = \{ x^{B e_1} \,\vert\, B\in \gva \} \subseteq \Rx$ consists of $2n$ distinct monomials, namely
\begin{equation}\label{eq_monomialsB}
	y_1:=x_1 ,\quad y_2:=x_2\,x_{1}^{-1} ,\quad \ldots ,\quad y_{n-1}:=x_{n-1}\,x_{n-2}^{-1} ,\quad  y_n:=x_{n}^{2}\,x_{n-1}^{-1}
\end{equation}
and their inverses.  For $1\leq i\leq n$, let $\sigma_i$ be the $i$--th elementary symmetric function in $n$ indeterminates and recall that $\orb{i}=\bigorb{e_i}$ is the $\gva$--invariant orbit polynomial associated to $e_i\in\Z^{n}$ from \Cref{defi_orbitpoly}.

\begin{proposition}\label{OrbitSumsBn}
In $\Qx$, we have
\begin{align*}
\sigma_i\left(\frac{y_1(x)+y_1(x)^{-1}}{2},\ldots,\frac{y_n(x)+y_n(x)^{-1}}{2}\right)		=&\,	\binom{ n}{ i } \, \orb{i}(x)\quad(1\leq i\leq n-1)	 \\
\tbox{and}
\sigma_{n}\left(\frac{y_1(x)+y_1(x)^{-1}}{2},\ldots,\frac{y_n(x)+y_n(x)^{-1}}{2}\right)	=&\,	\bigorb{2\,e_{n}}(x).
\end{align*}
\end{proposition}
\begin{proof}
It follows from \Cref{eq_monomialsB} that $x_i = y_1(x) \ldots y_i(x)$ and $x_n^2=y_1(x)\ldots y_n(x)$. Then the proof is analogous to \Cref{OrbitSumsCn}.
\end{proof}

We compute the explicit expression for the right hand side of \Cref{OrbitSumsBn} in terms of the fundamental invariants $\orb{i}$.

\begin{lemma}\label{OrbitSumsBn2}
In $\Qx$, we have
\[
\bigorb{2\,e_{n}}	=	\,	2^n\,\orb{n}^2 - \sum_{j=1}^{n-1}\binom{ n }{ j }\,\orb{j} -1.
\]
\end{lemma}
\begin{proof}
The cardinality of the orbit $\gva\, e_{n}$ is $2^n$. Let $\alpha\in\gva \, e_{n}$ and distinguish between the following cases.
\begin{enumerate}
\item If $\alpha = e_{n}$, then $\bigorb{e_{n}+\alpha} = \bigorb{2\,e_{n}}$ is the term for which we seek an explicit formula.
\item If $\alpha=-e_{n}$, then $\bigorb{e_{n}+\alpha} = \bigorb{0} = 1$.
\item Otherwise, it follows from \Cref{equation_WeightsRootsB} that there exists $1 \leq j \leq n-1$, such that $e_n+\alpha\in\gva\,e_j$. The number of $\alpha$, for which this is the case, is $\binom{n}{j}$.
\end{enumerate}
With the recurrence formula from \Cref{proposition_RecurrenceOrbitPolynomials}, we conclude
\[
	2^n \,\orb{n}^2
=	\nops{\gva\,e_n}\,\bigorb{e_n}\,\bigorb{e_n}
=	\sum_{\alpha\in\gva \, e_{n}} \, \bigorb{\alpha+e_n}
=	\bigorb{2\,e_{n}} +\sum_{j=1}^{n-1} \binom{ n }{ j  }\,\bigorb{e_j} + 1
=	\bigorb{2\,e_{n}} +\sum_{j=1}^{n-1} \binom{ n }{ j  }\,\orb{j} + 1
\]
and obtain the formula for $\bigorb{2 \,e_{n}}$.
\end{proof}

\begin{lemma}\label{proposition_PsiB}
The map
\[
\Upsilon:
\begin{array}[t]{ccl}
	\TTorus^{n}	&	\to 	& \TTorus^n,\\
	x 			&	\mapsto	& 	( y_1(x),\ldots,y_{n}(x) ),
\end{array}
\]
is surjective. 
Furthermore, every $y\in\TTorus^n$ has exactly two distinct preimages $x,x'\in\TTorus^n$ with
\[
	\orb{i}(x)=\orb{i}(x')\quad(1\leq i\leq n-1)
	\tbox{and}
	\orb{n}(x)=-\orb{n}(x').
\]
\end{lemma}
\begin{proof}
For $y\in\TTorus^n$, choose $x\in\TTorus^n$ with $ x_1= y_1, x_2= y_1\, x_1,\ldots, x_{n-1}= y_{n-1}\, x_{n-2}$ and $ x_n^2= y_n\, x_{n-1}$. 
Then $x$ is a preimage of $y$ under $\Upsilon$ and unique up to a sign in the last coordinate.

For $1\leq i\leq n-1$, we have $\orb{i}(x)=\orb{i}(x')$, because $y_k(x)=y_k(x')$ for all $1\leq k \leq n$.

It follows from \Cref{equation_WeightsRootsB} that the stabilizer of $\fweight{n}$ in $\weyl$ is $\mathfrak{S}_n$. 
Hence, we have $\weyl\, \fweight{n} = \{\pm 1\}^{n}\,\fweight{n}$ and, for all $\weight{}\in \weyl\, \fweight{n}$, there exist $\epsilon_i = \pm 1$ and $\nu\in\Z\,\fweight{1}\oplus\ldots\oplus\Z\,\fweight{n-1}$, such that
\[
	\weight{}=\dfrac{\epsilon_1}{2}\,e_1+\ldots+\dfrac{\epsilon_n}{2}\,e_n = \dfrac{\epsilon_1}{2}\,\fweight{1}+\sum\limits_{i=2}^{n-1} \dfrac{\epsilon_i}{2}\, (\fweight{i}-\fweight{i-1}) + \dfrac{\epsilon_n}{2}\,(2\,\fweight{n}-\fweight{n-1})=\epsilon_n\,\fweight{n} +\nu.
\]
Now let $\alpha$, $\beta\in\Z^n$, such that $\weight{}=\alpha_1\,\fweight{1}+\ldots+\alpha_n\,\fweight{n}$ and $\nu=\beta_1\,\fweight{1}+\ldots+\beta_n\,\fweight{n}$. 
Then $\beta_n=0$ and the monomial in $\orb{n}$ corresponding to $\weight{}$ is $x^{\alpha}	=	x_1^{\beta_1} \ldots x_{n-1}^{\beta_{n-1}}\,x_n^{\epsilon_n}$. Thus, $x^{\alpha}$ is linear in $x_n$. Since every monomial in $\orb{n}$ can be written in terms of such $\beta$ and $\epsilon_i$, $\orb{n}$ is linear in $x_n$ and with $x,x'$ as above we have $\orb{n}(x)=-\orb{n}(x')$.
\end{proof}

\subsection{Hermite characterization with standard monomials}

We now characterize, whether a given point $z$ is contained in the $\TTorus$--orbit space $\Image$ of $\gva$, that is, 
if the equation $\orb{i}(x)=z_i$ has a solution $x\in\TTorus^n$. 
As shown in previous subsections, this is equivalent to deciding whether a symmetric polynomial system of type $\mathrm{(II)}$ has its solution in $\TTorus^n$. 
We state our main result for $\RootB$ in the standard monomial basis.

\begin{theorem}\label{HermiteCharacterizationBn}
Define the matrix $H\in\RX^{n\times n}$ by
\begin{align*}
H(z)_{ij}	&=	\trace(C(z)^{i+j-2}-C(z)^{i+j}),
\tbox{where}
C(z)		=
\begin{bmatrix}
0		&	\cdots	&	0		&	-c_n(z)		\\
1		&			&	0		&	-c_{n-1}(z)	\\
		&	\ddots	&			&	\vdots		\\
0		&			&	1		&	-c_1(z)
\end{bmatrix},\\
c_i(z)		&=	(-1)^i \,\binom{ n }{ i } \, z_i \quad (1\leq i\leq n-1) \tbox{and}
c_n(z)		=	(-1)^n \left( 2^n\,z_n^2- \sum\limits_{i=1}^{n-1} \binom{n}{i}\, z_i - 1 \right) .
\end{align*}
Then $\Image=\Image_\R=\{ z\in \R^n \,\vert\, H(z)\succeq 0 \}$.
\end{theorem}
\begin{proof}
Let $z\in\R^n$ and set $c_i:=c_i(z)\in\R$ for $1\leq i \leq n$.

To show ``$\subseteq$'', assume that $z\in\Image$. 
Then there exists $ x\in\TTorus^{n}$, such that $\orb{i}(x)=z_i$ for $1\leq i\leq n$. 
By \Cref{OrbitSumsBn} and \Cref{OrbitSumsBn2,proposition_PsiB}, the solution of the symmetric polynomial system
\[
\mathrm{(II)}  \quad  \sigma_i\left(\frac{y_1+y_1^{-1}}{2}	,\ldots,\frac{y_n+y_n^{-1}}{2}	\right) = (-1)^i\,c_i \tbox{for} 1\leq i\leq n
\]
is $y=\Upsilon( x)\in\TTorus^n$. 
Applying \Cref{Corollary_SolutionsSystemII} yields $H(z)\succeq 0$.

For ``$\supseteq$'' on the other hand, assume $H(z)\succeq 0$. By \Cref{Corollary_SolutionsSystemII}, the solution $y$ of system $(\mathrm{II})$ with coefficients $c_i$ is contained in $\TTorus^n$. According to \Cref{proposition_PsiB}, $y$ has exactly two distinct preimages $ x, x'\in\TTorus^n$ under $\Upsilon$ with $x_1=x_1',\ldots,x_{n-1}=x_{n-1}'$ and $x_n = -x_n'$. 
We have $z_i=\orb{i}( x)=\orb{i}( x')$ for $1\leq i\leq n-1$ and $z_n^2=\orb{n}( x)^2=\orb{n}( x')^2$ with $\orb{n}( x)=-\orb{n}( x')$. 
Therefore, $z_n=\orb{n}( x)$ or $z_n=-\orb{n}( x)=\orb{n}( x')$ and thus, $z$ is contained in $\Image$.
\end{proof}

\subsection*{Example: $\RootB[2]$}

We study the root system from \Cref{example_rootsystemB2} with Weyl group $\weyl\cong\mathfrak{S}_2 \ltimes \{\pm 1\}^2$. 
Since $\weyl\fweight{1}=\{\fweight{1},-\fweight{1},\fweight{1}-2\fweight{2},2\fweight{2}-\fweight{1}\}$ and $\weyl\fweight{2}=\{\fweight{2},-\fweight{2},\fweight{1}-\fweight{1},\fweight{2}-\fweight{1}\}$, the integer representation of $\weyl$ is
\[
\gva
=	\left\{
\begin{bmatrix} 1 & 0 \\ 0 & 1 \end{bmatrix},
\begin{bmatrix} 1 & 1 \\ 0 & -1 \end{bmatrix},
\begin{bmatrix} 1 & 0 \\ -2 & -1 \end{bmatrix},
\begin{bmatrix} 1 & 1 \\ -2 & -1 \end{bmatrix},
\begin{bmatrix} -1 & -1 \\ 2 & 1 \end{bmatrix},
\begin{bmatrix} -1 & 0 \\ 2 & 1 \end{bmatrix},
\begin{bmatrix} -1 & -1 \\ 0 & 1 \end{bmatrix},
\begin{bmatrix} -1 & 0 \\ 0 & -1 \end{bmatrix}
\right\}.
\]
Let $z=(z_1,z_2) \in\R^2$. 
The matrix $C\in \RX^{2\times 2}$ from \Cref{HermiteCharacterizationBn} is
\[
C(z)=
\begin{bmatrix}
0	&	-4\,z_2^2+2\,z_1+1 \\
1	&	2\,z_1 \\
\end{bmatrix}.
\]
Then $z$ is contained in $\Image$ if and only if the resulting Hermite matrix
\[
H(z)=16
\begin{bmatrix}
-4\,z_1^2 + 8\,z_2^2 - 4\,z_1						&	-8\,z_1^3 + 24\,z_1\,z_2^2 - 12\,z_1^2 - 4\,z_1		\\
-8\,z_1^3 + 24\,z_1\,z_2^2 - 12\,z_1^2 - 4\,z_1		&	-16\,z_1^4 + 64\,z_1^2\,z_2^2 - 32\,z_2^4 - 32\,z_1^3 + 32\,z_1\,z_2^2 - 20\,z_1^2 + 8\,z_2^2 - 4\,z_1
\end{bmatrix}
\]
is positive semi--definite, which is equivalent to its determinant and trace 
\[
\begin{array}{rcl}
	\textcolor{red}{ \det (H(z))}		&\textcolor{red}{ =}	&	 \textcolor{red}{-64\,z_2^2\,(-z_2^2 + z_1)\,(z_1 + 1 + 2\,z_2)\,(z_1 + 1 - 2\,z_2)}	\\
	\textcolor{blue}{ \trace (H(z))}	&\textcolor{blue}{ =}	&	 \textcolor{blue}{-16\,z_1^4 + 64\,z_1^2\,z_2^2 - 32\,z_2^4 - 32\,z_1^3 + 32\,z_1\,z_2^2 - 24\,z_1^2 + 16\,z_2^2 - 8\,z_1}
\end{array}
\]
being nonnegative. The varieties of these two polynomials in $z_1, z_2$ are depicted below.
\begin{figure}[H]
\begin{center}
	\begin{subfigure}{.3\textwidth}
		\centering
		\includegraphics[width=4.7cm]{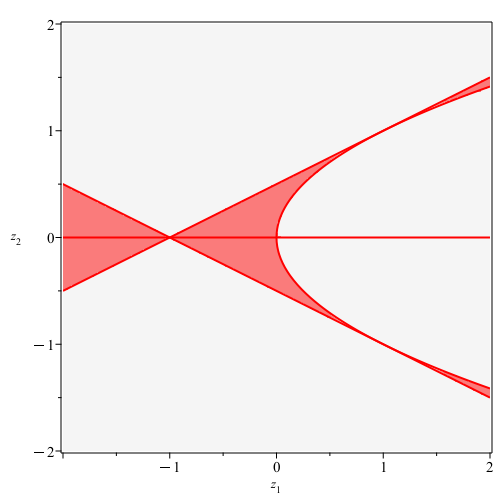}
		\caption{\textcolor{red}{$\det (H(z))\geq 0$}}
	\end{subfigure}%
	\begin{subfigure}{.3\textwidth}
		\centering
		\includegraphics[width=4.7cm]{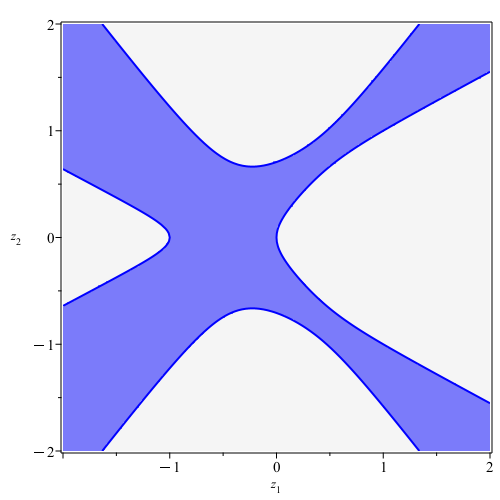}
		\caption{\textcolor{blue}{$\trace (H(z))\geq 0$}}
	\end{subfigure}
	\begin{subfigure}{.3\textwidth}
		\centering
		\includegraphics[width=4.7cm]{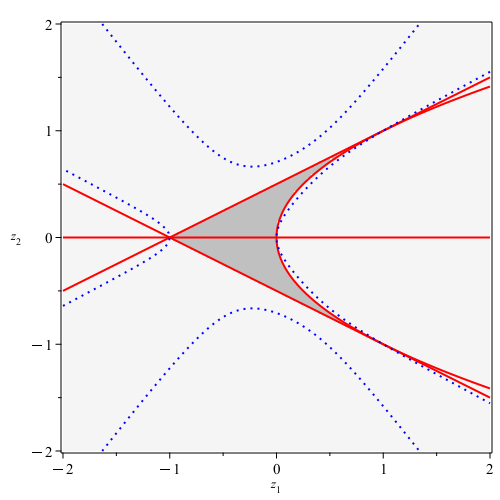}
		\caption{\textcolor{gray}{$\det (H(z)),\trace (H(z))\geq 0$}}
	\end{subfigure}
\caption{The semi--algebraic sets defined by the determinant and trace of $H(z)$.}
\label{example_figureB2}
\end{center}
\end{figure}

The $\TTorus$--orbit space in the $\RootB[2]$--case is obtained from the $\RootC[2]$--case in \Cref{example_figureC2} by permuting $z_1$ and $z_2$. 
In fact, this follows immediately from their Dynking graphs, which are the same up to permutation.
Let us confirm that it follows also from our semi--algebraic description: 
Indeed, we have $\det(H(z))/\det (H^{\RootC[2]}(z_2,z_1))=16\,z_2^2\geq 0$ on $\R^2$. 
Furthermore, the determinant of $H(z)$ is $0$ on the line ``$z_2=0$'' and the rank is $0$ in $z=(0,0)$ (although this is not a vertex).

\section{Type $\RootD$}
\label{section_proofD}
\setcounter{equation}{0}


The group $\{\pm 1\}^n_+:=\{\epsilon\in\{\pm 1\}^n\,\vert\,\epsilon_1\ldots\epsilon_n=1\}$ acts on $\R^n$ by multiplication of coordinates with $\pm 1$, where only an even amount of sign changes is admissible. The root system $\RootD$ given in \cite[Planche IV]{bourbaki456} is a root system in $\R^n$ with fundamental weights
\begin{equation}\label{equation_WeightsRootsD}
\fweight{i} = e_1 + \ldots + e_i \quad(1\leq i\leq n-2) \tbox{and} \fweight{n-1}=(e_1+\ldots +e_{n-1}-e_n)/2,\quad \fweight{n}=(e_1+\ldots + e_n)/2.
\end{equation}
The Weyl group of $\RootD$ is $\weyl \cong \mathfrak{S}_{n}\ltimes\{\pm 1\}^n_+$. For all $1\leq i \leq n$, we have $-\fweight{i}\in\weyl\,\fweight{i}$, except when $n$ is odd, where $-\fweight{n-1}\in\weyl\,\fweight{n}$. Furthermore, we have $\vert \weyl\,\fweight{i} \vert = 2^i\,\binom{n}{i}$ for $1\leq i \leq n-2$ and $\nops{\weyl\,\fweight{n-1}}=\nops{\weyl\,\fweight{n}}=2^{n-1}$.

In this section, we give a closed formula for the matrix polynomial from \Cref{MainThmIntro} in the standard monomial basis for $\RootD$. This is a root system of rank $n$. Hence, the ring of Laurent polynomials is $\Rx=\Q[x_1,x_1^{-1},\ldots, x_{n},x_{n}^{-1}]$ and the polynomial ring is $\RX=\Q[z_1,\ldots,z_{n}]$.

\subsection{Orbit polynomials}

We denote by $\gva\in\mathrm{GL}_{n}(\Z)$ the integer representation of $\weyl$ with respect to the fundamental weights in \Cref{equation_WeightsRootsD}. Then the orbit $\gva \cdot x_1 = \{ x^{B e_1} \,\vert\, B\in \gva \} \subseteq \Rx$ consists of $2n$ distinct monomials, namely
\begin{equation}\label{eq_monomialsD}
	y_1:=x_1 ,\quad y_2:=x_2\,x_{1}^{-1} ,\quad \ldots ,\quad y_{n-2}:=x_{n-2}\,x_{n-3}^{-1} ,\quad y_{n-1}:=x_n\,x_{n-1}\,x_{n-2}^{-1} ,\quad y_n:=x_n\,x_{n-1}^{-1}
\end{equation}
and their inverses.  For $1\leq i\leq n$, let $\sigma_i$ be the $i$--th elementary symmetric function in $n$ indeterminates and recall that $\orb{i}=\bigorb{e_i}$ is the $\gva$--invariant orbit polynomial associated to $e_i\in\Z^{n}$ from \Cref{defi_orbitpoly}.

\begin{proposition}\label{OrbitSumsDn}
In $\Qx$, we have
\begin{align*}
\sigma_i\left(\frac{y_1(x)+y_1(x)^{-1}}{2},\ldots,\frac{y_n(x)+y_n(x)^{-1}}{2}\right)		=&\, \binom{ n }{ i } \,\orb{i}(x) \quad(1\leq i\leq n-2) \\
\tbox{and}
\sigma_{n-1}\left(\frac{y_1(x)+y_1(x)^{-1}}{2},\ldots,\frac{y_n(x)+y_n(x)^{-1}}{2}\right)	=&\, n \, \bigorb{e_{n-1}+e_{n}}(x) , \\
\sigma_{n}\left(\frac{y_1(x)+y_1(x)^{-1}}{2},\ldots,\frac{y_n(x)+y_n(x)^{-1}}{2}\right)		=&\, \frac{\bigorb{2\,e_{n-1}}(x)+\bigorb{2\,e_{n}}(x)}{2}.
\end{align*}
\end{proposition}
\begin{proof}
Let $1\leq i \leq n-2$. 
It follows from \Cref{eq_monomialsD} that $x_i=y_1(x) \ldots y_i(x)$. 
Then the statement for $\orb{i}$ is proven analogously to \Cref{OrbitSumsCn}. 

With $x_n \, x_{n-1} = y_1(x) \ldots y_{n-1}(x)$, we obtain the equation for $\bigorb{e_{n-1}+e_{n}}(x)$ as well.

Finally, we have $x_{n-1}^2=y_1(x)\ldots y_{n-1}(x)\,y_n(x)^{-1}$ and $x_{n}^2=y_1(x)\ldots y_n(x)$. 
Thus, 
\[
	\frac{\bigorb{2\,e_{n-1}}(x)+\bigorb{2\,e_{n}}(x)}{2}
=	\frac{1}{2^n}\left(\sum_{\substack{\epsilon\in\{\pm 1\}^n\\ \epsilon_1\ldots\epsilon_n=-1}}	y(x)^\epsilon
+	\sum_{\substack{\epsilon\in\{\pm 1\}^n\\ \epsilon_1\ldots\epsilon_n=1}}		y(x)^\epsilon \right)
=	\frac{1}{2^n} \sum_{\epsilon\in\{\pm 1\}^n}		y(x)^\epsilon
=	\frac{1}{2^n} \prod_{i=1}^n	 y_i(x)+y_i(x)^{-1},
\]
where $y^\epsilon:=y_1^{\epsilon_1}\ldots y_n^{\epsilon_n}$. This proves the last equation.
\end{proof}

We give the explicit expression for the right hand side of \Cref{OrbitSumsDn} in terms of the fundamental invariants.

\begin{lemma}\label{OrbitSumsDn2}
\begin{enumerate}
\item If $n$ is even, then
\begin{align*}
\bigorb{e_{n-1}+e_{n}}	&=	\,	\dfrac{2^{n-1}}{n}\,\orb{n-1}\,\orb{n} -\dfrac{1}{n}\,\sum_{j=1}^{(n-2)/2}\binom{ n }{ 2j-1}\,\orb{2j-1},\\
\bigorb{2\,e_{n-1}}	&=	\,	2^{n-1}\,\orb{n-1}^2-\sum_{j=1}^{(n-2)/2} \binom{ n }{ 2j } \, \orb{2j} -1,\\
\bigorb{2\,e_{n}}	&=	\,	2^{n-1}\,\orb{n}^2-\sum_{j=1}^{(n-2)/2} \binom{ n }{ 2j } \,\orb{2j} -1.
\end{align*}
\item If $n$ is odd, then
\begin{align*}
\bigorb{e_{n-1}+e_{n}}	&=	\,	\dfrac{2^{n-1}}{n}\,\orb{n-1}\,\orb{n} -\dfrac{1}{n}\,\sum_{j=1}^{(n-3)/2}\binom{ n }{ 2j }\,\orb{2j} -\dfrac{1}{n},\\
\bigorb{2\,e_{n-1}}	&=	\,	2^{n-1}\,\orb{n-1}^2-\sum_{j=0}^{(n-3)/2} \binom{ n }{ 2j+1 } \,\orb{2j+1},\\
\bigorb{2\,e_{n}}	&=	\,	2^{n-1}\,\orb{n}^2-\sum_{j=0}^{(n-3)/2} \binom{ n }{ 2j+1 } \,\orb{2j+1}.
\end{align*}
\end{enumerate}
\end{lemma}
\begin{proof}
With \Cref{equation_WeightsRootsD} and the recurrence formula from \Cref{proposition_RecurrenceOrbitPolynomials}, the proof is analogous to that for the case of $\RootB$ in \Cref{OrbitSumsBn2}. The full computation can be found in \cite[Lemma 1.43]{TobiasThesis}
\end{proof}

\begin{lemma}\label{proposition_PsiD}
The map
\[
\Upsilon:
\begin{array}[t]{ccl}
	\TTorus^{n}	&	\to 	& \TTorus^n,\\
	x 			&	\mapsto	& 	( y_1(x),\ldots,y_{n}(x) ),
\end{array}
\]
is surjective. 
Furthermore, every $y\in\TTorus^n$ has exactly two distinct preimages $x,x'\in\TTorus^n$ under $\Upsilon$ with
\[
\orb{i}(x)=
\begin{cases}
 \orb{i}(x'),&\mbox{if}\quad 1\leq i \leq n-2\\
-\orb{i}(x'),&\mbox{if}\quad i\in\{n-1,n\}
\end{cases}.
\]
\end{lemma}
\begin{proof}
For $y\in\TTorus^n$, choose $x\in\TTorus^n$ with $ x_1= y_1, x_2= y_1\, x_1,\ldots, x_{n-2}= y_{n-2}\, x_{n-3}$ and $ x_{n-1}^2 = y_n^{-1}\, y_{n-1}\, x_{n-2}$, $x_n= y_n\, x_{n-1}$. 
Then $x$ is a preimage of $y$ under $\Upsilon$ and unique up to a common sign in $ x_{n-1}$ and $x_n$. 
This second preimage is denoted by $x'$. 
By \Cref{OrbitSumsDn}, we have $\orb{i}(x) = \orb{i}(x')$ whenever $1\leq i\leq n-2$. 

It follows from \Cref{equation_WeightsRootsD} that the stabilizers of $\fweight{n-1}$ and $\fweight{n}$ in $\weyl$ are both $\mathfrak{S}_n$. 
Hence, we have $\weyl\, \fweight{n-1} = \{\pm 1\}_+^{n}\,\fweight{n-1}$ and $\weyl\, \fweight{n} = \{\pm 1\}_+^{n}\,\fweight{n}$. 
Thus, for $\weight{}\in \weyl\, \fweight{n}$, there exist $\epsilon_i = \pm 1$ with $\epsilon_1\ldots\epsilon_n=1$ and $\nu \in \Z \, \fweight{1} \oplus \ldots \oplus \Z \, \fweight{n-2}$, such that
\[
\begin{array}{rl}
\weight{}	=&	\dfrac{\epsilon_1}{2}\,e_1+\ldots+\dfrac{\epsilon_n}{2}\,e_n =	\dfrac{\epsilon_1}{2}\,\fweight{1}+\sum\limits_{i=2}^{n-2} \dfrac{\epsilon_i}{2}\, (\fweight{i} - \fweight{i-1}) + \dfrac{\epsilon_{n-1}}{2}\,(\fweight{n}+\fweight{n-1}-\fweight{n-2}) + \dfrac{\epsilon_{n}}{2}\,(\fweight{n}-\fweight{n-1})\\
			=&	\dfrac{\epsilon_{n-1}+\epsilon_{n}}{2}\,\fweight{n} + \dfrac{\epsilon_{n-1}-\epsilon_{n}}{2}\,\fweight{n-1} + \nu \in \Weights.
\end{array}
\]
Now let $\alpha$, $\beta\in\Z^n$, such that $\weight{}=\alpha_1\,\fweight{1}+\ldots+\alpha_n\,\fweight{n}$ and $\nu=\beta_1\,\fweight{1}+\ldots+\beta_n\,\fweight{n}$. 
Then $\beta_{n-1}=\beta_n=0$ and the monomial in $\orb{n}$ corresponding to $\weight{}$ is 
\[
	x^{\alpha}	=	x_1^{\beta_1} \ldots x_{n-2}^{\beta_{n-2}}\,x_{n-1}^{(\epsilon_{n-1}-\epsilon_{n})/2} \, x_n^{(\epsilon_{n-1}+\epsilon_{n})/2}
\]
with $(\epsilon_{n-1} \pm \epsilon_{n}) / 2 \in \{-1,0,1\}$. Therefore, $x^{\alpha}$ is linear in $x_{n-1}$ and independent of $x_n$ or vice versa. 
With $x, x'$ as above we have $x^\alpha=-(x')^\alpha$. 
Since every monomial in $\orb{n}$ can be written in terms of such $\beta$ and $\epsilon_i$, we obtain $\orb{n}(x)=-\orb{n}(x')$. 
Analogously, for $\weight{}\in \weyl\, \fweight{n-1}$, we have $\epsilon_1\ldots\epsilon_n=-1$ and obtain the statement for $\orb{n-1}$.
\end{proof}

\subsection{Hermite characterization with standard monomials}

We now characterize, whether a given point $z$ is contained in the $\TTorus$--orbit space $\Image$ of $\gva$, that is, 
if the equation $\orb{i}(x)=z_i$ has a solution $x\in\TTorus^n$. 
As shown in previous subsections, this is equivalent to deciding whether a symmetric polynomial system of type $\mathrm{(II)}$ has its solution in $\TTorus^n$. 
We state our main result for $\RootD$ in the standard monomial basis.

\begin{theorem}\label{HermiteCharacterizationDn}
Define the $n$--dimensional $\R$--vector space
\[
	\mathcal{Z}:=
	\begin{cases}
		\R^n,&\tbox{if} n \tbox{is even}\\
		\{ z\in\C^n\,\vert\, z_1,\ldots,z_{n-2}\in\R, \overline{z_n} = z_{n-1} \},&\tbox{if} n \tbox{is odd}
	\end{cases}
\]
and the matrix $H\in\RX^{n\times n}$ by
\begin{align*}
H(z)_{ij}	&=	 \trace(C(z)^{i+j-2}-C(z)^{i+j}),
\tbox{where}
C(z)		=
\begin{bmatrix}
0		&	\cdots	&	0		&	-c_n(z)		\\
1		&			&	0		&	-c_{n-1}(z)	\\
		&	\ddots	&			&	\vdots		\\
0		&			&	1		&	-c_1(z)
\end{bmatrix},\\
c_i(z)		&=	(-1)^i\,\binom{n}{i}\, z_i \quad (1\leq i\leq n-2) \tbox{and} \\
c_{n-1}(z)	&=	(-1)^{n-1}\,
\begin{cases}
2^{n-1}\,z_n\,z_{n-1}\,-\sum\limits_{j=1}^{(n-2)/2} \binom{n}{2j-1}\,z_{2j-1},&\tbox{if} n \tbox{is even}\\
2^{n-1}\,z_n\,z_{n-1}\,-\sum\limits_{j=1}^{(n-3)/2} \binom{n}{2j}\,z_{2j}-1,&\tbox{if} n \tbox{is odd}
\end{cases} , \\
~\\
c_{n}(z)	&= (-1)^n \, 
\begin{cases}
2^{n-2}\,(z_n^2+z_{n-1}^2)\,-\sum\limits_{j=1}^{(n-2)/2} \binom{n}{2j}\,z_{2j}-1,&\tbox{if} n \tbox{is even}\\
2^{n-2}\,(z_n^2+z_{n-1}^2)\,-\sum\limits_{j=0}^{(n-3)/2} \binom{n}{2j+1}\,z_{2j+1},&\tbox{if} n \tbox{is odd}
\end{cases} .
\end{align*}
For all $z\in \mathcal{Z}$, we have $H(z)\in\R^{n\times n}$ and $\Image=\{ z\in \mathcal{Z} \,\vert\, H(z)\succeq 0 \}$.
\end{theorem}
\begin{proof}
Let $z\in\C^n$ and set $c_i:=c_i(z)$ for $1\leq i \leq n$.

To show ``$\subseteq$'', assume that $z\in\Image$. Then there exists $ x\in\TTorus^{n}$, such that $\orb{i}( x)=z_i$ for $1\leq i \leq n$. Furthermore, we have $z\in \mathcal{Z}$ and $c_i\in\R$. By \Cref{OrbitSumsDn} and \Cref{OrbitSumsDn2,proposition_PsiD}, the solution of the symmetric polynomial system
\[
\mathrm{(II)}  \quad  \sigma_i\left(\frac{y_1+y_1^{-1}}{2}	,\ldots,\frac{y_n+y_n^{-1}}{2}	\right) = (-1)^i\,\tilde{c}_i \quad (1\leq i\leq n)
\]
is $y=\Upsilon( x)\in\TTorus^n$. Applying \Cref{Corollary_SolutionsSystemII} yields $H(z)\succeq 0$.

For ``$\supseteq$'' on the other hand, assume $z\in \mathcal{Z}$ with $H(z)\succeq 0$. 
Hence, $c_i\in\R$ and by \Cref{Corollary_SolutionsSystemII}, the solution $y$ of the above system $(\mathrm{II})$ is contained in $\TTorus^n$. 
According to \Cref{proposition_PsiD}, $y$ has two distinct preimages $ x, x'\in\TTorus^n$. 
We have $z_i=\orb{i}( x)=\orb{i}( x')$ for $1\leq i \leq n-2$ and $\orb{n-1}( x)=-\orb{n-1}( x'),\orb{n}( x)=-\orb{n}( x')$. 
Furthermore, $z_{n-1}^2+z_n^2=\orb{n-1}( x)^2+\orb{n}( x)^2=\orb{n-1}( x')^2+\orb{n}( x')^2$ and $z_{n-1}\,z_n=\orb{n-1}( x)\,\orb{n}( x)=\orb{n-1}( x')\,\orb{n}( x')$. 
Therefore, $\{z_{n-1},z_n\}\in\{\{\orb{n-1}( x),\orb{n}( x)\},\{\orb{n-1}( x'),\orb{n}( x')\}\}$. 
If $z_{n-1}=\orb{n-1}( x),z_n=\orb{n}( x)$, then $z=\bigcos(x)$. 
Otherwise, set $\tilde{x} = (x_1,\ldots,x_{n-2},x_n,x_{n-1})$ to obtain $z_{n-1}=\orb{n-1}( \tilde{x}),z_n=\orb{n}( \tilde{x})$ and $z=\bigcos(\tilde{x})$. 
An analogous argument applies to $x'$ and thus, $z$ is contained in $\Image$.
\end{proof}

\section{Type \RootG[2]}
\label{section_proofG}
\setcounter{equation}{0}

The group $\mathfrak{S}_{3} \ltimes \{\pm 1\}$ acts on $\R^3$ by permutation of coordinates and scalar multiplication with $\pm 1$ and leaves the subspace $V = \R^3/\langle[1,1,1]^t\rangle =\{u\in\R^{n}\,\vert\,u_1+u_2 +u_{3}=0\}$ invariant. The root system $\RootG[2]$ from \cite[Planche IX]{bourbaki456} is a root system in $V$ with fundamental weights
\begin{equation}\label{equation_WeightsRootsG}
	\fweight{1}=(1,-1,0) \tbox{and}
	\fweight{2}=(-2,1,1).
\end{equation}
The Weyl group of $\RootG[2]$ is $\weyl\cong\mathfrak{S}_{3}\ltimes \{\pm 1\}$. We have $-I_3\in\weyl$ and thus, $-\fweight{1}\in\weyl\,\fweight{1}$ as well as $-\fweight{2}\in\weyl\,\fweight{2}$. 
Furthermore, we have$\nops{\weyl\,\fweight{1}} = \nops{\weyl\,\fweight{2}} = 6$, see \Cref{example_rootsystemG2}.

In this section, we give a closed formula for the matrix polynomial from \Cref{MainThmIntro} in the standard monomial basis for $\RootG[2]$. This is a root system of rank $2$. Hence, the ring of Laurent polynomials is $\Rx=\Q[x_1,x_1^{-1},x_{2},x_{2}^{-1}]$ and the polynomial ring is $\RX=\Q[z_1,z_2]$.

\begin{figure}[H]
\begin{center}
	\begin{overpic}[width=0.32\textwidth,grid=false,tics=10]{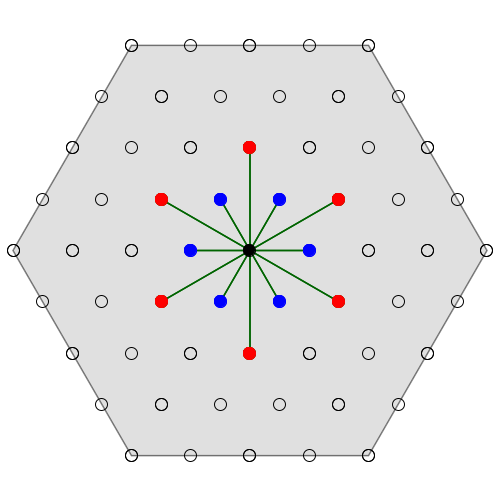}
	\put (63,52) {\large \textcolor{blue}{$\displaystyle \fweight{1}$}}
	\put (63,65) {\large \textcolor{red}{$\displaystyle \fweight{2}$}}
	\end{overpic}
	\quad
	\begin{overpic}[width=0.58\textwidth,grid=false,tics=10]{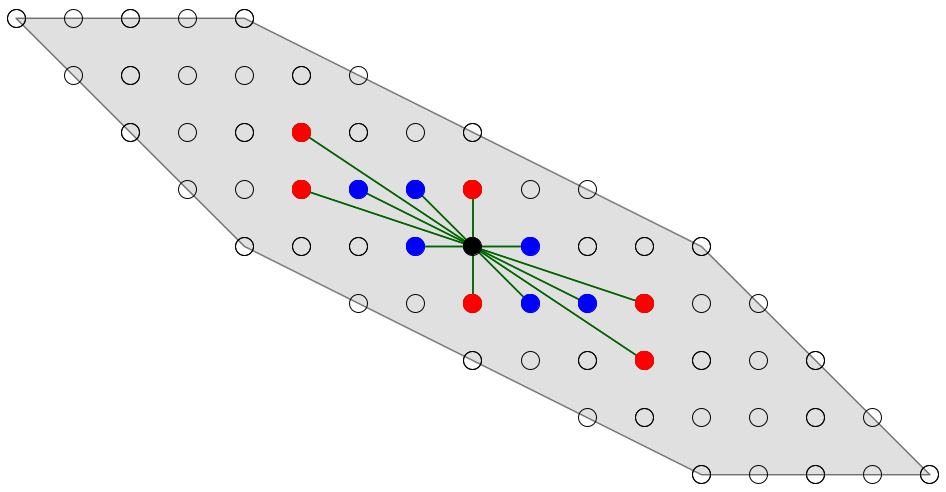}
	\put (57,27) {\large \textcolor{blue}{$\displaystyle e_{1}$}}
	\put (51,33) {\large \textcolor{red}{$\displaystyle e_{2}$}}
	\end{overpic}
\end{center}
\caption{The \textcolor{OliveGreen}{root system} $\RootG[2]$ and the weight lattice in the usual orthogonal representation and the integer representation. The orbits of the fundamental weights are the \textcolor{blue}{blue} and \textcolor{red}{red} lattice elements.}
\label{example_rootsystemG2}
\end{figure}

\subsection{Orbit polynomials}

The integer representation of $\weyl$ with respect to the fundamental weights in \Cref{equation_WeightsRootsG} is
\[
	\gva
=	\left\{
	\pm \begin{bmatrix} 1 & 0 \\ 0 & 1 \end{bmatrix},
	\pm \begin{bmatrix} -1 & 0 \\ 1 & 1 \end{bmatrix},
	\pm \begin{bmatrix} 1 & 3 \\ 0 & -1 \end{bmatrix},
	\pm \begin{bmatrix} 1 & 3 \\ -1 & -2 \end{bmatrix},
	\pm \begin{bmatrix} 2 & 3 \\ -1 & -1 \end{bmatrix},
	\pm \begin{bmatrix} 2 & 3 \\ -1 & -2 \end{bmatrix}
	\right\} .
\]
Then the orbit $\gva \cdot x_1 = \{ x^{B e_1} \,\vert\, B\in \gva \} \subseteq \Rx$ consists of $6$ distinct monomials, namely
\begin{equation}\label{eq_monomialsG}
	y_1:=x_1,\quad y_2:=x_1\,x_2^{-1},\quad y_3 := x_1^{-2} \, x_2 
\end{equation}
and their inverses. For $1\leq i\leq 3$, let $\sigma_i$ be the $i$--th elementary symmetric function in $3$ indeterminates and recall that, for $i\leq 2$, $\orb{i}=\bigorb{e_i}$ is the $\gva$--invariant orbit polynomial associated to $e_i\in\Z^{2}$ from \Cref{defi_orbitpoly}.

\begin{proposition}\label{OrbitsG2}
In $\Qx$, we have
\begin{align*}
	\sigma_1\left(\frac{y_1(x)+y_1(x)^{-1}}{2},\frac{y_2(x)+y_2(x)^{-1}}{2},\frac{y_3(x)+y_3(x)^{-1}}{2}\right)
&=\,	3\, \orb{1}(x), \\
	\sigma_2\left(\frac{y_1(x)+y_1(x)^{-1}}{2},\frac{y_2(x)+y_2(x)^{-1}}{2},\frac{y_3(x)+y_3(x)^{-1}}{2}\right)
&=\,	\frac{3( \orb{1}(x) +  \orb{2}(x))}{2} \\
\tbox{and}
	\sigma_3\left(\frac{y_1(x)+y_1(x)^{-1}}{2},\frac{y_2(x)+y_2(x)^{-1}}{2},\frac{y_3(x)+y_3(x)^{-1}}{2}\right)
&=\,	\frac{9\, \orb{1}(x)^2 - 3 \, \orb{1}(x) - 3 \, \orb{2}(x) - 1}{2}.
\end{align*}
\end{proposition}
\begin{proof}
In principle, this can be checked by hand, but we give a constructive proof to show how to obtain the equations. 
It follows from \Cref{eq_monomialsG} that $y_1(x)\,y_2(x)\,y_3(x) = 1$ and $x_1=y_1(x)$, $x_1^{-1}=y_2(x)\,y_3(x)$ and $x_2=y_1(x)^2\,y_3(x)$, $x_2^{-1}=y_2(x)^2\,y_3(x)$. 
Thus, after computing the orbits $\gva\cdot x_1$ and $\gva\cdot x_2$, we find polynomials $g_1,g_2\in\R[y_1,y_2,y_3]$ , such that $\orb{i}(x) = g_i(y(x))$. 
Consider the ideal
\[
	\mathcal{I}
:=	\langle \orb{1} - g_1(y) , \orb{2} - g_2(y) , 1 - y_1\,y_2\,y_3 \rangle
\subseteq \R[y_1,y_2,y_3,\orb{1},\orb{2}].
\]
Choose an elimination ordering $\{y_1, y_2, y_3\} \succeq \{\orb{1}, \orb{2}\}$ and let $G$ be a Gr\"obner basis of $\mathcal{I}$. 
Note that the expression on the left--hand side of the statement is $\gva$--invariant. 
We can write it as $\sigma_i(y_1+y_2\,y_3,y_2+y_1\,y_3,y_3+y_1\,y_2)/2^i$, that is, a polynomial in $\R[y_1,y_2,y_3,\orb{1},\orb{2}]$. 
Because of the $\gva$--invariance, the normal form of this polynomial with respect to $G$ only depends on $\orb{1}$ and $\orb{2}$. 
One obtains the right--hand side. 
\end{proof}

The next statement follows immediately from \Cref{eq_monomialsG}. 

\begin{lemma}\label{proposition_PsiG}
The map
\[
\Upsilon: \begin{array}[t]{ccl}
\TTorus^{2}&\to & \TToruss^3,\\
x &\mapsto & ( y_1(x), y_2(x), y_3(x) ) 
\end{array}
\]
is bijective.
\end{lemma}

\subsection{Hermite characterization with standard monomials}

With \Cref{OrbitsG2} and \Cref{proposition_PsiG}, the proof of the following is analogous to \Cref{HermiteCharacterizationCn}.

\begin{theorem}\label{HermiteCharacterizationGn}
Define the matrix $H\in\RX^{3\times 3}$ by
\begin{align*}
H(z)_{ij}	=	 \trace((C(z))^{i+j-2}-(C(z))^{i+j}),
\tbox{where}
C(z)		=
\begin{bmatrix}
0	&	0	&	(9\, z_1^2 - 3 \, z_1 - 3 \, z_2 -1)/2 \\
1	&	0	&	-3(z_1 + z_2)/2 \\
0	&	1	&	3\, z_1
\end{bmatrix}.
\end{align*}
Then $\Image=\Image_\R=\{ z\in \R^2 \,\vert\, H(z)\succeq 0 \}$.
\end{theorem}

\subsection*{Example}

Our result applies to the root system from \Cref{example_rootsystemG2} as follows. Let $z=(z_1,z_2) \in\R^2$. The Hermite matrix polynomial from \Cref{HermiteCharacterizationGn} is $4/3\,H(z)=$
\begin{tiny}
\begin{gather*}
	\begin{bmatrix}
	-12\,z_1^2 + 4\,z_1 + 4\,z_2 + 4 														& \cdots \\
	-36\,z_1^3 + 18\,z_1\,z_2 + 10\,z_1 + 6\,z_2 + 2										& \cdots \\
	-108\,z_1^4 + 72\,z_1^2\,z_2 + 30\,z_1^2 + 12\,z_1\,z_2 - 6\,z_2^2 + 4\,z_1 - 4\,z_2	& \cdots
	\end{bmatrix}\\
	\begin{bmatrix}
	\cdots & -36\,z_1^3 + 18\,z_1\,z_2 + 10\,z_1 + 6\,z_2 + 2																									& \cdots \\
	\cdots & -108\,z_1^4 + 72\,z_1^2\,z_2 + 30\,z_1^2 + 12\,z_1\,z_2 - 6\,z_2^2 + 4\,z_1 - 4\,z_2																& \cdots \\
	\cdots & -324\,z_1^5 + 270\,z_1^3\,z_2 + 126\,z_1^3 + 45\,z_1^2\,z_2 - 45\,z_1\,z_2^2 + 15\,z_1^2 - 48\,z_1\,z_2 - 15\,z_2^2 - 11\,z_1 - 11\,z_2 - 2	& \cdots 
	\end{bmatrix}\\
	\begin{bmatrix}
	\cdots & -108\,z_1^4 + 72\,z_1^2\,z_2 + 30\,z_1^2 + 12\,z_1\,z_2 - 6\,z_2^2 + 4\,z_1 - 4\,z_2 \\ 
	\cdots & -324\,z_1^5 + 270\,z_1^3\,z_2 + 126\,z_1^3 + 45\,z_1^2\,z_2 - 45\,z_1\,z_2^2 + 15\,z_1^2 - 48\,z_1\,z_2 - 15\,z_2^2 - 11\,z_1 - 11\,z_2 - 2 \\ 
	\cdots & -972\,z_1^6 + 972\,z_1^4\,z_2 + 432\,z_1^4 + 162\,z_1^3\,z_2 - 243\,z_1^2\,z_2^2 + 63\,z_1^3 - 207\,z_1^2\,z_2 - 81\,z_1\,z_2^2 + 9\,z_2^3 - 45\,z_1^2 - 66\,z_1\,z_2 - 3\,z_2^2 - 14\,z_1 - 6\,z_2 - 1
	\end{bmatrix}
\end{gather*}
\end{tiny}
and shall have characteristic polynomial
\[
	\det (x\,I_3 - H(z)) = x^3-\textcolor{OliveGreen}{h_1(z)}\,x^2+\textcolor{blue}{h_2(z)}\,x-\textcolor{red}{h_3(z)}.
\]
Again, we have $z\in \Image$ if and only if $h_i(z)\geq 0$ for $1\leq i\leq 3$ with
\[
\begin{array}{rcl}
	\textcolor{red}{h_3(z)}&\textcolor{red}{=}					&	\textcolor{red}{-\coeff(x^0,\det(x\,I_3-H(z)))} \\
	&\textcolor{red}{=}										&	\textcolor{red}{81/16\,(24\,z_1^3 - 12\,z_1\,z_2 - z_2^2 - 6\,z_1 - 4\,z_2 - 1)\,(3\,z_1^2 - 2\,z_2 - 1)^2\,(3\,z_1 + 1)^2}, \\
	\textcolor{blue}{h_2(z)}&\textcolor{blue}{=}				&	\phantom{-} \textcolor{blue}{\coeff(x^1,\det(x\,I_3-H(z)))} \\
	&\textcolor{blue}{=}									&	\textcolor{blue}{27/16\,(3\,z_1^2 - 2\,z_2 - 1)\,(648\,z_1^6 - 540\,z_1^5 - 972\,z_1^4\,z_2 - 216\,z_1^3\,z_2^2 - 648\,z_1^4 + 36\,z_1^3\,z_2}\\
	&														&	\textcolor{blue}{ + 369\,z_1^2\,z_2^2 + 126\,z_1\,z_2^3 + 9\,z_2^4 - 18\,z_1^3 + 480\,z_1^2\,z_2 + 246\,z_1\,z_2^2 + 36\,z_2^3 + 129\,z_1^2 + 202\,z_1\,z_2} \\
	&														&	\textcolor{blue}{ + 37\,z_2^2 + 44\,z_1 + 28\,z_2 + 4)}, \\
	\textcolor{OliveGreen}{h_1(z)}&\textcolor{OliveGreen}{=}	&	\textcolor{OliveGreen}{-\coeff(x^2,\det(x\,I_3-H(z)))} \\
	&\textcolor{OliveGreen}{=}								&	\textcolor{OliveGreen}{9/4 - 729\,z_1^6 + 243\,z_1^4 + 729\,z_1^4\,z_2 - 729/4\,z_1^2\,z_2^2 + 243/2\,z_1^3\,z_2 + 189/4\,z_1^3 - 405/4\,z_1^2\,z_2}\\
	&														&	\textcolor{OliveGreen}{ - 243/4\,z_1\,z_2^2 + 27/4\,z_2^3 - 81/4\,z_1^2 - 81/2\,z_1\,z_2 - 27/4\,z_2^2 - 9/2\,z_1 - 9/2\,z_2}.
\end{array}
\]
\begin{figure}[H]
	\begin{center}
		\begin{subfigure}{.4\textwidth}
			\centering
			\includegraphics[width=6cm]{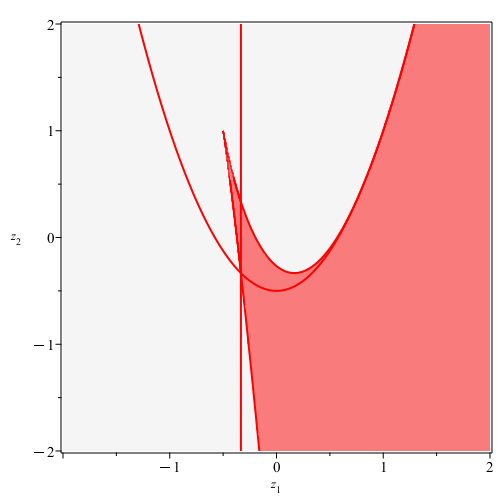}
			\caption{\textcolor{red}{$h_3(z)\geq 0$}}
		\end{subfigure}%
		\begin{subfigure}{.4\textwidth}
			\centering
			\includegraphics[width=6cm]{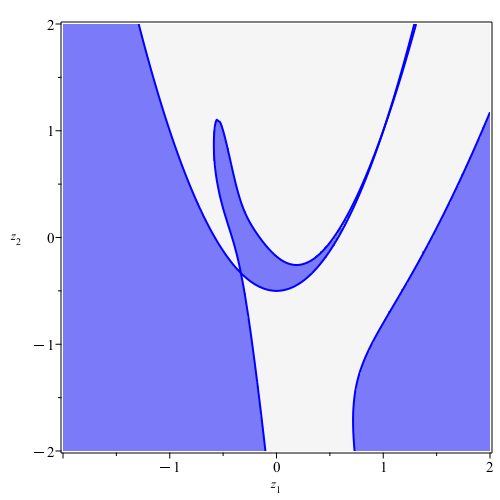}
			\caption{\textcolor{blue}{$h_2(z)\geq 0$}}
		\end{subfigure}
		\begin{subfigure}{.4\textwidth}
			\centering
			\includegraphics[width=6cm]{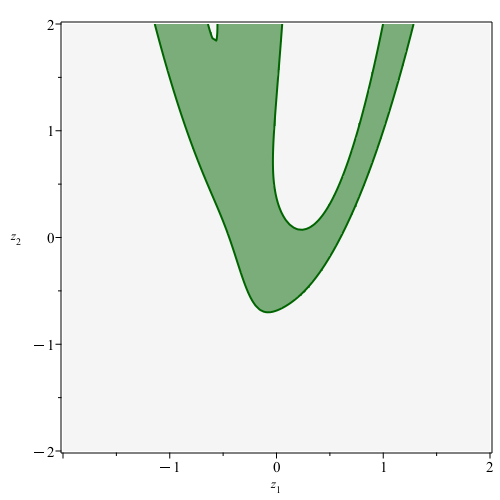}
			\caption{\textcolor{OliveGreen}{$h_1(z)\geq 0$}}
		\end{subfigure}
		\begin{subfigure}{.4\textwidth}
			\centering
			\includegraphics[width=6cm]{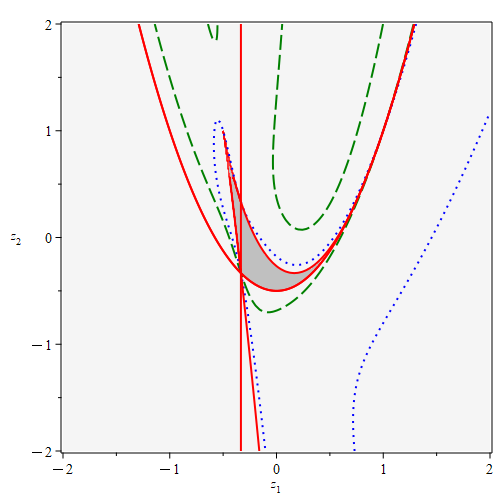}
			\caption{\textcolor{gray}{$h_1(z),h_2(z),h_3(z)\geq 0$}}
		\end{subfigure}
		\caption{The semi--algebraic sets defined by the coefficients of the characteristic polynomial of $H(z)$.}
		\label{example_figureG2}
	\end{center}
\end{figure}
The vertices are
\begin{align*}
&	\mathrm{Vertex}_1
:=	\bigcos(\mexp{\fweight{1},\fweight{1}/3},\mexp{\fweight{2},\fweight{1}/3})
=	\bigcos\left(-\frac{1}{2}+\mathrm{i}\frac{\sqrt{3}}{2},1\right)
=	\left(-\frac{1}{2},1\right), \\
&	\mathrm{Vertex}_2
:=	\bigcos(\mexp{\fweight{1},\fweight{2}/6},\mexp{\fweight{2},\fweight{2}/6})
=	\bigcos(-1,1)
=	\left(-\frac{1}{3},-\frac{1}{3}\right), \\
&	\mathrm{Vertex}_3
:=	\bigcos(\mexp{\fweight{1},0},\mexp{\fweight{2},0})
=	\bigcos(1,1)
=	(1,1).
\phantom{\dfrac{\sqrt{1}}{3}}
\end{align*}

\section{Hermite characterization of $\Image_\R$ with Chebyshev polynomials}
\label{section_proofMain}
\setcounter{equation}{0}

We shall finally prove \Cref{MainThmIntro} from the introduction, which characterizes the real $\TTorus$--orbit space $\Image_\R$ in terms of the real generalized Chebyshev polynomials $\TT_\alpha$ from \Cref{eq_RealPartCheby}. 

Let $n\in\N$ and $\Roots$ be a root system of type $\RootA$, $\RootB$, $\RootC$ ($n\geq 2$), $\RootD$ ($n\geq 4$) or $\RootG[2]$ as in \cite[Planche I -- IV, IX]{bourbaki456}. 
In particular, the rank $r:=\rank(\Roots)$ of $\Roots$ is $r=n-1$ for $\RootA$, $r=n-1=2$ for $\RootG[2]$ and $r=n$ otherwise. 
What all of these root systems have in common is that the symmetric group $\mathfrak{S}_n$ is a semi--direct factor of the Weyl group of $\Roots$, a fact that made the characterization via Hermite quadratic forms in the previous \Cref{section_proofA,section_proofC,section_proofB,section_proofD,section_proofG} possible. 

The Weyl group has an integer representation $\gva\subseteq\mathrm{GL}_{r}(\Z)$ as in \Cref{eq_integerrepresentation} with respect to the fundamental weights in \Cref{equation_WeightsRootsA,equation_WeightsRootsC,equation_WeightsRootsB,equation_WeightsRootsD,equation_WeightsRootsG}, respectively. 
We denote by $\RX=\Q[z_1,\ldots,z_{r}]$ the multivariate polynomial ring.
In particular, the size of the corresponding Hermite matrix is $n$ (determined by the symmetric group $\mathfrak{S}_n$), while the number of coordinates of the $\TTorus$--orbit space is $r$ (determined by the rank of $\Roots$).

\begin{theorem}\label{MainThm}
Let $H \in \RX^{n\times n}$ be the matrix polynomial with entries
\begin{align*}
	2^{i+j} \, H _{ij}
=&	- \TT_{(i+j)\,e_1} + \sum\limits_{\ell=1}^{\lceil (i+j)/2 \rceil -1} \left( 4 \binom{i+j-2}{\ell-1} - \binom{i+j}{\ell} \right) \TT_{(i+j-2\,\ell)\,e_1}	\\
&+	\frac{1}{2} \begin{cases}
	4\binom{i+j-2}{(i+j)/2-1}-\binom{i+j}{(i+j)/2}	,&	\tbox{if} i+j \tbox{is even}	\\
	0												,&	\tbox{if} i+j \tbox{is odd}
	\end{cases}.
\end{align*}
Then the real $\TTorus$--orbit space of $\gva$ is $\Image_\R = \{z\in\R^{r}\,\vert\,H(z)\succeq 0\}$.
\end{theorem}

We need two lemmata involving the orbit polynomials $\bigorb{\alpha}=\frac{1}{\nops{\gva}}\sum\limits_{B\in\gva} x^{B\alpha}$ for $\alpha\in\Z^r$ from \Cref{defi_orbitpoly}.

\begin{lemma}\label{lemma_OrbitMain}
In $\Rx$, define the $2\,n$ distinct monomials $\{y_1^{\pm 1},\ldots,y_n^{\pm 1}\}:=\gva\cdot x_1 \cup \gva\cdot x_1^{-1}$. Then, for $x\in \TTorus^{r}$ and $k\in\N$, we have $(y_i(x)^k + y_i(x)^{-k})/2\in[-1,1]$ and
\[
\frac{1}{n} \sum\limits_{i=1}^n y_i(x)^k + y_i(x)^{-k}
=	\bigorb{k\,e_{1}} (x) + \bigorb{-k\,e_{1}} (x) .
\]
\end{lemma}
\begin{proof}
This follows immediately from \Cref{OrbitSumsAn,OrbitSumsCn,OrbitSumsBn,OrbitSumsDn,OrbitsG2}.
\end{proof}

\begin{lemma}\label{lemma_OrbitMain2}
For $x\in\TTorus^{r}$ and $y_1,\ldots,y_n$ the monomials from \Cref{lemma_OrbitMain}, let $C\in\R^{n\times n}$ be a matrix with eigenvalues $(y_i(x)+y_i(x)^{-1})/2$. Then, for $k\in \N$, the trace of the $k$--th power of $C$ is
\[
\trace(C^k)
=	\dfrac{n}{2^k}	\sum\limits_{\ell=0}^{\lceil k/2 \rceil -1} 
\left( \binom{k}{\ell}\,(\bigorb{(k-2\,\ell)\,e_{1}} (x) + \bigorb{(2\,\ell-k)\,e_{1}} (x)) \right)
+	\dfrac{n}{2^k} 
\begin{cases}
	\binom{k}{k/2}	,&	k \tbox{even}	\\
	0				,&	k \tbox{odd}
\end{cases}.
\]
\end{lemma}
\begin{proof}
The trace of the $k$--th power of $C$ is the sum of the $k$--th powers of the eigenvalues. Thus,
\begin{align*}
	\trace(C^k)
	=&	\sum\limits_{i=1}^n	\left(\dfrac{y_i(x)+y_i(x)^{-1}}{2}\right)^k
	=	\,\dfrac{1}{2^k}\sum\limits_{i=1}^n	\left( \sum\limits_{\ell=0}^k \binom{k}{\ell}\,y_i(x)^{k-\ell}\,y_i(x)^{-\ell} \right)	\\
	=&	\,\dfrac{1}{2^k}\sum\limits_{i=1}^n	\left( \sum\limits_{\ell=0}^{\lceil k/2 \rceil -1} \binom{k}{\ell}\,y_i(x)^{k-2\,\ell}
	+	\sum\limits_{\ell=\lfloor k/2 \rfloor + 1}^k \binom{k}{\ell}\,y_i(x)^{k-2\,\ell}
	+	\begin{cases}
		\binom{k}{k/2}	,&	k \tbox{even}	\\
		0				,&	k \tbox{odd}
	\end{cases} \right)	\\
	=&	\,\dfrac{1}{2^k}\sum\limits_{\ell=0}^{\lceil k/2 \rceil -1} \left( \binom{k}{\ell}\,\sum\limits_{i=1}^n	\left( y_i(x)^{k-2\,\ell}+y_i(x)^{2\,\ell-k} \right) \right)
	+	\dfrac{n}{2^k} \begin{cases}
		\binom{k}{k/2}	,&	k \tbox{even}	\\
		0				,&	k \tbox{odd}
	\end{cases}	.
\end{align*}
Applying \Cref{lemma_OrbitMain} yields the formula.
\end{proof}

\begin{proof}[Proof of \emph{\Cref{MainThm}}]
The $\TTorus$--orbit space of $\gva$ is by definition $\Image = \{\bigcos(x)\,\vert\,x\in\TTorus^{r}\}$. According to \Cref{HermiteCharacterizationAn,HermiteCharacterizationCn,HermiteCharacterizationBn,HermiteCharacterizationDn,HermiteCharacterizationGn}, we have $\Image = \{\tilde{z}\in \mathcal{Z} \,\vert\, H(\tilde{z})\succeq 0\}$, where $\mathcal{Z}$ is an $\R$--vector space with $\dim(\mathcal{Z})=r$ and $H\in\RX^{n\times n}$ has entries $H_{ij} = \trace(C^{i+j-2})-\trace(C^{i+j})$ for a certain matrix $C\in\RX^{n\times n}$. 

For $x\in\TTorus^{r}$ and $\tilde{z} := \bigcos(x) \in \Image \subseteq \mathcal{Z}$, it is now essential to note that $C(\tilde{z})\in\R^{n\times n}$ has eigenvalues $(y_i(x)+y_i(x)^{-1})/2\in [-1,1]$ with $y_i$ as in \Cref{lemma_OrbitMain}. This follows from the construction, by which $C$ is the matrix of a multiplication map, see \Cref{subsec_SturmSylvester}. In particular, we obtain a formula for the entries $H(\tilde{z})_{ij}$ with \Cref{lemma_OrbitMain2}. In the basis of (real) generalized Chebyshev polynomials from \Cref{defi_ChebyPoly1} and \Cref{eq_RealPartCheby}, we have
\[
\bigorb{(k-2\,\ell)\,e_{1}} (x) + \bigorb{(2\,\ell-k)\,e_{1}} (x)
=	T_{(k-2\,\ell)\,e_{1}}(\bigcos(x)) + T_{(2\,\ell-k)\,e_{1}}(\bigcos(x))
=	2\,\TT_{(k-2\,\ell)\,e_{1}}(\bigcos_\R(x)).
\]
We now substitute $\tilde{z} \mapsto z:=\bigcos_\R(x) \in\Image_\R \subseteq \R^n$. For $1\leq i,j\leq n$ and $k:=i+j$, the entry $H(z)_{ij}$ is
\begin{align*}
 &	\, \trace(C(\tilde{z})^{k-2})-\trace(C(\tilde{z})^k)	\\
=&	\,\dfrac{2\,n}{2^{k}} \left( 4 \sum\limits_{\ell=0}^{\lceil k/2 \rceil -2} \left( \binom{k-2}{\ell}\,\TT_{(k-2\,(\ell+1))\,e_1}(z) \right) - \sum\limits_{\ell=0}^{\lceil k/2 \rceil -1} \left( \binom{k}{\ell}\,\TT_{(k-2\,\ell)\,e_1}(z) \right) \right)	\\
 &	+ \dfrac{n}{2^k}	\begin{cases}
						4\binom{k-2}{k/2-1}-\binom{k}{k/2}	,&	k \tbox{even}	\\
						0									,&	k \tbox{odd}
						\end{cases}		\\
=&	\,\dfrac{2\,n}{2^k} \left( 4 \sum\limits_{\ell=1}^{\lceil k/2 \rceil -1} \left( \binom{k-2}{\ell-1}\,\TT_{(k-2\,\ell)\,e_1}(z) \right) - \sum\limits_{\ell=0}^{\lceil k/2 \rceil -1} \left( \binom{k}{\ell}\,\TT_{(k-2\,\ell)\,e_1}(z) \right) \right)	\\
 &	+ \dfrac{n}{2^k} 	\begin{cases}
						4\binom{k-2}{k/2-1}-\binom{k}{k/2}	,&	k \tbox{even}	\\
						0									,&	k \tbox{odd}
						\end{cases}		\\
=&	\,\dfrac{2\,n}{2^k} \left( - \TT_{k\,e_1}(z) + \sum\limits_{\ell=1}^{\lceil k/2 \rceil -1} \left( 4 \binom{k-2}{\ell-1} - \binom{k}{\ell} \right) \TT_{(k-2\,\ell)\,e_1}(z) \right)	
	+ \dfrac{n}{2^k}	\begin{cases}
						4\binom{k-2}{k/2-1}-\binom{k}{k/2}	,&	k \tbox{even}	\\
						0									,&	k \tbox{odd}
						\end{cases}.
\end{align*}
Dividing by $2\,n$ does not change whether $H$ is positive semi--definite in $z$ and so we obtain the formula.
\end{proof}

\begin{remark}\label{remark_mainthm}
The matrix polynomial $H\in\RX^{n\times n}$ from \Cref{MainThm} follows the pattern
\[
\begin{bmatrix}
	\frac{\TT_{0}-\TT_{2\,e_{1}}}{4}& 
	\frac{\TT_{e_{1}} -\TT_{3\,e_{1}}}{8}& 
	\frac{\TT_{0}- \TT_{4\,e_{1}}}{16}&
	\frac{2\TT_{e_{1}}- \TT_{3\,e_{1}} - \TT_{5\,e_{1}}}{32}&
	\cdots\\
	
	\frac{\TT_{e_{1}} -\TT_{3\,e_{1}}}{8}& 
	\frac{\TT_{0}- \TT_{4\,e_{1}}}{16}&
	\frac{2 \TT_{e_{1}}- \TT_{3\,e_{1}} - \TT_{5\,e_{1}}}{32}&
	\frac{2 \TT_{0} +  \TT_{2\,e_{1}}-2 \TT_{4\,e_{1}} -  \TT_{6\,e_{1}}}{64}&
	\cdots\\
	
	\frac{\TT_{0}- \TT_{4\,e_{1}}}{16}& 
	\frac{2\TT_{e_{1}}- \TT_{3\,e_{1}} - \TT_{5\,e_{1}}}{32}& 
	\frac{2 \TT_{0} +  \TT_{2\,e_{1}}-2 \TT_{4\,e_{1}} -  \TT_{6\,e_{1}}}{64}&
	\frac{5 \TT_{e_1} - \TT_{3\,e_1} - 3 \TT_{5\,e_1} - \TT_{7\,e_1}}{128}&
	\cdots\\
	
	\frac{2\TT_{e_{1}}- \TT_{3\,e_{1}} - \TT_{5\,e_{1}}}{32}&
	\frac{2 \TT_{0} +  \TT_{2\,e_{1}}-2 \TT_{4\,e_{1}} -  \TT_{6\,e_{1}}}{64}&
	\frac{5 \TT_{e_1} - \TT_{3\,e_1} - 3 \TT_{5\,e_1} - \TT_{7\,e_1}}{128}&
	\frac{5 \TT_{0} + 4 \TT_{2\,e_1} - 4 \TT_{4\,e_1} - 4 \TT_{6\,e_1} - \TT_{8\,e_1}}{256}&
	\cdots\\
	
	\vdots & \vdots & \vdots & \vdots & \ddots
\end{bmatrix}.
\]
\begin{enumerate}
\item $H$ only depends on the $\TT_\alpha$, where $\alpha\in\N^r$ is a multiple of $e_{1}$. 
An entry of $H$ is a linear combination of at most $n+1$ such polynomials. 
The largest possible degree of an entry is $2\,n = \deg(\TT_{2\,n\,e_1})$. 
\item $H$ has Hankel structure, that is, $H_{ij}=H_{k\ell}$ whenever $i+j=k+\ell$. 
Therefore, the number of distinct entries is at most  $2\,n-1$. 
\item Denote by $H^{(n)}$ the matrix $H$ for fixed $n$. Then the leading principal submatrices of $H$ have the structure of $H^{(k)}$ for $k\leq n$.
\item If $\Roots$ is of type $\RootB$, $\RootC$ or $\RootG[2]$, then $\Image_\R=\Image$ and $\TT_{k\,e_i}(z)=T_{k\,e_i}(z)$ is the usual generalized Chebyshev polynomial of the first kind from \Cref{defi_ChebyPoly1}. The same holds for $\RootD$ when $n$ is even. When $n$ is odd on the other hand, then $T_{k\,e_i}(z) = \TT_{k\,e_i}(\tilde{z})$ with $\frac{z_{n-1}+z_n}{2}=\tilde{z}_{n-1}$ and $\frac{z_{n-1}-z_n}{2\mathrm{i}}=\tilde{z}_{n}$. Finally, for $\RootA$, we have $\frac{T_{k\,e_i}(z)+T_{-k\,e_i}(z)}{2} = \frac{T_{k\,e_i}(z)+T_{k\,e_{n-i}}(z)}{2} = \TT_{k\,e_i}(\tilde{z})$ with $\frac{z_{n-i}+z_i}{2}=\tilde{z}_{i}$ and $\frac{z_{n-i}-z_i}{2\mathrm{i}}=\tilde{z}_{n-i}$.
\end{enumerate} 
\end{remark}




\section*{Conclusion and future work}

We have characterized the $\TTorus$--orbit space of a Weyl group acting nonlinearly on the compact torus as a compact basic semi--algebraic set. 
The characterization is given through a polynomial matrix inequality and holds for the Weyl groups of type $\RootA$, $\RootB$, $\RootC$, $\RootD$ and $\RootG[2]$. 
It is the first explicit characterization for nonlinear actions in the sense that we have given a closed formula for the matrix. 
By expressing the entries in the basis of generalized Chebyshev polynomials, we have shown that these matrices have a common structure. 

For applications in graph theory \cite{BdCOV} and optimization \cite{chromatic22}, it is desirable to give such a closed formula also for the remaining cases $\RootE[6,7,8]$ and $\RootF[4]$. 
We will address those in a future work.




\section*{Acknowledgments}

The authors wish to acknowledge the insights that Claudio Procesi (Rome) shared and are grateful for discussions with Philippe Moustrou (Toulouse) and Robin Schabert (Troms\o).


The majority of the work of Tobias Metzlaff was carried out during his doctoral studies \cite{TobiasThesis} at Inria d'Universit\'{e} C\^{o}te d'Azur, supported by European Union’s Horizon 2020 research and innovation programme under the Marie Sk\l odowska-Curie Actions, grant agreement 813211 (POEMA).
Minor changes were applied during his postdoctoral research at RPTU Kaiserslautern--Landau, supported by the Deutsche Forschungsgemeinschaft transregional collaborative research centre (SFB--TRR) 195 ``Symbolic Tools in Mathematics and their Application''.




\bibliographystyle{alpha}
{\bibliography{mybib.bib}}

\end{document}